\documentclass[a4paper]{article}
\usepackage{lineno,hyperref}

\usepackage{graphicx}
\usepackage{epstopdf}
\usepackage{amssymb}
\usepackage{amsmath}
\usepackage{amsfonts}
\usepackage{amsthm} 

\usepackage{lipsum}

\usepackage{mathtools}
\usepackage{a4wide}
\usepackage{caption}
\usepackage{subcaption}
\usepackage{setspace}
\usepackage{xr}
\usepackage{color}
\usepackage{soul}
\usepackage{xcolor}
\usepackage{mathtools}
\usepackage{setspace}
\usepackage{float}
\usepackage{enumerate}

\usepackage{url}
\usepackage[utf8]{inputenc}
\usepackage[T1]{fontenc}
\usepackage{lmodern}
\usepackage[capitalise,noabbrev,nameinlink]{cleveref}

\usepackage{tabularx}
\usepackage{array}
\usepackage{booktabs}
\usepackage{multirow}

\usepackage{eurosym}
\usepackage{textcomp}

\usepackage[affil-it]{authblk}
\usepackage{longtable}

\newtheorem{theorem}{Theorem}[section]

\newtheorem{lemma}[theorem]{Lemma}

\newtheorem{proposition}[theorem]{Proposition}

\theoremstyle{definition}

\newtheorem{remark}[theorem]{Remark}

\begin{document}



\title{
Time geodesics on a slippery cross slope \\ under gravitational wind 
}

\author{Nicoleta Aldea$^1$}
\author{Piotr Kopacz$^{2}$}



\affil{\small{$^1$Transilvania University of Bra\c{s}ov, Faculty of Mathematics and Computer Science\\ Iuliu Maniu 50, Bra\c{s}ov, Romania}}
\affil{$^2$Gdynia Maritime University, Faculty of Navigation \\  Al. Jana Paw{\l}a II 3, 81-345 Gdynia, Poland}

\date{}
\date{{\normalsize {\small {e-mail:} \texttt{codruta.aldea@unitbv.ro, p.kopacz@wn.umg.edu.pl}}}}
\maketitle

\begin{abstract} 

In this work, we pose and solve the time-optimal navigation problem considered on a slippery~mountain slope modeled by a Riemannian manifold of an arbitrary dimension, under the action of a cross gravitational wind. The impact of both lateral and longitudinal components of gravitational wind on the time geodesics is discussed. The varying along-gravity effect depends on traction in the presented model, whereas the 
 cross-gravity additive is 
 taken entirely 
in the equations of motion, for any direction 
 and gravity force. 
 We obtain the conditions for strong convexity and the purely geometric solution to the problem is given by a new 
 Finsler metric, which belongs to the type of general $(\alpha, \beta)$-metrics. 
 The proposed model 
 enables us to create 
 a direct link between the Zermelo navigation problem and the slope-of-a-mountain problem under the action 
 of a cross gravitational wind. 
Moreover, the behavior 
 of the Finslerian indicatrices 
 and time-minimizing trajectories in relation to the traction coefficient 
and 
 gravitational wind force are explained and illustrated 
 by a few examples in dimension two. This also 
 compares the corresponding solutions on the slippery slopes under various cross- and along-gravity effects, including the classical Matsumoto's slope-of-a-mountain problem and Zermelo's navigation.

\end{abstract}






\section{Introduction}

\label{Intro}

We begin by recalling 
and briefly describing  the types of time-optimal navigation problems, in particular, in the presence of a gravitational wind, which have generally been investigated with a purely geometric approach in the literature on Riemann-Finsler geometry. We hope that such concise presentation of the preceding outcomes 
  will lead the reader to the main concept of the current study in a clear way, since our work is a natural complementation and generalization of some previous results. Although the notion of gravitational wind in the context of the Zermelo navigation data (\cite{BRS}) was  introduced very recently (\cite{slippery}), this lets us collect and describe all time-optimal problems mentioned below, including the classical ones, in a convenient, unified and effective manner. Roughly speaking, it also stands for the leitmotiv of the below summary and the new investigations presented further in this paper. The key aspect we want to emphasize 
 is the \textit{type 
 and range of compensation of the gravity effects} in the described models of the mountain slopes, which then characterize 
 the general equations of motion and, consequently, the related Finsler metric in each case. Such approach to the subject enables us to put and review the preceding results in a more general perspective,  based on the theory being developed (\cite{slippery,cross}). 

\subsection{
The time-optimal navigation problems 
 on a mountain slope under various gravity effects 
}

In order to gain some intuition, we consider the 2-dimensional models of the slopes including the  inclined planes in what follows, while the general purely geometric solutions to the time-optimal navigation problems described are valid for an arbitrary dimension. We first refer to 
 the classical problems investigated initially in the works of Ernst Zermelo and Makoto Matsumoto \cite{Zer0,Zer,matsumoto}. 


\paragraph{\textit{Zermelo's navigation}}One of the most iconic examples 
 in optimal control, as well as Riemann-Finsler geometry, is Zermelo's navigation problem (ZNP for short). This is about finding time-minimizing paths of a craft which moves at a maximum speed with respect to a surrounding and flowing medium between two positions at sea, on the river or in the air in the presence of  variable water stream or wind, modeled as a perturbing vector field $W$. The problem was initially formulated in the Euclidean spaces of low dimensions and solved with the use of variational calculus by Zermelo in 1930's \cite{Zer0, Zer}. Further on,  ZNP was investigated with application of Pontryagin’s maximum principle in optimal control. More recently, the problem was recalled with purely geometric formulation and generalized to Riemannian manifolds $(M, h)$ of an arbitrary dimension in Finsler geometry and spacetime \cite{SH, chern_shen, BRS, JS,cr
}. The key geometric property is that the Finslerian length of a piecewise $C^\infty$-curve on a manifold $M$ can 
 be interpreted as the time measure. In particular, if the acting vector field is weak and space-dependent, then the solution is given by a Randers metric, which has various applications in physics and mathematical modelling; see, for example, \cite{markvorsen,brody3}. Next, the Zermelo navigation was referred, extended and generalized in the applied and theoretical studies by many authors, who considered, among others, stronger winds, time dependence or variable speed of a navigating craft \cite{Y-Sabau,JS,kopi6,Yajima}. 

For our purposes herein, note that the entire wind vector is always taken into consideration in the equations of motion (equivalently, in the related 
 Finsler metric) in the Zermelo problem. Furthermore, remark that a gradient vector field can be treated as a particular wind in the sense of the navigation data $(h, W)$ 
(\cite{Nicprw}). This  relates 
 to the notion of a \textit{gravitational wind} being a component $\mathbf{G}^{T}$ of a gravitational field $\mathbf{G}=\mathbf{G}^{T}+\mathbf{G}^{\perp}$, where 
$\mathbf{G}^{T}$ 
 is tangent to $M$ and acts in the steepest downhill direction (considering a 2-dimensional model of the slope), and $\mathbf{G}^{\perp }$ is normal 
 to $M$. 
 In such case, with the wind data $W:=\mathbf{G}^{T}$ the general equations of motion  read 
\begin{equation*}
v_{\textit{\tiny{ZNP}}}=u+\mathbf{G}^{T}, 
\end{equation*}
where $u$ denotes a self-velocity (a control vector) and $||u||_h=1$ represents the maximum self-speed of a sailing or flying craft. The resultant indicatrix 
 is given by an $h$-circle (an ellipse) rigidly translated by $W$. 
In general, $\mathbf{G}^{T}$ depends on the position-dependent gradient vector field related to 
 a slope $M$ and a given acceleration of gravity, where $||\mathbf{G}^{T}||_{h}=\sqrt{h(\mathbf{G}^{T},\mathbf{G}^{T})}$. 


\paragraph{\textit{Matsumoto's slope of a mountain
}}Another well-known problem in Finsler geometry, which is also related to minimization of time, was studied first by Matsumoto in \cite{matsumoto} (MAT for short). The author investigated time geodesics on a slope of a mountain under the effect of gravity, taken into consideration that walking uphill is more tiring than walking downhill. It was assumed in his model that the cross-gravity additive, i.e. $\textnormal{Proj}_{u^{\perp }}\mathbf{G}^{T}$ is always canceled and thus, it has not any  impact on the resultant path\footnote{This issue was justified by Matsumoto in a word, saying that \textit{the component perpendicular to the velocity $u$ is regarded to be canceled by planting the walker's legs on the  road determined by $u$} \cite[p. 19]  {matsumoto}}, where $u^{\perp}$ denotes the orthogonal 
 direction to $u$.  
At the same time, the along-gravity effect  is considered to be maximum, i.e.  $\textnormal{Proj}_{u}\mathbf{G}^{T}$, for each direction of motion and wind force $||\mathbf{G}^{T}||_{h}$, and only the longitudinal (w.r.t. direction of the walker's self-velocity $u$) component of the gravitational wind is taken into account in the equations of motion\footnote{For brevity, we shall write $\mathbf{G}_{MAT}$ for $\textnormal{Proj}_{u}\mathbf{G}^{T}$, and $\mathbf{G}_{MAT}^{\perp}$ for $\textnormal{Proj}_{u^{\perp }}\mathbf{G}^{T}$ in the attached figures and text on further reading.}. The solution is represented here by a Matsumoto metric (a.k.a. a slope metric), which has also been used next, for example, to describe a range of the wildfire spreading mechanism \cite{markvorsen,JPS}. The resultant velocity in this case is defined as follows 
\begin{equation*}
v_{\textit{\tiny{MAT}}}=u+\text{Proj}_{u}\mathbf{G}^{T},
\end{equation*}
and it is evident that $\mathbf{G}^{T}=\textnormal{Proj}_{u}\mathbf{G}^{T} + \textnormal{Proj}_{u^{\perp }}\mathbf{G}^{T}$. 
In particular, this implies that the velocities $u$ and $v$ are always collinear, which contrasts to all other navigation problems described in this section. The related 
  indicatrix is given by a lima\c{c}on in a two-dimensional model of the slope. 

For comparison, it is worth of noting that the gravitational force $||\mathbf{G}^{T}||_{h}$ on the slope 
 always pushes a walker in one fixed direction, i.e. the steepest descent, whereas the perturbing wind in the Zermelo navigation in general blows 
 in an arbitrary direction on 
 a surface (more general, a Riemannian manifold). 
Mention that the solutions to both problems, i.e. the Randers 
 (or Kropina) 
 and Matsumoto 
 metrics belong to a type of $(\alpha ,\beta )$-metrics. 
For clarity, both ZNP and MAT are distinguished in the slope model in \cref{fig_slippery_slope}. 


\paragraph{\textit{Slippery slope}} Recently, a direct link between the Matsumoto slope-of-a-mountain problem and the Zermelo navigation problem  in a gravitational wind  was  shown in \cite{slippery}. Both these problems were generalized and studied in a \textit{slippery slope} model, including a traction coefficient\footnote{On further reading, this will be called a \textit{cross-traction coefficient $\eta$}  to distinguish it from an \textit{along-traction coefficient $\tilde{\eta}$} introduced in this paper.}  introduced therein, expressed by a real parameter $\eta\in[0, 1]$. In this setting the longitudinal component of gravitational wind always acts with full power, for each direction of motion and wind force $||\mathbf{G}^{T}||_{h}$, whereas the lateral one is subject to compensation due to traction. The solution is represented by a slippery slope metric, which belongs to a larger class of general $(\alpha, \beta )$-metrics, and the 
 related general equation of motion reads
\begin{equation}
v_\eta=u+(1-\eta)\text{Proj}_{u^{\perp }}\mathbf{G}^{T}+\text{Proj}_{u}\mathbf{G}^{T}. 
\label{eqs_slippery}
\end{equation}

Moreover, MAT and ZNP 
 become the edge and particular cases of the study presented, 
 what follows from the last equation. Namely, if 
 traction on the slope is minimum (zeroed), then the slippery slope problem leads to Zermelo's navigation, where the gravitational drift off the desired route is strongest. We can say that such controlled motion including sliding like on a perfect ice resembles free sailing or flying in the presence of water or air currents in the spirit of Zermelo model. On the other hand, if 
 traction is maximum ($\eta=1$), and consequently, the lateral (transverse) component of the gravitational wind is compensated entirely, then there does not occur any drag to the side caused by gravity, so $v_1$ and $u$ are collinear\footnote{$v_1\equiv v_{MAT}$, i.e. $v_\eta$ with $\eta=1$  in the slippery slope problem.}. This yields classical Matsumoto's problem described above. The graphical illustration of the situation on a slippery slope is presented in \cref{fig_slippery_slope}. 

\begin{figure}[h!]
\centering
~\includegraphics[width=0.68\textwidth]{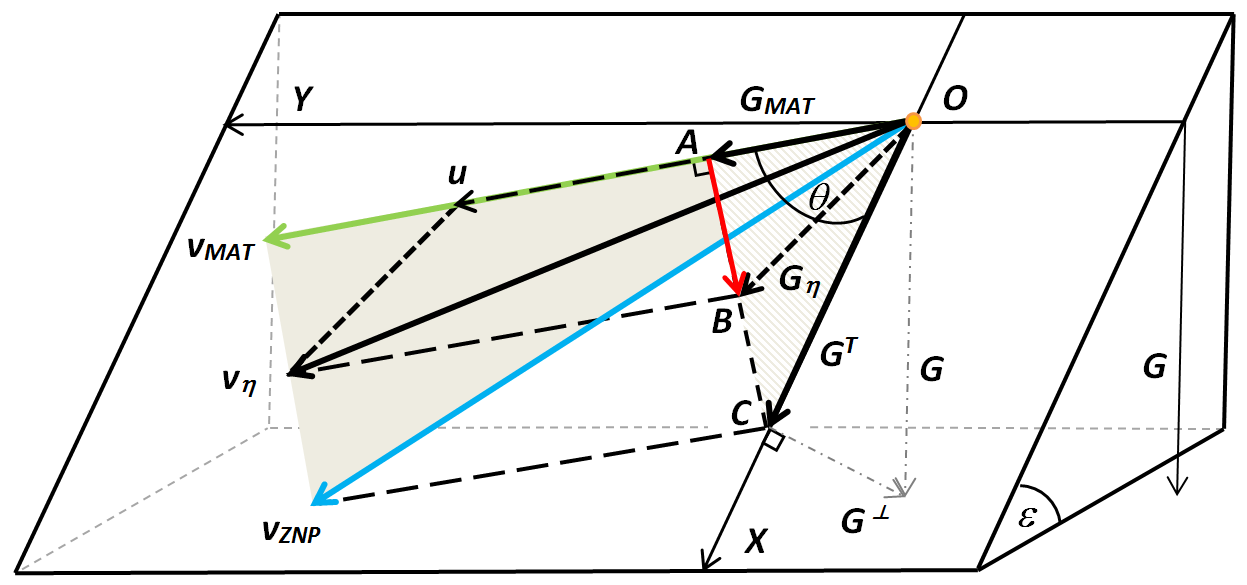} 
~\includegraphics[width=0.29\textwidth]{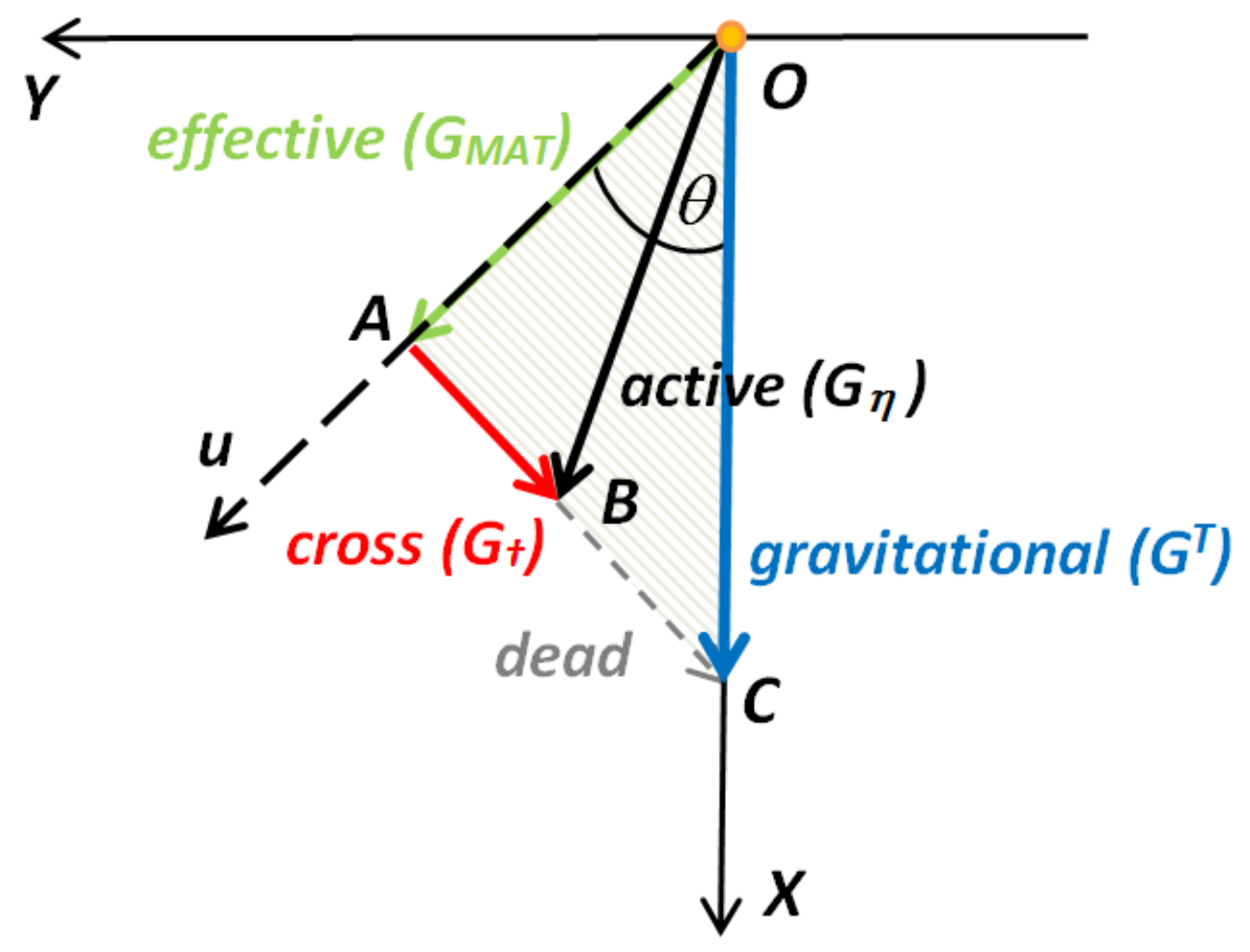}
\caption{Left: A model of a planar slippery slope $M$ being an inclined plane of the slope
angle $\protect\varepsilon$ in $\mathbb{R}^{3}$, under the gravity field $\mathbf{G}=\mathbf{G}%
^{T}+\mathbf{G}^{\perp}$, which is perpendicular to the horizontal plane (the base of the
slope). The gravitational wind $\mathbf{G}^{T}$ ``blows'' tangentially to the slope in the steepest downhill direction $X$ and $\mathbf{G}^{\perp}$ is the component of gravity normal to the slope $M$; $OX\perp OY$, $O\in M$. The resultant velocity is represented by $v_{\eta}$ including the boundary cases for $\eta=0, 1$, i.e. the Zermelo case ($v_{ZNP}$, blue) and the Matsumoto case ($v_{MAT}$, green), respectively. The lateral component $\protect\overrightarrow{AB}$ (a cross wind; red) of the active wind  $\mathbf{G}_{\eta}=\protect\overrightarrow{OB}$ depends in particular on the traction coefficient $\eta$. 
Right: A gravitational wind and its decompositions on the slippery slope $M$, where $OA\perp AC$. 
The gravitational wind $\mathbf{G}^{T}$ is a vector sum of an active wind and a dead wind ($\protect\overrightarrow{BC}$). The active wind is decomposed into two orthogonal components: an effective wind (which coincides with  $\mathbf{G}_{MAT}$ in this case) and a cross wind $\mathbf{G}_{\dag}$. The cross-gravity
effect in general varies on the slippery slope, i.e. $||\mathbf{G}_{\dag}||_h\in[0, ||\textnormal{Proj}_{u^{\perp}}\mathbf{G}^{T}||_h]$, while the along-gravity effect  is always full, i.e. equal to $||\mathbf{G}_{MAT}||_h$.    
The  direction $\theta$ of ``unperturbed'' motion indicated by the Riemannian velocity $u$ (the control vector, dashed black) is measured clockwise from $OX$, where $||u||_h=1$. 
}
\label{fig_slippery_slope}
\end{figure}

For convenience and clarity, we 
  now recall briefly the relevant material 
 from \cite{slippery}, which will also be used on further reading, thus making our exposition self-contained.    
The part of gravitational wind that is cancelled due to nonzero traction is named a \textit{dead wind}, and this has actually 
 no influence on the optimal trajectories (time geodesics). In turn, the remaining part of gravitational wind, which is included in the equations of motion on the slippery slope, we call an \textit{active wind} (denoted by $\mathbf{G}_{{\eta}}$ in the slippery slope model). We can rewrite \eqref{eqs_slippery} as $v_\eta=u+\mathbf{G}_{{\eta}}$, \linebreak 
 where $\mathbf{G}_{{\eta}}=(1-\eta)\text{Proj}_{u^{\perp }}\mathbf{G}^{T}+\text{Proj}_{u}\mathbf{G}^{T}$. 

 Clearly, a vector sum of the active and dead winds in each type of navigation problem under consideration 
 yields the entire gravitational wind $\mathbf{G}^{T}$. However, both these components in general are not orthogonal to each other. 
Furthermore, the active wind can be decomposed into two orthogonal components, 
  namely, an \textit{effective wind} and a \textit{cross wind}, being its projections upon the velocity $u$ and the orthogonal 
 direction to $u$, 
 respectively. For instance, the maximum effective wind together with the minimum (zeroed) cross wind, for each direction of motion and gravity force $||\mathbf{G}^{T}||_{h}$, determines the classical Matsumoto's  slope-of-a-mountain problem. On the other hand, the reversed setting yields the cross slope problem (\cite{cross}) recalled  right below. For clarity's sake, we also mention that a vector sum of the strongest effective wind and strongest cross wind gives full gravitational wind, so $\mathbf{G}_{{\eta}}=\mathbf{G}^{T}$, as is in ZNP.
 

\paragraph{\textit{Cross slope}}A different approach to time-optimal navigation on a mountain slope has been presented in \cite{cross} very recently (CROSS for short). Namely, unlike Matsumoto 
 and for comparison to  his standard set-up, the transverse component of gravitational wind was taken into account entirely in the equations of motion, whereas the along-gravity effect was reduced completely (see \cref{fig_slippery_xslope}, where CROSS is presented as a particular and edge case in further  investigation). 
Observe that, in general, the impact of the lateral component on the resultant velocity, and  on time geodesics, can be stronger 
  than 
 the longitudinal one, depending on the direction of motion 
on the hillside. Thus, it is reasonable to consider a reversed scenario, including the full cross-gravity effect and vanishing the entire along-gravity additive, which stands for the dead wind in this case.  
Recall from \cite{cross} that the slope model under the action of only cross-gravity impact has been called a \textit{cross slope}. The solution is given by another Finsler metric of general  $(\alpha ,\beta )$-type, i.e. a cross-slope metric. 
 The resultant velocity reads here
\begin{equation*}
v_\dag=u+\text{Proj}_{u^{\perp }}\mathbf{G}^{T}.
\end{equation*}
   The corresponding indicatrix is given by a lima\c{c}on, however different from that of the Matsumoto problem. In this case the active and dead winds coincide with the cross and effective winds, respectively. Such setting can be applied, for example, to the description of indicatrices being the pedal (algebraic) curves and surfaces \cite{Sabau_pedal}. In nature there is some analogy to the behavior of animals that move sideways, while being influenced by a natural force field, e.g. a sidewinder rattlesnake under gravity, or a hummingbird in wind. 
Furthermore, this kind of movement is also related to the linear transverse ship's sliding motion side-to-side (as known as sway) on a dynamic surface of a sea, 
while the linear front-back motion (as known as surge) is stabilized. 


\subsection{Model of a slippery cross slope under gravitational wind}

The very recent results \cite{cross, slippery} have been encouraging enough to merit further investigation. Continuing the above line of research naturally led us to the new model of a slippery cross slope presented below. 

\paragraph{\textit{Slippery cross slope}}Let us observe that actually each of two orthogonal components of gravitational wind, i.e. $\textnormal{Proj}_{u}\mathbf{G}^{T},  \textnormal{Proj}_{u^{\perp }}\mathbf{G}^{T}$ can be reduced partially due to traction, making use of a real parameter, and not only entirely like in \cite{matsumoto} (the lateral one) or \cite{cross} (the longitudinal one). As mentioned above, this has already been done in the case of the transverse component in a slippery slope model, where the cross-traction coefficient $\eta$ runs through the interval $[0, 1]$, linking MAT and ZNP \cite{slippery}. 
By analogy to such 
 compensation of $\mathbf{G}^{T}$,  we aim at considering a slippery slope model in the current study, however concerning the along-gravity scaling and introducing another parameter called an \textit{along-traction coefficient} $\tilde{\eta}\in[0, 1]$. 
We assume that, while the Earth's gravity impacts a walker or a craft on the slope, the 
cross wind being perpendicular to a desired direction of motion $u$ is regarded to act always entirely, whereas the effective wind,  
which pushes the craft downwards,  
can be compensated as depending on 
 traction.    
 In other words, the proposed model refers to a mountain slope, fixing the maximum cross-track additive continuously, for any direction of motion $\theta$ and gravity force $||\mathbf{G}^{T}||_{h}$ and admitting the along-track changes at the same time (longitudinal sliding). 
 Consequently, the corresponding Finslerian indicatrix in the new setting will be based on the direction-dependent deformation of the background Riemannian metric $h$ again. However, unlike all the preceding  problems listed above, the equations of motion in the general form will now be 
\begin{equation}
v_{\tilde{\eta}}=u+\text{Proj}_{u^{\perp}}\mathbf{G}^{T}+(1-\tilde{\eta})\text{Proj}_{u}\mathbf{G}^{T}.  
\label{eqs_motion}
\end{equation}
 Thus, we can say that the influences of both components of gravitational wind are now somewhat reversed in comparison to the slippery slope investigated in \cite{slippery}; cf. Eq. \eqref{eqs_slippery}.  To simplify the writing and to be in agreement with our previous notation, we will write $\mathbf{G}_{MAT}$ for $\textnormal{Proj}_{u}\mathbf{G}^{T}$, and $\mathbf{G}_{MAT}^{\perp}$ for $\textnormal{Proj}_{u^{\perp }}\mathbf{G}^{T}$.  
The new model of the mountain slope we call a \textit{slippery cross slope} and the related time-minimal navigation  generalizes or complements previous investigations.

\begin{figure}[h!]
\centering
\includegraphics[width=0.70\textwidth]{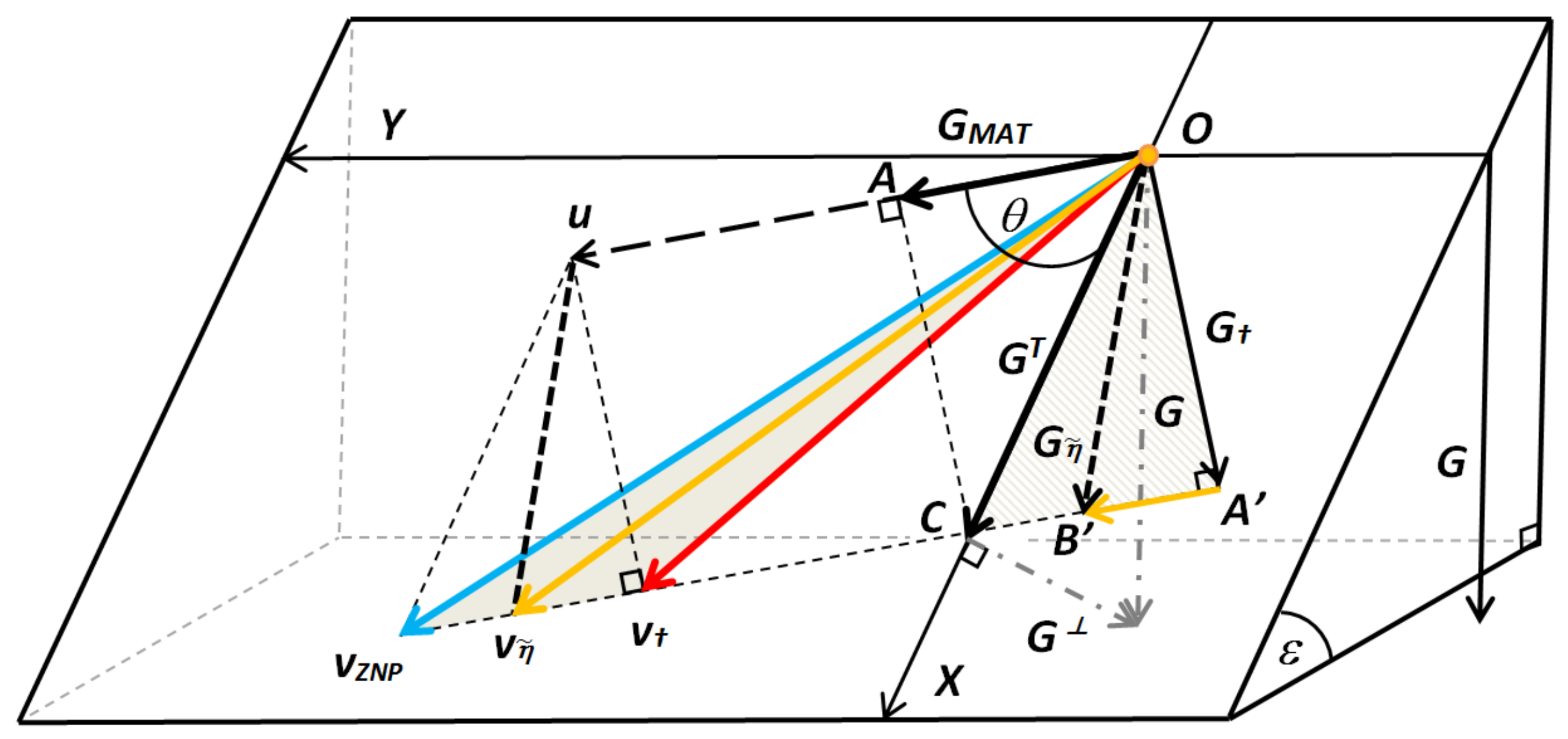}
\includegraphics[width=0.29\textwidth]{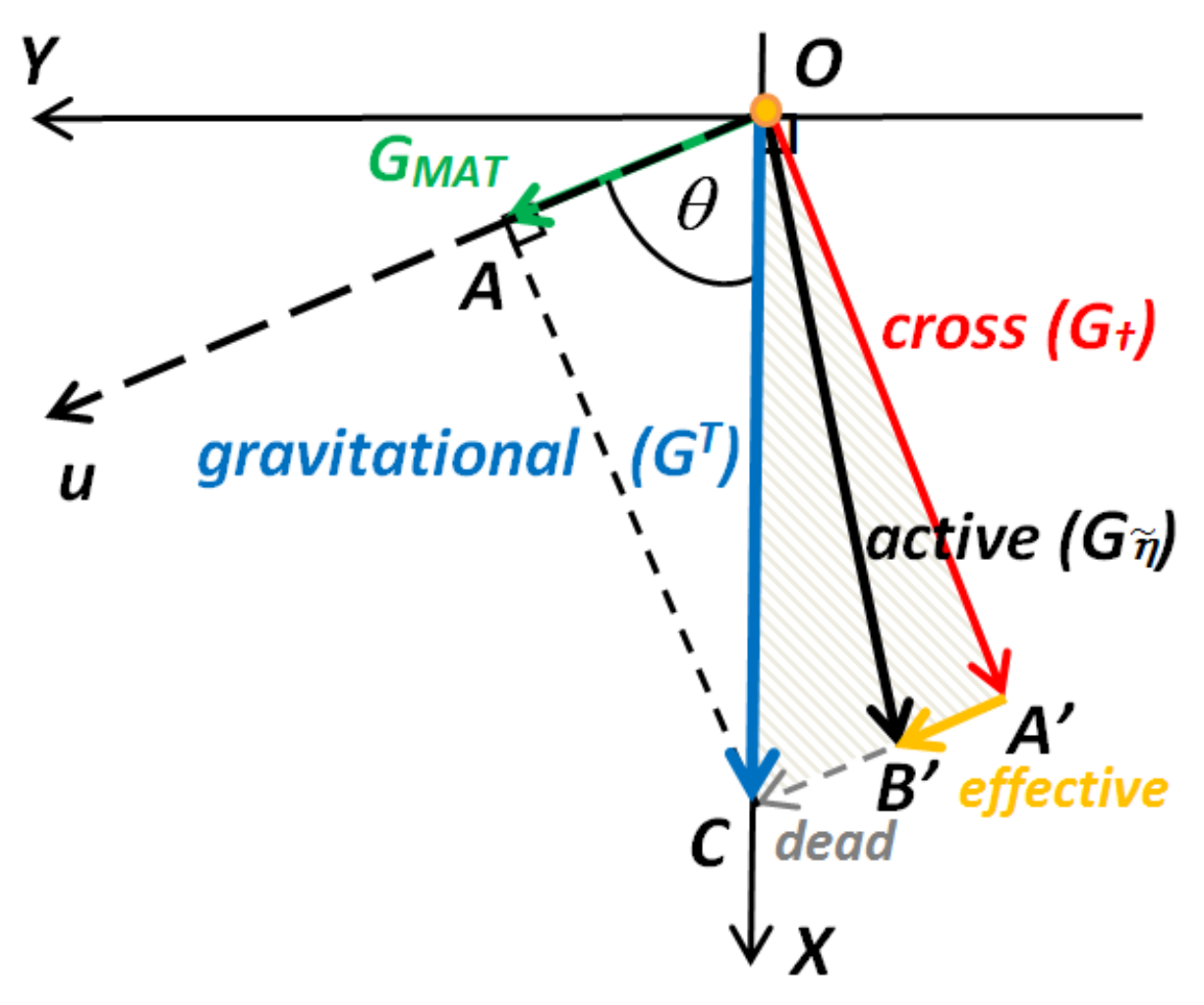}
\caption{Left: A model of a planar slippery cross slope $M$ under gravity $\mathbf{G}=\mathbf{G}^{T}+\mathbf{G}^{\perp}$, being analogous to the slippery slope as in \cref{fig_slippery_slope}. 
Now the resultant velocity is represented by $v_{\tilde{\eta}}$ (yellow, $\tilde{\eta}\in[0, 1]$) including the boundary cases for $\tilde{\eta}\in\{0, 1\}$, i.e. the Zermelo case ($v_{ZNP}$, blue) and the cross slope case ($v_{\dag}$, red), respectively. The longitudinal component $\protect\overrightarrow{A'B'}$  of the active wind  $\protect\overrightarrow{OB'}$ (denoted by  $\mathbf{G}_{\tilde{\eta}}$) w.r.t. $u$ depends in particular on the along-traction coefficient $\tilde{\eta}$. 
Right: A gravitational wind and its decompositions on the slippery cross slope $M$, where $OA\perp OA^{\prime }$. The gravitational wind $\mathbf{G}^{T}$ is a vector sum of an active wind and a dead wind ($\protect\overrightarrow{B'C}$). The former 
 is in turn decomposed into two orthogonal components: an effective wind $\protect\overrightarrow{A'B'}$ and a cross wind  $\protect\overrightarrow{OA^{\prime }}=\mathbf{G}_{\dag}$ (which coincides with $\textnormal{Proj}_{u^{\perp}}\mathbf{G}^{T}$ in this case). The along-gravity force 
 in general varies from 0 to $||\mathbf{G}_{MAT}||_h$ on the slippery cross slope, while the cross-gravity effect is always full, i.e. equal to $||\textnormal{Proj}_{u^{\perp}}\mathbf{G}^{T}||_h$.    
The  direction $\theta$ of ``unperturbed'' motion indicated by the Riemannian velocity $u$ (the control vector, dashed black) is measured clockwise from $OX$, where $||u||_h=1$. 
}
\label{fig_slippery_xslope}
\end{figure}

From Eq. \eqref{eqs_motion} it follows manifestly that the self-velocity $u$ is 
perturbed now by $\mathbf{G}_{\tilde{\eta}}$. Hence, the resultant velocity is expressed as\footnote{For brevity, we shall drop the subscript $\tilde{\eta}$ on $v_{\tilde{\eta}}$ when confusion is unlikely.  
}
\begin{equation}
v_{\tilde{\eta}}=u+ \mathbf{G}_{\tilde{\eta}},
\label{janta}
\end{equation}
where the active wind reads
$\mathbf{G}_{\tilde{\eta}}=
\mathbf{G}_{MAT}^{\perp}+(1-\tilde{\eta})\mathbf{G}_{MAT},$
where the former component stands for cross wind and the latter for effective wind. 
An equivalent formulation of the last relation is
$\mathbf{G}_{\tilde{\eta}}=\tilde{\eta} \mathbf{G}_{\dag}+(1-\tilde{\eta} )\mathbf{G}^{T},$  
since the cross wind is ``blowing'' with maximum force,  so  $\mathbf{G}_{\dag}=\mathbf{G}_{MAT}^{\perp}$ in this case. 
Moreover, it follows evidently from the above that 
\begin{equation} \mathbf{G}_{\tilde{\eta}}= -\tilde{\eta}\mathbf{G}_{MAT}+\mathbf{G}^{T},   
\label{nn}
\end{equation}
where the component $\tilde{\eta}\mathbf{G}_{MAT}$ represents the dead wind.  
In particular, it is reasonable to expect that the edge cases, 
 i.e. $\tilde{\eta}=1$ and $\tilde{\eta}=0$, will describe now, respectively, the cross-slope navigation (the action of maximum cross wind and minimum effective wind
\footnote{For clarity's sake, see \cref{fig_slippery_xslope},  where $\protect\overrightarrow{OA}=\text{Proj}_{u}\mathbf{G}^{T}$ is in general the maximum effective wind, and $\protect\overrightarrow{OA^{\protect\prime }}=\text{Proj}_{u^{\perp}}\mathbf{G}^{T}$ the maximum cross wind, for any given $\protect\theta $ and $||\protect\mathbf{G}^{T}||_{h}$. A component of the gravitational wind $\protect\mathbf{G}^{T}$ acts in full force if it is not reduced (partially or entirely), e.g. due to traction or drag.}), i.e. $\mathbf{G}_{\tilde{\eta}}=\mathbf{G}_{MAT}^{\perp}$, and the Zermelo navigation (the action of maximum both cross and effective winds
), i.e. $\mathbf{G}_{\tilde{\eta}}=\mathbf{G}^{T}$. For the sake of clarity, see  \cref{fig_slippery_xslope}.  

 Furthermore, it can be seen 
 that the slippery slope problem with the cross-traction coefficient $\eta=0$ from \cite{slippery} and  the current investigation with the new 
 along-traction coefficient $\tilde{\eta}=0$ should coincide. Then this means that the scenario
like in Zermelo's navigation under weak gravitational wind $\mathbf{G}^{T}$ will be located somewhat right in the middle between both approaches pieced together. It will also  stand  naturally for the boundary and meeting case of the slippery slope and slippery-cross-slope solutions, where the gravitational wind ``blows'' in full\footnote{Both effective and cross winds are maximal in the Zermelo case, for any direction $\theta$ and wind force $||\mathbf{G}^{T}||_{h}$.} on the mountain slope, and the time geodesics come from Finsler metric of Randers type.   
 
Finally, the current model of a slippery slope under cross gravitational wind complements the preceding investigations on the slope-of-a-mountain problems in a natural way. Namely, this fills in a missing part regarding the compensation of the along-gravity effect concerning the   direction of motion indicated by the velocity $u$.  
As we will see on further reading, the obtained strong convexity conditions, being the basis for  desired optimality of the trajectories on the slope, differ significantly from all those of the preceding navigation problems discussed above. In particular, for some $\tilde{\eta}$, it will be admitted the norm of gravitational wind $||\mathbf{G}^{T}||_{h}$ to be greater than $||u||_h=1$. For comparison, recall that the convexity condition in the Zermelo navigation, i.e. for a Randers metric is $||\mathbf{G}^{T}||_{h}<1$ (\cite{BRS}) and the corresponding conditions in both the Matsumoto and cross-slope metrics are most restrictive among all others, i.e $||\mathbf{G}^{T}||_{h}<1/2$ (\cite{slippery, cross}). Moreover, in contrast to the situation on the slippery slope investigated in \cite{slippery}, the behavior of the Finslerian indicatrix (time front), being now subject to along-traction expressed by $\tilde{\eta}$, is quite different. For instance, the new indicatrix 
 crosses both edge cases (ZNP and CROSS) in two fixed points, while the parameter $\tilde{\eta}$ is running through the interval $(0, 1)$. This is shown in the presented example with an inclined plane  (\cref{Sec_3}), comparing the outcomes obtained in this work and the preceding studies. 

\subsection{Statement of the main results 
}

The problem of time-optimal navigation on a slippery slope under the
cross-gravity effect \linebreak (S-CROSS for short) can be posed as follows

\begin{quote}
\begin{itshape}
\noindent {Suppose a craft or a vehicle goes 
 on a horizontal plane at maximum constant speed, while gravity acts 
perpendicularly 
on this plane. Imagine the craft 
 moves 
 now on a slippery cross slope of a mountain, 
with a given along-traction coefficient and under 
 gravity.  
What path should be followed by the craft 
to get 
from one point to another in the minimum time?}
\end{itshape}
\end{quote}

\noindent Our first goal is to provide the Finsler metric which serves as
the solving tool for the S-CROSS problem. More precisely, in the general context
of an $n$-dimensional Riemannian manifold with $\mathbf{G}^{T}=-\bar{g}%
\omega ^{\sharp }$, where $\omega ^{\sharp }$ is the gradient vector field
and $\bar{g}$ is the rescaled gravitational acceleration $g$ (see %
\Cref{Sec_3}), we prove the following result:

\begin{theorem}
\label{Theorem1}\textnormal{(Slippery-cross-slope metric)} Let a slippery
cross slope of a mountain  be an $n$-dimensional Riemannian manifold $(M,h)$, 
$n>1$, with the along-traction coefficient $\tilde{\eta}\in \lbrack 0,1]$ and the gravitational wind $\mathbf{G}^{T}$ on $M$. The time-minimal paths on $(M,h)$
under the action of an active wind $\mathbf{G}_{\tilde{\eta}}$ as in Eq. %
\eqref{nn} are the geodesics of the slippery-cross-slope metric $\tilde{F}_{%
\tilde{\eta}}$ which satisfies 
\begin{equation}
\tilde{F}_{\tilde{\eta}}\sqrt{\alpha ^{2}+2\bar{g}\beta \tilde{F}_{\tilde{%
\eta}}+||\mathbf{G}^{T}||_{h}^{2}\tilde{F}_{\tilde{\eta}}^{2}}=\alpha
^{2}+(2-\tilde{\eta})\bar{g}\beta \tilde{F}_{\tilde{\eta}}+(1-\tilde{\eta})||%
\mathbf{G}^{T}||_{h}^{2}\tilde{F}_{\tilde{\eta}}^{2},  \label{TH_mama}
\end{equation}%
with $\alpha =\alpha (x,y),$ $\beta =\beta (x,y)$ given by Eq. \eqref{NOT},
where either $\tilde{\eta}\in \lbrack 0,\frac{1}{3}]$ and $||\mathbf{G}%
^{T}||_{h}<\frac{1}{1-\tilde{\eta}}$, or $\tilde{\eta} \in (\frac{1}{3},1]$ and $||%
\mathbf{G}^{T}||_{h}<\frac{1}{2\tilde{\eta}}$. In particular, if $\tilde{\eta%
}=1$, then the slippery-cross-slope metric yields the cross-slope
metric, and if $\tilde{\eta}=0$, then it is the Randers metric which solves
the Zermelo navigation problem on a Riemannian manifold under a
gravitational wind $\mathbf{G}^{T}.$
\end{theorem}

\noindent In addition, S-CROSS, which provides the slippery-cross-slope metric $\tilde{F}_{\tilde{\eta}}$ by Eq. \eqref{TH_mama}, leads to a new application and a natural model of Finsler spaces with general $(\alpha ,\beta )$ metrics (\cite{Yu}).

To find the time geodesics of the slippery-cross-slope metric, we
exploit its geometrical and analytical properties, the main key being Eq. %
\eqref{TH_mama}, and answering the above stated question this way. 
Thus, our second main result is
\begin{theorem}
\label{Theorem2} \textnormal{(Time geodesics)} Let a slippery cross slope of a
mountain  be an $n$-dimensional Riemannian manifold $(M,h)$, $n>1$, with the along-traction coefficient $\tilde{\eta}\in \lbrack 0,1]$ and the gravitational wind $\mathbf{G}^{T}$ on $M$. The time-minimal paths on $(M,h)$ under the action of an active wind $\mathbf{G}_{\tilde{\eta}}$ as in Eq. \eqref{nn}
are the time-parametrized solutions $\gamma (t)=(\gamma ^{i}(t)),$ $i=1,...,n$ of the ODE
system
\begin{equation}
\ddot{\gamma}^{i}(t)+2\tilde{\mathcal{G}}_{\tilde{\eta}}^{i}(\gamma (t),\dot{%
\gamma}(t))=0,  \label{GGG}
\end{equation}%
where%
\begin{eqnarray*}
\tilde{\mathcal{G}}_{\tilde{\eta}}^{i}(\gamma (t),\dot{\gamma}(t)) &=&%
\mathcal{G}_{\alpha }^{i}(\gamma (t),\dot{\gamma}(t))+\left[ \tilde{\Theta}%
(r_{00}+2\alpha ^{2}\tilde{R}r)+\alpha \tilde{\Omega}r_{0}\right] \frac{\dot{%
\gamma}^{i}(t)}{\alpha } \\
&&-\left[ \tilde{\Psi}(r_{00}+2\alpha ^{2}\tilde{R}r)+\alpha \tilde{\Pi}r_{0}%
\right] \frac{w^{i}}{\bar{g}}-\tilde{R}w_{\text{ }|j}^{i}\frac{\alpha
^{2}w^{j}}{\bar{g}^{2}},
\end{eqnarray*}%
with%
\begin{equation}
\begin{array}{l}
\mathcal{G}_{\alpha }^{i}(\gamma (t),\dot{\gamma}(t))=\frac{1}{4}%
h^{im}\left( 2\frac{\partial h_{jm}}{\partial x^{k}}-\frac{\partial h_{jk}}{%
\partial x^{m}}\right) \dot{\gamma}^{j}(t)\dot{\gamma}^{k}(t),\qquad \mathit{%
\tilde{\Psi}}=\frac{\bar{g}^{2}\alpha ^{2}}{2\tilde{E}}(\alpha ^{4}\tilde{A}%
^{2}\tilde{B}+\tilde{\eta}^{2}), \\ 
~ \\ 
r_{00}=-\frac{1}{\bar{g}}w_{j|k}\dot{\gamma}^{j}(t)\dot{\gamma}%
^{k}(t),\qquad r_{0}=\frac{1}{\bar{g}^{2}}w_{j|k}\dot{\gamma}%
^{j}(t)w^{k},\qquad r=-\frac{1}{\bar{g}^{3}}w_{j|k}w^{j}w^{k}, \\ 
~ \\ 
\tilde{R}=\frac{\bar{g}^{2}}{2\alpha ^{4}\tilde{B}}[\left( 1-\tilde{\eta}%
\right) \alpha ^{2}\tilde{B}-\tilde{\eta}],\qquad \mathit{\tilde{\Theta}}=%
\frac{\bar{g}\alpha }{2\tilde{E}}(\alpha ^{6}\tilde{A}\tilde{B}^{2}-\tilde{%
\eta}^{2}\bar{g}\beta ), \\ 
~ \\ 
\mathit{\tilde{\Omega}}=\frac{\bar{g}^{2}}{\alpha ^{2}\tilde{B}\tilde{E}}%
\{[\left( 1-\tilde{\eta}\right) \alpha ^{2}\tilde{B}-\tilde{\eta}](\alpha
^{6}\tilde{B}^{3}+\tilde{\eta}^{2}||\mathbf{G}^{T}||_{h}^{2})-\tilde{\eta}%
^{2}\alpha ^{2}(\bar{g}\beta \tilde{B}+||\mathbf{G}^{T}||_{h}^{2}\tilde{A}%
)\}, \\ 
~ \\ 
\mathit{\tilde{\Pi}}=\frac{\bar{g}^{3}}{2\alpha ^{3}\tilde{B}\tilde{E}}%
\{[\left( 1-\tilde{\eta}\right) \alpha ^{2}\tilde{B}-\tilde{\eta}](2\alpha
^{6}\tilde{A}\tilde{B}^{2}-\tilde{\eta}^{2}\bar{g}\beta )+\tilde{\eta}%
^{2}\alpha ^{2}\tilde{B}(2\alpha ^{2}+\bar{g}\beta )\}, \\ 
~ \\ 
\tilde{A}=-\frac{1}{\alpha ^{2}}\{\left[ 1-\left( 2-\tilde{\eta}\right)
\left( 1-\tilde{\eta}\right) ||\mathbf{G}^{T}||_{h}^{2}\right] -(2-\tilde{%
\eta})^{2}\bar{g}\beta -(2-\tilde{\eta})\alpha ^{2}\}, \\ 
~ \\ 
\tilde{B}=-\frac{1}{\alpha ^{2}}\{[1-2(1-\tilde{\eta})||\mathbf{G}%
^{T}||_{h}^{2}]-2(2-\tilde{\eta})\bar{g}\beta -2\alpha ^{2}\}, \\ 
~ \\ 
\tilde{C}=\frac{1}{\alpha }\left( \alpha ^{2}\tilde{B}+\bar{g}\beta \tilde{A}%
\right) ,\qquad \tilde{E}=\alpha ^{6}\tilde{B}\tilde{C}^{2}+(||\mathbf{G}%
^{T}||_{h}^{2}\alpha ^{2}-\bar{g}^{2}\beta ^{2})(\alpha ^{4}\tilde{A}^{2}%
\tilde{B}+\tilde{\eta}^{2})%
\end{array}
\label{geo_tilde}
\end{equation}%
and $\alpha =\alpha (\gamma (t),\dot{\gamma}(t)),$ $\beta =\beta (\gamma (t),%
\dot{\gamma}(t))$, and $w^{i}$ denoting the components of $\mathbf{G}^{T}$.
\end{theorem}

\noindent 

\noindent 

Our paper is organized in the following way. \Cref{Sec_2} gives a condensed exposition
of some general notions and results regarding Riemann-Finsler geometry
that are thereafter used in the proofs of the main results. \Cref{Sec_3}
is devoted to the proof of \Cref{Theorem1}, appropriately divided into two
steps which are based on the decomposition of the active wind $\mathbf{G}_{%
\tilde{\eta}}$. The first one refers to the direction-dependent deformation
of the background Riemannian metric, analyzing the impact of the
along-traction coefficient $\tilde{\eta}$ and force of dead wind $\tilde{\eta}\mathbf{G}_{MAT}$ (\Cref{Lema1,Lema2}). Further on, a rigid
translation by the gravitational wind $\mathbf{G}^{T}$ 
leads to the slippery-cross-slope metric in the second step, including primarily the
analysis of the strong convexity conditions (\Cref{Lema3}). \Cref{Sec_4}
describes the proof of \Cref{Theorem2} which is based on some technical
results formulated into \Cref{Lema4,Lema5,Prop5}. Finally, \Cref{Sec_5} deals with a few examples in dimension 2, including the discussion and comparison of indicatrices' behavior in relation to various 
 traction on the hillslope and force of a gravitational wind.

\section{Preliminaries
}
\label{Sec_2}

Before the proofs of our results are presented we briefly recall some miscellaneous notations and notions from Riemann-Finsler geometry which will be used on further reading; for more details, see,  e.g., \cite{chern_shen,B-Miron,BRS,SH,Yu,Kristaly,Y-Sabau,CJS,Musznay}.


Let $M$ be a manifold of dimension $n$ and $TM=\underset{x\in M}{\cup }%
T_{x}M $ be the tangent bundle which is itself a manifold, where $T_{x}M$ is
the tangent space at $x\in M.$ By considering the coordinates $(x^{i}),$ $%
i=1,...,n$ in a local chart in $x$, $\left\{ \frac{\partial }{\partial x^{i}}%
\right\} $ is the natural basis for $TM$, and thus, for any $y\in T_{x}M,$ $%
y=y^{i}\frac{\partial }{\partial x^{i}}$. It turns out that $(x^{i},y^{i}),$ 
$i=1,...,n$ denotes the coordinates on a local chart in $(x,y)\in TM.$ A 
\textit{Finsler metric} $F$ is a natural generalization of a Riemannian
metric on $M$. More precisely, the pair $(M,F)$ is a Finsler manifold if $%
F:TM\rightarrow \lbrack 0,\infty )$ is the continuous function with a few
additional properties:

i) $F$ is a $C^{\infty }$-function on the slit tangent bundle $%
TM_{0}=TM\backslash \{0\}$;

ii) positive homogeneity, i.e., $F(x,{c}y)={c}F(x,y)$, for all ${c}>0$;

iii) positive definiteness of the Hessian $g_{ij}(x,y)=\frac{1}{2}\frac{%
\partial ^{2}F^{2}}{\partial y^{i}\partial y^{j}},$ for all $(x,y)\in
TM_{0}. $

\noindent The last property means that the indicatrix of $F$, denoted by $%
I_{F}=\left\{ (x,y)\in TM_{0}\text{ }|\text{ }F(x,y)=1\right\}$, is strongly
convex.

A \textit{spray} on $M$ is a smooth vector field on $TM_{0},$ locally given
by $S=y^{i}\frac{\partial }{\partial x^{i}}-2\mathcal{G}^{i}\frac{\partial }{%
\partial y^{i}},$ where the functions $\mathcal{G}^{i}=\mathcal{G}^{i}(x,y),$
$i=1,...,n$ are positively\ homogeneous of degree two with respect to $y,$
i.e. $\mathcal{G}^{i}(x,{c}y)={c}^{2}\mathcal{G}^{i}(x,y)$, for all ${c}>0,$
and they denote the spray coefficients, \cite{chern_shen}. If the spray $S$
is induced by a Finsler metric ${F=}\sqrt{g_{ij}(x,y)y^{i}y^{j}}$, then its
spray coefficients $\mathcal{G}^{i}$ are defined as 
\begin{equation}
\mathcal{G}^{i}=\frac{1}{4}g^{il}\big\{\lbrack {F}^{2}]_{x^{k}y^{l}}y^{k}-[{F%
}^{2}]_{x^{l}}\big\}=\frac{1}{4}g^{il}\left( 2\frac{\partial g_{jl}}{%
\partial x^{k}}-\frac{\partial g_{jk}}{\partial x^{l}}\right) y^{j}y^{k}.
\label{S1}
\end{equation}%
Let $\gamma :[0,1]\rightarrow M$ be a regular piecewise $C^{\infty }$-curve
on $M,$ $\gamma (t)=(\gamma ^{i}(t)),$ $i=1,...,n.$ The curve $\gamma $ is
called $F$-geodesic if its velocity vector $\dot{\gamma}(t)=\frac{d\gamma }{%
dt}$ is parallel along the curve, i.e. in the local coordinates, $\gamma
^{i}(t),$ $i=1,...,n$ are the solutions of the ODE system%
\begin{equation}
\ddot{\gamma}^{i}(t)+2\mathcal{G}^{i}(\gamma (t),\dot{\gamma}(t))=0.
\label{geo1}
\end{equation}

By considering a quadratic form $\alpha ^{2}=a_{ij}(x)y^{i}y^{j}$, where $%
a_{ij}(x)$ is a Riemannian metric, and a differential $1$-form $\beta
=b_{i}(x)dx^{i}$, expressed also as $\beta =b_{i}y^{i}$, on n-dimensional
manifold $M$, we recall the following notations 
\begin{equation}
\begin{array}{l}
r_{ij}=\frac{1}{2}(b_{i|j}+b_{j|i}),\quad r_{i}=b^{j}r_{ji},\quad
r^{i}=a^{ij}r_{j},\quad r_{00}=r_{ij}y^{i}y^{j},\quad r_{0}=r_{i}y^{i},\quad
r=b^{i}r_{i},\quad \\ 
~ \\ 
s_{ij}=\frac{1}{2}(b_{i|j}-b_{j|i}),\quad s_{i}=b^{j}s_{ji},\quad
s^{i}=a^{ij}s_{j},\quad s_{0}^{i}=a^{ij}s_{jk}y^{k},\quad s_{0}=s_{i}y^{i},%
\end{array}
\label{rs}
\end{equation}%
with $b^{j}=a^{ji}b_{i}$, $b_{i|j}=\frac{\partial b_{i}}{\partial x^{j}}%
-\Gamma _{ij}^{k}b_{k}$ and $\Gamma _{ij}^{k}=\frac{1}{2}a^{km}\left( \frac{%
\partial a_{jm}}{\partial x^{i}}+\frac{\partial a_{im}}{\partial x^{j}}-%
\frac{\partial a_{ij}}{\partial x^{m}}\right) $ being the Christoffel
symbols of the metric $a_{ij}$. We notice that $\beta $ is closed if and
only if $s_{ij}=0$ \cite{chern_shen}.

Moreover, endowed with a quadratic form $\alpha ^{2}$ and a differential $1$%
-form $\beta $, the manifold $(M,F)$ is called Finsler manifold with general 
$(\alpha ,\beta )$-metric if $F$ can be read as $F=\alpha \phi (b^{2},s).$
This is for positive $C^{\infty }$-functions $\phi (b^{2},s)$ in the
variables $b^{2}$ and $s$, with $|s|\leq b<b_{0}$ and $0<b_{0}\leq \infty $,
where $s=\frac{\beta }{\alpha }$ and $b=||\beta ||_{\alpha }=\sqrt{%
a^{ij}b_{i}b_{j}}$ (for more details, see \cite{Yu}). The slippery slope
metrics and cross-slope metric, recently presented in \cite{slippery,cross},
provide the examples of general $(\alpha ,\beta )$-metrics. In particular,
if the function $\phi $ depends only on the variable $s$, then by $F=\alpha
\phi (s)$ one can emphasize $(\alpha ,\beta )$-metrics. For instance,
Randers metric $F=\alpha +\beta ,$ with $\phi (s)=1+s$ solves the Zermelo
navigation problem under the influence of a weak wind, i.e. $|s|\leq b<1$ 
\cite{chern_shen}. Another example is the Matsumoto metric $F=\frac{\alpha
^{2}}{\alpha -\beta },$ with $\phi (s)={\frac{1}{1-s}}$ and $|s|\leq b<{%
\frac{1}{2}},$ which solves Matsumoto's slope-of-a-mountain problem \cite%
{matsumoto}.

Further on, we recall few key results for our arguments.

\begin{proposition}
\label{Prop1} \cite{Yu} Let $M$ be an $n$-dimensional manifold. $F=\alpha
\phi (b^{2},s)$ is a Finsler metric for any Riemannian metric $\alpha $ and $%
1$-form $\beta ,$ with $||\beta ||_{\alpha }<b_{0}$ if and only if $\phi
=\phi (b^{2},s)$ is a positive $C^{\infty }$-function satisfying%
\begin{equation*}
\phi -s\phi _{2}>0,\text{ \ \ }\phi -s\phi _{2}+(b^{2}-s^{2})\phi _{22}>0,
\end{equation*}%
when $n\geq 3$ or%
\begin{equation*}
\phi -s\phi _{2}+(b^{2}-s^{2})\phi _{22}>0,
\end{equation*}%
when $n=2$, where $s=\frac{\beta }{\alpha }$ and $b=||\beta ||_{\alpha }$  satisfy  $|s|\leq b<b_{0}$.
\end{proposition}

\noindent We mention that $\phi _{1}$ and $\phi _{2}$ denote the derivatives
of the function $\phi $ with respect to the first variable $b^{2}$ and the
second variable $s$, respectively. In the same way, $\phi _{12}$ and $\phi
_{22}$ denote the derivatives of $\phi _{1}$ and $\phi _{2}$ with respect to 
$s.$ In particular, when $\phi $ is a function only of variable $s$, then
the derivatives $\phi _{2}$ and $\phi _{22}$ will be simply denoted by $\phi
^{\prime }$ and $\phi ^{\prime \prime }$, respectively.

\begin{proposition}
\label{Prop2} \cite{Yu} For a general $(\alpha ,\beta )$\textit{-}metric $%
F=\alpha \phi (b^{2},s)$, its spray coefficients $\mathcal{G}^{i}$ are
related to the spray coefficients $\mathcal{G}_{\alpha }^{i}$ of $\alpha $ by%
\begin{eqnarray*}
\mathcal{G}^{i} &=&\mathcal{G}_{\alpha }^{i}+\alpha Qs_{0}^{i}+\left\{
\Theta (-2\alpha Qs_{0}+r_{00}+2\alpha ^{2}Rr)+\alpha \Omega
(r_{0}+s_{0})\right\} \frac{y^{i}}{\alpha } \\
&&+\left\{ \Psi (-2\alpha Qs_{0}+r_{00}+2\alpha ^{2}Rr)+\alpha \Pi
(r_{0}+s_{0})\right\} b^{i}-\alpha ^{2}R(r^{i}+s^{i}),
\end{eqnarray*}%
where%
\begin{eqnarray*}
Q &=&\frac{\phi _{2}}{\phi -s\phi _{2}},\qquad \qquad \qquad \qquad \quad \
\ \Theta\ \ =\ \ \frac{(\phi -s\phi _{2})\phi _{2}-s\phi \phi _{22}}{2\phi
\lbrack \phi -s\phi _{2}+(b^{2}-s^{2})\phi _{22}]}, \\
\Psi &=&\frac{\phi _{22}}{2[\phi -s\phi _{2}+(b^{2}-s^{2})\phi _{22}]},\quad
\ \ \ \ \Pi \ \ =\ \ \frac{(\phi -s\phi _{2})\phi _{12}-s\phi _{1}\phi _{22}%
}{(\phi -s\phi _{2})[\phi -s\phi _{2}+(b^{2}-s^{2})\phi _{22}]}, \\
\Omega &=&\frac{2\phi _{1}}{\phi }-\frac{s\phi +(b^{2}-s^{2})\phi _{2}}{\phi 
}\Pi, \qquad \ R\ \ =\ \frac{\phi _{1}}{\phi -s\phi _{2}}.
\end{eqnarray*}
\end{proposition}

Beyond the assigned meaning of the Zermelo navigation as the research
problem, it is also applied as an effective technique to construct new
Finsler metrics by perturbing an arbitrary Finsler metric by a vector field $%
W,$ i.e. a time-independent wind on a manifold $M$, under some restrictions.
In particular, considering the background metric as a Riemannian one,
denoted by $h,$ the Randers metrics solve Zermelo's problem of navigation in
the case of weak winds $W$, i.e. $||W||_{h}<1$ \cite{BRS,chern_shen}.
Moreover, if $W$ is a critical wind, i.e. $||W||_{h}=1,$ then the problem is
solved by Kropina metric \cite{Y-Sabau}.

\begin{proposition}
\label{Prop3} \cite{SH} Let $(M,F)$ be a Finsler manifold and $W$ a vector
field on $M$ such that $F(x,-W)<1$. Then the solution of the Zermelo
navigation problem with the navigation data $(F,W)$ is a Finsler metric $%
\tilde{F}$ obtained by solving the equation%
\begin{equation}
F(x,y-\tilde{F}(x,y)W)=\tilde{F}(x,y),\text{ }  \label{MAIN}
\end{equation}%
for any $y\in T_{x}M$, $x\in M.$
\end{proposition}

\noindent We notice that Eq. \eqref{MAIN} has a unique positive solution $%
\tilde{F}$ for any $y\in T_{x}M$ because of the strong convexity of the
indicatrix $I_{F}$ and $F(x,-W)<1$ (see \cite{CJS}). Moreover, the sharp
inequality $F(x,-W)<1$ ensures that $\tilde{F}$ is a Finsler metric and
thus, the indicatrix \linebreak $I_{\tilde{F}}=\left\{ (x,y)\in TM|\tilde{F}%
(x,y)=1\right\} $ is strongly convex. In addition, for any regular piecewise 
$C^{\infty }$-curve $\gamma :[0,1]\rightarrow M$ parametrized by time that
represents a trajectory in Zermelo's problem, the $\tilde{F}$-length of $%
\gamma $ is $1,$ i.e. $\tilde{F}(\gamma (t),\dot{\gamma}(t))=1,$ where $\dot{%
\gamma}(t)$ is the velocity vector \cite{chern_shen}.


\section{Proof of Theorem 1.1}
\label{Sec_3}

\label{Sec_3.1}Posing the navigation problem on a slippery slope of a mountain
represented by an $n$-dimensional Riemannian manifold $(M,h)$, $n>1,$ under
the action of the active wind $\mathbf{G}_{\tilde{\eta}}$ given by Eq. 
\eqref{nn} here, supplies the slippery-cross-slope metrics as well as the
necessary and sufficient conditions for their strong convexity. As in \cite%
{slippery,Nicprw,cross}, the gravitational wind $\mathbf{G}^{T}=-\bar{g}%
\omega ^{\sharp }$ turns out to be the main tool in our study, where $\bar{g}
$ is the rescaled magnitude of the acceleration of gravity $g$ (i.e. $\bar{g}%
=\lambda g,$ $\lambda >0)$, and $\omega ^{\sharp }=h^{ji}\frac{\partial p}{%
\partial x^{j}}\frac{\partial }{\partial x^{i}}$ is the gradient vector
field, where $p:M\rightarrow \mathbb{R}$ is a $C^{\infty }$-function on $M$.

Let $u$ be the self-velocity of a moving craft on the slope. We assume $||u||_{h}=1,$ as  standard in most theoretical investigations on the Zermelo navigation
(see, e.g. \cite{BRS}). Taking into account the effect of the active
wind $\mathbf{G}_{\tilde{\eta}}$, the resultant velocity $v_{\tilde{\eta}}=u+%
\mathbf{G}_{\tilde{\eta}}$ allows us to describe the slippery-cross-slope
metrics following the same technique as in \cite{slippery}. Indeed, a
crucial role is played by the active wind, because it can be expressed as $\mathbf{G}_{\tilde{\eta}}=-\tilde{\eta}\mathbf{G}_{MAT}+\mathbf{G}^{T},$ $\tilde{\eta}\in \lbrack 0,1]$, where $\mathbf{G}_{MAT}$ is the orthogonal projection of $\mathbf{G}^{T}$ on $u$. We notice that in this geometric context only the gravitational wind is known, the vector field $\mathbf{G}_{MAT}$ depends on direction of the
self-velocity. In order to carry out the proof of \cref{Theorem1}, we conveniently
divided it into two steps including a sequence of cases and lemmas. The first
step describes deformation of the background Riemannian metric by the vector
field $-\tilde{\eta}\mathbf{G}_{MAT}$, which is a direction-dependent
deformation. The second step is mostly based on the resulting Finsler metric $F$
of Matsumoto type, afforded by the first step. This is deformed by the
gravitational vector field $\mathbf{G}^{T}$, i.e. a rigid translation, under
the condition $F(x,-\mathbf{G}^{T})<1$ which practically ensures that 
a craft on the mountainside can go forward in any direction (for more details, see \cite{SH}).

\paragraph{Step I}

We show that, when $\tilde{\eta}||\mathbf{G}_{MAT}||_{h}<1,$ the
direction-dependent deformation of the background Riemannian metric $h$ by
the vector field $-\tilde{\eta}\mathbf{G}_{MAT}$ defines a Finsler metric
if and only if $||\mathbf{G}^{T}||_{h}<\frac{1}{2\tilde{\eta}},$ for any $%
\tilde{\eta}\in (0,1]$. More precisely, the deformation of the Riemannian
metric $h$ by the vector field $-\tilde{\eta}\mathbf{G}_{MAT}$ invokes the
expression of the resulting velocity, i.e. $v=u-\tilde{\eta}\mathbf{G}_{MAT}$%
, for any $\tilde{\eta}\in (0,1].$ Notice that if $\tilde{\eta}=0,$ this
turns out $v=u$. Otherwise, when $\tilde{\eta}||\mathbf{G}_{MAT}||_{h}\geq 1$
at some direction, this deformation cannot provide a Finsler metric.

Let $\theta $ be the angle between $\mathbf{G}^{T}$ and $u$. Because $%
\mathbf{G}_{MAT}=$Proj$_{u}\mathbf{G}^{T}$, the vectors $v,$ $u$ and $%
\mathbf{G}_{MAT}$ are collinear. Moreover, denoting by $\bar{\theta}$ the
angle between $u$ and $\mathbf{G}_{MAT},$ we see that $\bar{\theta}$ can be $%
0$ or $\pi $, or it is not determined if $\theta $ is $\frac{\pi }{2}$ or $%
\frac{3\pi }{2}$ (i.e. $u$ and $\mathbf{G}^{T}$ are orthogonal and $\mathbf{G%
}_{MAT}=0$).

Taking into account all possibilities for the force of $-\tilde{\eta}\mathbf{%
G}_{MAT}$, we then analyze the following cases: 1. $\tilde{\eta}||\mathbf{G}%
_{MAT}||_{h}<1,$ 2. $\tilde{\eta}||\mathbf{G}_{MAT}||_{h}=1$ and 3. $\tilde{%
\eta}||\mathbf{G}_{MAT}||_{h}>1.$

\noindent \paragraph{Case 1} Assuming that $\tilde{\eta}||\mathbf{G}_{MAT}||_{h}<1$, the angle between $%
\mathbf{G}^{T}$ and $v$ is also $\theta , $ i.e. the vectors $u$ and $v$
point the same direction. Regarding $\bar{\theta},$ we study two subcases:

i) If $\bar{\theta}=0$ (going downhill), then $\theta \in \lbrack 0,\frac{%
\pi }{2})\cup (\frac{3\pi }{2},2\pi )$ and the angle between $\mathbf{G}^{T}$
and $\mathbf{G}_{MAT}$ is either $\theta $ or $2\pi -\theta $. It is clear
that $||\mathbf{G}_{MAT}||_{h}=||\mathbf{G}^{T}||_{h}\cos \theta $ and $h(v,%
\mathbf{G}_{MAT})=||v||_{h}||\mathbf{G}_{MAT}||_{h}=||v||_{h}||\mathbf{G}%
^{T}||_{h}\cos \theta =h(v,\mathbf{G}^{T})$. It also results $\tilde{\eta}||%
\mathbf{G}^{T}||_{h}\cos \theta <1$ and $\frac{h(v,\mathbf{G}^{T})}{||v||_{h}%
}<\frac{1}{\tilde{\eta}},$ $\tilde{\eta}\in (0,1].$

ii) If $\bar{\theta}=\pi $ (going uphill), then $\theta \in (\frac{\pi }{2},%
\frac{3\pi }{2})$ and the angle between $\mathbf{G}^{T}$ and $\mathbf{G}%
_{MAT}$ is $|\pi -\theta| $. Thus, we have $||\mathbf{G}_{MAT}||_{h}=-||%
\mathbf{G}^{T}||_{h}\cos \theta $ and $h(v,\mathbf{G}_{MAT})=-||v||_{h}||%
\mathbf{G}_{MAT}||_{h}=||v||_{h}||\mathbf{G}^{T}||_{h}\cos \theta =h(v,%
\mathbf{G}^{T})$. This implies that $-\tilde{\eta}||\mathbf{G}^{T}||_{h}\cos
\theta <1$ and $-\frac{h(v,\mathbf{G}^{T})}{||v||_{h}}<\frac{1}{\tilde{\eta}}%
,$ $\tilde{\eta}\in (0,1].$

By the above possibilities and noting that $v=u$ if $\theta \in \{\frac{\pi 
}{2},\frac{3\pi }{2}\}$, it turns out that 
\begin{equation}
h(v,\mathbf{G}_{MAT})=h(v,\mathbf{G}^{T})=||v||_{h}||\mathbf{G}%
^{T}||_{h}\cos \theta ,\text{ for any }\theta \in \lbrack 0,2\pi ).
\label{I.1}
\end{equation}%
Moreover, the condition $\tilde{\eta}||\mathbf{G}_{MAT}||_{h}<1$ can be
rewritten as 
\begin{equation}
\tilde{\eta}||\mathbf{G}^{T}||_{h}|\cos \theta |<1\,\ \ \text{\ \ or \ \ }%
\frac{|h(v,\mathbf{G}^{T})|}{||v||_{h}}<\frac{1}{\tilde{\eta}},\text{ }
\label{I.2}
\end{equation}%
for any $\tilde{\eta}\in (0,1]$ and $\theta \in \lbrack 0,2\pi ).$ We note
that, due to the condition $\tilde{\eta}||\mathbf{G}_{MAT}||_{h}<1$, there
is not any direction where the resultant vector $v$ vanishes. Now, making
use of Eqs. \eqref{I.1} and \eqref{I.2}, a straightforward computation
starting with $1=||u||_{h}=||v+\tilde{\eta}\mathbf{G}_{MAT}||_{h}$ leads us
to the equation 
\begin{equation*}
||v||_{h}^{2}+2\tilde{\eta}||v||_{h}||\mathbf{G}^{T}||_{h}\cos \theta -(1-%
\tilde{\eta}^{2}||\mathbf{G}^{T}||_{h}^{2}\cos ^{2}\theta )=0
\end{equation*}%
which admits a sole positive root, that is 
\begin{equation}
||v||_{h}=1-\tilde{\eta}||\mathbf{G}^{T}||_{h}\cos \theta ,\text{ for any }%
\theta \in \lbrack 0,2\pi ).  \label{I.3}
\end{equation}%
\bigskip If we rewrite \eqref{I.3} as $g_{1}(x,v)=0,$ where $%
g_{1}(x,v)=||v||_{h}^{2}-||v||_{h}+\tilde{\eta}h(v,\mathbf{G}^{T})$, then we
obtain the function $F(x,v)=\frac{||v||_{h}^{2}}{||v||_{h}-\tilde{\eta}h(v,%
\mathbf{G}^{T})}$ as the solution of the equation $g_{1}(x,\frac{v}{F})=0,$
by Okubo's method \cite{matsumoto}. The function $F(x,v)$ can be extended to
an arbitrary nonzero vector $y\in T_{x}M,$ for any $x\in M,$ the outcome
being the positive homogeneous $C^{\infty }$-function on $TM_{0}$, namely 
\begin{equation}
F(x,y)=\frac{||y||_{h}^{2}}{||y||_{h}-\tilde{\eta}h(y,\mathbf{G}^{T})}\text{ 
},\text{ \ for any \ }\tilde{\eta}\in (0,1].  \label{eta_mat}
\end{equation}%
This is because any nonzero $y$ can be expressed as $y=cv,$ $c>0,$ and $%
F(x,v)=1$. Actually, we are able to prove that $\tilde{\eta}||\mathbf{G}%
_{MAT}||_{h}<1$ is a neccesary and sufficient condition for $F(x,y)$,
obtained in \eqref{eta_mat}, to be positive on all $TM_{0}.$ Indeed, the
positivity of \eqref{eta_mat} on $TM_{0}$ means that
\begin{equation}
||y||_{h}-\tilde{\eta}h(y,\mathbf{G}^{T})>0,  \label{ineq}
\end{equation}
for all nonzero $y$ and any $\tilde{\eta}\in (0,1].$ If the positivity holds
on $TM_{0},$ we can substitute $y$ with $\mathbf{G}^{T}~\neq~0$ into \eqref{ineq} and thus, $\tilde{\eta}||\mathbf{G}%
^{T}||_{h}<1.$ Since $||\mathbf{G}_{MAT}||_{h}\leq ||\mathbf{G}^{T}||_{h}$
in any direction, ($||\mathbf{G}_{MAT}||_{h}=||\mathbf{G}^{T}||_{h}|\cos
\theta |,$ for any $\theta \in \lbrack 0,2\pi )$), it turns out that $\tilde{%
\eta}||\mathbf{G}_{MAT}||_{h}<1$ on all $TM_{0}.$ Conversly, suppose $\tilde{%
\eta}||\mathbf{G}_{MAT}||_{h}<1$ on all $TM_{0}$ (the case $\mathbf{G}_{MAT}=0$
is also included). Using \eqref{I.2}, it results $\frac{|h(y,\mathbf{G}^{T})|}{%
||y||_{h}}<\frac{1}{\tilde{\eta}}$ for any nonzero $y,$ which assures %
\eqref{ineq}. Thus, $F(x,y)$ is positive on $TM_{0}.$

Making use of the same notations as in \cite{slippery,cross} 
\begin{equation}
\alpha ^{2}=||y||_{h}^{2}=h_{ij}y^{i}y^{j}\text{ \ \ \ \ and \ \ }\beta =-%
\frac{1}{\bar{g}}h(y,\mathbf{G}^{T})=h(y,\omega ^{\sharp })=b_{i}y^{i},
\label{NOT}
\end{equation}%
$\alpha =\alpha (x,y),$ $\beta =\beta (x,y)$ and $||\beta ||_{h}=||\omega
^{\sharp }||_{h},$ the resulting function \eqref{eta_mat} is of Matsumoto
type, i.e. 
\begin{equation}
F(x,y)=\frac{\alpha ^{2}}{\alpha +\tilde{\eta}\bar{g}\beta },\text{ for any
\ }\tilde{\eta}\in (0,1],  \label{Matsumoto_eta}
\end{equation}%
with the corresponding indicatrix%
\begin{equation*}
I_{F}=\left\{ (x,y)\in TM_{0}\text{ }|\text{ }\alpha ^{2}(\alpha +\tilde{\eta%
}\bar{g}\beta )^{-1}=1\right\} .
\end{equation*}%
Moreover, $F(x,y)$ can be extended continuously to all $TM,$ i.e. $F(x,0)=0$
for any $x\in M,$ because $y=0$ does not lie in the closure in $TM$ of the
indicatrix $I_{F}$ \cite{CJS}.

In order to establish the necessary and sufficient conditions for the
function \eqref{Matsumoto_eta} to be a Finsler metric for any $\tilde{\eta}%
\in (0,1]$, let us write $F(x,y)=\alpha \phi (s),$ where $\phi (s)=\frac{1}{%
1+\tilde{\eta}\bar{g}s}\,$\ and $s=\frac{\beta }{\alpha }.$ Since the second
inequality in \eqref{I.2} can be read as $|s|<\frac{1}{\tilde{\eta}\bar{g}}$%
, for arbitrary nonzero $y\in T_{x}M$ and $x\in M$, it turns out that $\phi $
is a positive $C^{\infty }$-function in this case.

Now, we are able to prove some additional properties regarding $\phi $ and
to control the force of the gravitational wind $\mathbf{G}^{T}$ via the
variable $s.$ More precisely, we have

\begin{lemma}
\label{Lema1} For any $\tilde{\eta}\in (0,1]$, fhe following statements are
equivalent:

\begin{itemize}
\item[i)] $\phi (s)-s\phi ^{\prime }(s)+(b^{2}-s^{2})\phi ^{\prime \prime
}(s)>0$, where $b=||\omega ^{\sharp }||_{h};$

\item[ii)] $|s|\leq b<b_{0}$, where $b_{0}=\frac{1}{2\tilde{\eta}\bar{g}}$;

\item[iii)] $||\mathbf{G}^{T}||_{h}<\frac{1}{2\tilde{\eta}}.$
\end{itemize}
\end{lemma}

\begin{proof}
i) $\Leftrightarrow $ ii). The Cauchy-Schwarz inequality $|h(y,\omega
^{\sharp })|\leq ||y||_{h}||\omega ^{\sharp }||_{h}$ turns out $|s|\leq
||\omega ^{\sharp }||_{h}.$ Thus it is clear that $|s|\leq b$ and due to $%
|s|<\frac{1}{\tilde{\eta}\bar{g}}$, we get $(b^{2}-s^{2})\phi ^{\prime
\prime }(s)=(b^{2}-s^{2})\frac{2\tilde{\eta}^{2}\bar{g}^{2}}{(1+\tilde{\eta}%
\bar{g}s)^{3}}\geq 0$. Moreover, the minimum value of $(b^{2}-s^{2})\phi
^{\prime \prime }(s)$ is achieved when $|s|=b.$

If we assume that $\phi (s)-s\phi ^{\prime }(s)+(b^{2}-s^{2})\phi ^{\prime
\prime }(s)>0$, then for $s=-b$ it gives $1-2\tilde{\eta}\bar{g}b>0$, and so 
$b<\frac{1}{2\tilde{\eta}\bar{g}}.$ Therefore, $|s|\leq b<\frac{1}{2\tilde{%
\eta}\bar{g}}.$

Conversely, if $|s|\leq b<\frac{1}{2\tilde{\eta}\bar{g}}$, then%
\begin{equation*}
\phi (s)-s\phi ^{\prime }(s)+(b^{2}-s^{2})\phi ^{\prime \prime }(s)=\frac{(1+%
\tilde{\eta}\bar{g}s)(1+2\tilde{\eta}\bar{g}s)+2\tilde{\eta}^{2}\bar{g}%
^{2}(b^{2}-s^{2})}{(1+\tilde{\eta}\bar{g}s)^{3}}
\end{equation*}%
which implies that $\phi (s)-s\phi ^{\prime }(s)+(b^{2}-s^{2})\phi ^{\prime
\prime }(s)\geq \frac{(1+\tilde{\eta}\bar{g}s)(1+2\tilde{\eta}\bar{g}s)}{(1+%
\tilde{\eta}\bar{g}s)^{3}}=\frac{1+2\tilde{\eta}\bar{g}s}{(1+\tilde{\eta}%
\bar{g}s)^{2}}~>~0.$

ii) $\Leftrightarrow $ iii). If\textbf{\ }$||\mathbf{G}^{T}||_{h}<\frac{1}{2%
\tilde{\eta}}$ and making use of the inequality $|s|\leq ||\omega ^{\sharp
}||_{h}$ and $||\omega ^{\sharp }||_{h}=\frac{1}{\bar{g}}||\mathbf{G}%
^{T}||_{h}$, it results $|s|\leq \frac{1}{\bar{g}}||\mathbf{G}^{T}||_{h}<%
\frac{1}{2\tilde{\eta}\bar{g}}.$ The converse implication is trivial.
\end{proof}

\noindent

We notice that the statement $|s|\leq b<\frac{1}{2\tilde{\eta}\bar{g}},$ for
any $\tilde{\eta}\in (0,1]$ also implies that $\phi (s)-s\phi ^{\prime
}(s)>0.$ Now we are in a position to apply \cite[Lemma 1.1.2]{chern_shen}
and \cref{Prop1}, obtaining the following result.

\begin{lemma}
\label{Lema2} For any $\tilde{\eta}\in (0,1]$, the following statements
are equivalent:

\begin{itemize}
\item[i)] $F(x,y)=\frac{\alpha ^{2}}{\alpha +\tilde{\eta}\bar{g}\beta }$ is
a Finsler metric;

\item[ii)] $||\mathbf{G}^{T}||_{h}<\frac{1}{2\tilde{\eta}}$.
\end{itemize}
\end{lemma}

\noindent In summary, we emphasize that the indicatrix $I_{F}$ is strongly
convex if and only if $||\mathbf{G}^{T}||_{h}<\frac{1}{2\tilde{\eta}}$, for
any $\tilde{\eta}\in (0,1].$

\noindent \paragraph{Case 2} Now we assume that $\tilde{\eta}||\mathbf{G}_{MAT}||_{h}=1.$ We notice that
a traverse of a mountain, i.e. $\theta \in \{\frac{\pi }{2},\frac{3\pi }{2}%
\} $ cannot be followed here because it gives $\tilde{\eta}||\mathbf{G}%
_{MAT}||=0,$ which contradicts our assumption.

Due to the fact that $||\mathbf{G}_{MAT}||_{h}\leq ||\mathbf{G}^{T}||_{h}$,
it turns out that $||\mathbf{G}^{T}||_{h}\geq \frac{1}{\tilde{\eta}}$, for
any $\tilde{\eta}\in (0,1].$ By analyzing two possibilities for $\bar{\theta}%
,$ we have:

i) If $\bar{\theta}=0,$ then $\theta \in \lbrack 0,\frac{\pi }{2})\cup (%
\frac{3\pi }{2},2\pi ).$ Hence $u=\tilde{\eta}\mathbf{G}_{MAT}$, and thus $%
||v||_{h}=0$. This means that the resultant velocity $v$ vanishes, while
attempting to go down on the slope.

ii) If $\bar{\theta}=\pi ,$ then $\theta \in (\frac{\pi }{2},\frac{3\pi }{2}%
) $ and $u=-\tilde{\eta}\mathbf{G}_{MAT}$. It results that $||\mathbf{G}%
_{MAT}||_{h}=-||\mathbf{G}^{T}||_{h}\cos \theta $ and $h(v,\mathbf{G}%
_{MAT})=-||v||_{h}||\mathbf{G}_{MAT}||_{h}=||v||_{h}||\mathbf{G}%
^{T}||_{h}\cos \theta =h(v,\mathbf{G}^{T}).$ Since $v=u-\tilde{\eta}\mathbf{G%
}_{MAT},$ then $v=-2\tilde{\eta}\mathbf{G}_{MAT}$ and thus, $||v||_{h}=2.$ A
few consequences of the last equation can be pointed out. On one hand, the
relation $1=$ $||u||_{h}=||v+\tilde{\eta}\mathbf{G}_{MAT}||_{h}$ leads to $%
\tilde{\eta}h(v,\mathbf{G}^{T})=-2.$ This equation, together with $2||%
\mathbf{G}^{T}||_{h}\cos \theta =h(v,\mathbf{G}^{T})$ and $||\mathbf{G}%
^{T}||_{h}\geq \frac{1}{\tilde{\eta}}$ give us $\cos \theta \leq -1$.
Therefore, the angle $\theta $ can only be $\pi $ and so, $\mathbf{G}_{MAT}=\mathbf{G}^{T}$ and  $\tilde{\eta}||\mathbf{G}^{T}||_{h}=1$, for any $%
\tilde{\eta}\in (0,1]$. On the other hand, we can write $||v||_{h}=2$ as $%
g_{2}(x,v)=0,$ where $g_{2}(x,v)=||v||_{h}-2.$ By applying Okubo's technique 
\cite{matsumoto}, one obtains the function%
\begin{equation}
F(x,y)=\frac{1}{2}||y||_{h},  \label{KROP}
\end{equation}%
which is the solution of the equation $g_{2}(x,\frac{y}{F})=0,$ for any
nonzero tangent vector $y$ collinear with $\mathbf{G}^{T}\neq 0$ such that for each $x\in M,$ $y=-c\mathbf{G}^{T},$ $c>0,$ 
because only the direction corresponding to $\theta =\pi $ can be
followed, for any $\tilde{\eta}\in (0,1]$.  Thus, we define the open conic
subset $\mathcal{A}=\{(x,y)~\in~TM\ |\  y=-c\mathbf{G}^{T},$ $c>0\},$ i.e. 
for each $x\in M,$ $\mathcal{A}_{x}=\mathcal{A\ \cap\ }T_{x}M$ satisfies: if $%
y\in \mathcal{A}_{x},$ then $\lambda y$ $\in \mathcal{A}_{x}$ for every $%
\lambda >0,$ \cite{CJS,JS}. The function \eqref{KROP} can be treated as a
conic Finsler metric (i.e. $F$ satisties the properties of a Finsler metric
only on $\mathcal{A}$, \cite{CJS,JS}) which is homothetic with the background Riemannian metric $h$ on $\mathcal{A}$. Anyway, this case does not povide
a Finsler metric.

\noindent \paragraph{Case 3} Under the requirement $\tilde{\eta}||\mathbf{G}_{MAT}||_{h}>1$ it follows
that the resultant velocity $v$ and $-\tilde{\eta}\mathbf{G}_{MAT}$ point
the same (uphill) direction and also it implies that $||\mathbf{G}^{T}||_{h}>%
\frac{1}{\tilde{\eta}}$. Further on, we split the study into two subcases:

i) If $\bar{\theta}=0$, then it results $\theta \in \lbrack 0,\frac{\pi }{2}%
)\cup (\frac{3\pi }{2},2\pi )$ and $\measuredangle (\mathbf{G}^{T},v)\in
\{\pi +\theta ,\theta -\pi \}$. Consequently, $||\mathbf{G}_{MAT}||_{h}=||%
\mathbf{G}^{T}||_{h}\cos \theta $ and $h(v,\mathbf{G}_{MAT})=-||v||_{h}||%
\mathbf{G}_{MAT}||_{h}=-||v||_{h}||\mathbf{G}^{T}||_{h}\cos \theta =h(v,%
\mathbf{G}^{T}).$ Also, it follows that $\tilde{\eta}||\mathbf{G}%
^{T}||_{h}\cos \theta >1$ and $-\frac{h(v,\mathbf{G}^{T})}{||v||_{h}}>\frac{1%
}{\tilde{\eta}},$ for any $\tilde{\eta}\in (0,1].$

ii) If $\bar{\theta}=\pi ,$ then $\theta \in (\frac{\pi }{2},\frac{3\pi }{2}%
) $ and the angle between $\mathbf{G}^{T}$ and $v$ is also $\theta $. In
consequence, it turns out that

$||\mathbf{G}_{MAT}||_{h}=-||\mathbf{G}^{T}||_{h}\cos \theta $ and $h(v,%
\mathbf{G}_{MAT})=-||v||_{h}||\mathbf{G}_{MAT}||_{h}=||v||_{h}||\mathbf{G}%
^{T}||_{h}\cos \theta =h(v,\mathbf{G}^{T}).$ Moreover, it results $\tilde{%
\eta}||\mathbf{G}^{T}||_{h}\cos \theta <-1$ and $-\frac{h(v,\mathbf{G}^{T})}{%
||v||_{h}}>\frac{1}{\tilde{\eta}},$ for any $\tilde{\eta}\in (0,1].$

We point out that if $\theta \in \{\frac{\pi }{2},\frac{3\pi }{2}\}$, we then get $v=u$. This yields $\tilde{\eta}||G_{MAT}||_{h}=0$, which is contrary
to our assumption. 

To sum up, by both of the above subcases and since $G_{MAT}$ cannot be
vanished, we obtain 
\begin{equation}
h(v,\mathbf{G}_{MAT})=h(v,\mathbf{G}^{T})=-||v||_{h}||\mathbf{G}%
^{T}||_{h}|\cos \theta |,\ \text{for any}\ \theta \in \lbrack 0,2\pi
)\smallsetminus \{\pi /2,3\pi /2\},  \label{II.1}
\end{equation}%
and the condition $\tilde{\eta}||\mathbf{G}_{MAT}||_{h}>1$ is equivalent
to 
\begin{equation}
|\cos \theta |>\frac{1}{\tilde{\eta}||\mathbf{G}^{T}||_{h}}\,\ \ \text{\ \
or \ \ }-\frac{h(v,\mathbf{G}^{T})}{||v||_{h}}>\frac{1}{\tilde{\eta}},\text{ 
}  \label{II.2}
\end{equation}%
for any $\tilde{\eta}\in (0,1].$ Thus, among the directions corresponding
to $\theta \in \lbrack 0,2\pi )\smallsetminus \{\pi /2,3\pi /2\}$ only such directions for which $|\cos \theta |>\frac{1}{\tilde{\eta}%
||\mathbf{G}^{T}||_{h}}$ can be
followed in this case. Also, $\tilde{\eta}||\mathbf{G}%
_{MAT}||_{h}>1$ attests that there is not any direction where the resultant
vector $v$ vanishes. Using the results \eqref{II.1}, \eqref{II.2} and $1=$ $%
||u||_{h}=||v+\tilde{\eta}\mathbf{G}_{MAT}||_{h}$, we arrive at the
equation 
\begin{equation*}
||v||_{h}^{2}-2\tilde{\eta}||v||_{h}||\mathbf{G}^{T}||_{h}|\cos \theta |-(1-%
\tilde{\eta}^{2}||\mathbf{G}^{T}||_{h}^{2}\cos ^{2}\theta )=0.
\end{equation*}%
This admits two positive roots%
\begin{equation}
||v||_{h}=\pm 1+\tilde{\eta}||\mathbf{G}^{T}||_{h}|\cos \theta |
\label{II.3}
\end{equation}%
because of the assumed condition $|\cos \theta |>\frac{1}{\tilde{\eta}||%
\mathbf{G}^{T}||_{h}},$ for any$\ \theta \in \lbrack 0,2\pi )\smallsetminus
\{\pi /2,3\pi /2\}.$ Taking into consideration Eq. \eqref{II.1}, Eq. %
\eqref{II.3} can be rewritten as $g_{3}(x,v)=0,$ where $%
g_{3}(x,v)=||v||_{h}^{2}\mp ||v||_{h}+\tilde{\eta}h(v,\mathbf{G}^{T}).$ By
applying again Okubo's method \cite{matsumoto}, the functions $F_{1,2}(x,v)=%
\frac{||v||_{h}^{2}}{\pm ||v||_{h}-\tilde{\eta}h(v,\mathbf{G}^{T})}$ are the
solutions of the equation $g_{3}(x,\frac{v}{F})=0.$ Moreover, the obtained
positive solutions can be extended to an arbitrary nonzero vector $y\in 
\mathcal{A}_{x}^{\ast }=\mathcal{A}^{\ast }\mathcal{\ \cap\  }T_{x}M,$ for any $x\in M,$ where
\begin{equation}
\mathcal{A}^{\ast }=\{(x,y)\in TM\text{ }|\text{ } \ ||y||_{h}+\tilde{%
\eta}h(y,\mathbf{G}^{T})<0\}
\end{equation}
 is an open conic subset of $TM_{0},$ for any $%
\tilde{\eta}\in (0,1].$ Note that $-\mathbf{G}^{T}\in \mathcal{A}_{x}^{\ast
}$ and $c\mathbf{G}^{T}\notin \mathcal{A}_{x}^{\ast }$, where $c>0.$ It turns out the positive homogeneous functions%
\begin{equation}
F_{1,2}(x,y)=\frac{||y||_{h}^{2}}{\pm ||y||_{h}-\tilde{\eta}h(y,\mathbf{G}%
^{T})}  \label{F12}
\end{equation}%
on $\mathcal{A}^{\ast }$, with $F_{1,2}(x,v)=1$. Based on the notations %
\eqref{NOT}, they are of Matsumoto type, i.e. 
\begin{equation}
F_{1,2}(x,y)=\frac{\alpha ^{2}}{\pm \alpha +\tilde{\eta}\bar{g}\beta }.
\label{Fstrong}
\end{equation}%
However, $F_{1,2}$ can give us at most conic Finsler metrics due to their
conic domain $\mathcal{A}^{\ast }$, rewritten as $\mathcal{A}^{\ast
}=\{(x,y)\in TM$ $|$ $\alpha -\tilde{\eta}\bar{g}\beta <0\}$. Applying  
\cite[Corollary 4.15]{JS}, we obtain that both $F_{1,2}$ are strongly convex
on $\mathcal{A}^{\ast }$ and thus, they are conic Finsler metrics on $%
\mathcal{A}^{\ast },$ for any $\tilde{\eta}\in (0,1].$ Indeed, for $F_{1,2}$
the strongly convex conditions $(\alpha \pm 2\tilde{\eta}\bar{g}\beta
)(\alpha \pm \tilde{\eta}\bar{g}\beta )>0$ are satisfied for any $(x,y)\in 
\mathcal{A}^{\ast }$ and $\tilde{\eta}\in (0,1].$

\bigskip\ Consequently, the direction-dependent deformation of the
background Riemannian metric $h$ by the vector field $-\tilde{\eta}\mathbf{G}%
_{MAT}$, restricted to $\tilde{\eta}||\mathbf{G}_{MAT}||_{h}<1$ for every
direction (which is equivalent to $||\mathbf{G}^{T}||_{h}<\frac{1}{\tilde{%
\eta}})$ provides the Finsler metric of Matsumoto type $F(x,y)=\frac{\alpha
^{2}}{\alpha +\tilde{\eta}\bar{g}\beta }$ if and only if $||\mathbf{G}%
^{T}||_{h}<\frac{1}{2\tilde{\eta}},$ for any $\tilde{\eta}\in (0,1]$.

\paragraph{Step II}

Taking into consideration \cite{CJS}, the second step concerns
the fact that the addition of the gravitational wind $\mathbf{G}^{T}$
only generates a rigid translation to the indicatrix provided by $v=u-\tilde{%
\eta}\mathbf{G}_{MAT}$ in the first step. 
 We can discard the case $\tilde{\eta}||\mathbf{G}_{MAT}||_{h}\geq 1$ because this gave us only conic Finsler metrics and thus, going forward in any direction is not possible. Moreover, the translations of the resulting conic Finsler metrics
(from {\it{Cases 2}} and {\it{3}}) may not exist for any $\tilde{\eta}\in (0,1]$. More
precisely, under the condition $F(x,-\mathbf{G}^{T})<1$, the translation of
the conic Finsler 
metric \eqref{KROP} by $\mathbf{G}^{T}$ only exists for $%
\tilde{\eta}\in (\frac{1}{2},1]$ and it is of Randers type. Similarly, the
translations of the conic Finsler metrics \eqref{Fstrong}, restricted to $%
F_{1,2}(x,-\mathbf{G}^{T})<1,$ cannot exist for any $\tilde{\eta}\in (0,1]$.
Indeed, $F_{1}(x,-\mathbf{G}^{T})<1$ can hold only if $\tilde{\eta}\in (%
\frac{1}{2},1]$ and no $\tilde{\eta}\in (0,1]$ exists such that 
$F_{2}(x,-\mathbf{G}^{T})<1.$

We are now going to apply \cref{Prop3}. More precisely, we deal with
Zermelo's navigation problem on the Finsler manifold $(M,F),$ having the
navigation data $(F,\mathbf{G}^{T}),$ where $F$ is either the Finsler metric %
\eqref{Matsumoto_eta} if $\tilde{\eta}\in (0,1]$ or the background
Riemannian metric $h$ if $\tilde{\eta}=0$, under the requirement 
\begin{equation}
F(x,-\mathbf{G}^{T})<1.  \label{CC}
\end{equation}%
The unique positive solution $\tilde{F}$ of the equation%
\begin{equation}
F(x,y-\tilde{F}(x,y)\mathbf{G}^{T})=\tilde{F}(x,y),\text{ }  \label{II}
\end{equation}%
for any $(x,y)\in TM_{0},$ provides the strongly convex slippery-cross-slope
metric which arises as the solution to the above Zermelo's navigation
problem. In order to achieve this aim, we develop Eq. \eqref{II}, writing
the Finsler metric $F$\ as $F(x,y)=\frac{\alpha ^{2}}{\alpha +\tilde{\eta}%
\bar{g}\beta },$\ for any\textbf{\ }$\tilde{\eta}\in \lbrack 0,1].$ Some
standard computations show clearly that 
\begin{equation*}
\alpha ^{2}\left( x,y-\tilde{F}(x,y)\mathbf{G}^{T}\right) =\alpha ^{2}+2\bar{%
g}\beta \tilde{F}+||\mathbf{G}^{T}||_{h}^{2}\tilde{F}^{2}
\end{equation*}%
and%
\begin{equation*}
\beta \left( x,y-\tilde{F}(x,y)\mathbf{G}^{T}\right) =\beta +\frac{1}{\bar{g}%
}||\mathbf{G}^{T}||_{h}^{2}\tilde{F}
\end{equation*}%
because $\beta (x,\mathbf{G}^{T})=-\frac{1}{\bar{g}}||\mathbf{G}%
^{T}||_{h}^{2}$. Consequently, we have 
\begin{equation*}
F(x,y-\tilde{F}(x,y)\mathbf{G}^{T})=\frac{\alpha ^{2}+2\bar{g}\beta \tilde{F}%
+||\mathbf{G}^{T}||_{h}^{2}\tilde{F}^{2}}{\sqrt{\alpha ^{2}+2\bar{g}\beta 
\tilde{F}+||\mathbf{G}^{T}||_{h}^{2}\tilde{F}^{2}}+\tilde{\eta}\bar{g}\beta +%
\tilde{\eta}||\mathbf{G}^{T}||_{h}^{2}\tilde{F}},
\end{equation*}%
where $\alpha $, $\beta $ and $\tilde{F}$ are evaluated at $(x,y)$.
Therefore, Eq. \eqref{II} leads us to the following irrational equation%
\begin{equation}
\tilde{F}\sqrt{\alpha ^{2}+2\bar{g}\beta \tilde{F}+||\mathbf{G}^{T}||_{h}^{2}%
\tilde{F}^{2}}=\alpha ^{2}+(2-\tilde{\eta})\bar{g}\beta \tilde{F}+(1-\tilde{%
\eta})||\mathbf{G}^{T}||_{h}^{2}\tilde{F}^{2},  \label{MAMA_general}
\end{equation}%
for any\textbf{\ }$\tilde{\eta}\in \lbrack 0,1].$

We can select two special cases from Eq. \eqref{MAMA_general}, considering
all values of the parameter $\tilde{\eta}\in \lbrack 0,1]$. One of them is
for $\tilde{\eta}=0$, when the active wind coincides with the gravitational
wind $\mathbf{G}^{T}.$ Although the direction of $\mathbf{G}^{T}$ is fixed,
i.e. the steepest descent, its force is not reduced now and $\mathbf{G}^{T}$
thus behaves as a standard wind which is included in the navigation data in
the Zermelo problem. So, replacing $\tilde{\eta}=0$ in Eq. %
\eqref{MAMA_general}, it results 
\begin{equation}
\tilde{F}\sqrt{\alpha ^{2}+2\bar{g}\beta \tilde{F}+||\mathbf{G}^{T}||_{h}^{2}%
\tilde{F}^{2}}=\alpha ^{2}+2\bar{g}\beta \tilde{F}+||\mathbf{G}^{T}||_{h}^{2}%
\tilde{F}^{2}.  \label{RANDERS_mama}
\end{equation}%
This can be reduced to 
\begin{equation*}
(1-||\mathbf{G}^{T}||_{h}^{2})\tilde{F}^{2}-2\bar{g}\beta \tilde{F}-\alpha
^{2}=0,
\end{equation*}%
because $\alpha ^{2}+2\bar{g}\beta \tilde{F}+||\mathbf{G}^{T}||_{h}^{2}%
\tilde{F}^{2}>0$. The last equation admits the unique positive root%
\begin{equation*}
\tilde{F}(x,y)=\frac{\sqrt{\alpha ^{2}(1-||\mathbf{G}^{T}||_{h}^{2})+\bar{g}%
^{2}\beta ^{2}}+\bar{g}\beta }{1-||\mathbf{G}^{T}||_{h}^{2}},
\end{equation*}%
under the weak gravitational wind, i.e. $||\mathbf{G}^{T}||_{h}<1.$
Actually, this is the Randers metric $\tilde{F}(x,y)=\tilde{\alpha}+\tilde{%
\beta}$ which solves Zermelo's navigation problem in presence of weak
gravitational wind $\mathbf{G}^{T}$, where 
\begin{equation*}
\tilde{\alpha}^{2}=\frac{\alpha ^{2}(1-||\mathbf{G}^{T}||_{h}^{2})+\bar{g}%
^{2}\beta ^{2}}{(1-||\mathbf{G}^{T}||_{h}^{2})^{2}}\text{ \ and \ \ }\tilde{%
\beta}=\frac{\bar{g}\beta }{1-||\mathbf{G}^{T}||_{h}^{2}}.
\end{equation*}

The second edge case is for $\tilde{\eta}=1$, when the active wind coincides
with the cross wind $\mathbf{G}_{\dag }$ and thus, Eq. \eqref{MAMA_general}
is reduced to%
\begin{equation}
||\mathbf{G}^{T}||_{h}^{2}\tilde{F}^{4}+2\bar{g}\beta \tilde{F}^{3}+(\alpha
^{2}-\bar{g}^{2}\beta ^{2})\tilde{F}^{2}-2\bar{g}\alpha ^{2}\beta \tilde{F}%
-\alpha ^{4}=0.  \label{CROSS}
\end{equation}%
The last equation corresponds to the slope-of-a-mountain problem in a cross
gravitational wind (studied recently in \cite{cross}), which is solved by
the cross-slope metric provided by the unique positive root $\tilde{F}$ of
Eq. \eqref{CROSS}, when \ $||\mathbf{G}^{T}||_{h}<\frac{1}{2}$; cf. \cite[%
Theorem 1.1]{cross}.

We now come back to the general case of Eq. \eqref{MAMA_general}, where $%
\tilde{\eta}\in \lbrack 0,1]$. It is equivalent to the polynomial equation
of degree four, namely 
\begin{equation}
\begin{array}{c}
||\mathbf{G}^{T}||_{h}^{2}[1-\left( 1-\tilde{\eta}\right) ^{2}||\mathbf{G}%
^{T}||_{h}^{2}]\tilde{F}^{4}+2\left[ 1-\left( 2-\tilde{\eta}\right) \left( 1-%
\tilde{\eta}\right) ||\mathbf{G}^{T}||_{h}^{2}\right] \bar{g}\beta \tilde{F}%
^{3} \\ 
~ \\ 
+\{\left[ 1-2\left( 1-\tilde{\eta}\right) ||\mathbf{G}^{T}||_{h}^{2}\right]
\alpha ^{2}-\left( 2-\tilde{\eta}\right) ^{2}\bar{g}^{2}\beta ^{2}\}\tilde{F}%
^{2}-2\left( 2-\tilde{\eta}\right) \bar{g}\alpha ^{2}\beta \tilde{F}-\alpha
^{4}=0,%
\end{array}
\label{MAMA_4}
\end{equation}%
which admits four roots if $1-\left( 1-\tilde{\eta}\right) ^{2}||\mathbf{G}%
^{T}||_{h}^{2}\neq 0$. However, taking into consideration the condition %
\eqref{CC}, (see \cite[p. 10 and Proposition 2.14]{CJS}), it results that
there is a unique positive root, i.e. the slippery-cross-slope metric.
Subsequently, it is denoted by $\tilde{F}_{\tilde{\eta}}$ and it satisfies
Eq. \eqref{MAMA_general}, for each $\tilde{\eta}\in \lbrack 0,1]$. Notice
that along any regular piecewise $C^{\infty }$-curve $\gamma,$ parametrized
by time that represents a trajectory in Zermelo's problem, $\tilde{F}(\gamma
(t),\dot{\gamma}(t))=1,$ i.e. the time in which a craft or a vehicle goes
along it.

Finally, we handle the condition \eqref{CC} which secures that the
indicatrix of $\tilde{F}_{\tilde{\eta}}$ is strongly convex and thus, the
geodesics of $\tilde{F}_{\tilde{\eta}}$ locally minimize time.

\begin{lemma}
\label{Lema3} The following statements are equivalent:

\begin{itemize}
\item[i)] the indicatrix $I_{\tilde{F}_{\tilde{\eta}}}$ of the
slippery-cross-slope metric $\tilde{F}_{\tilde{\eta}}$ is strongly convex;

\item[ii)] the gravitational wind $\mathbf{G}^{T}$ is restricted to either $%
||\mathbf{G}^{T}||_{h}<\frac{1}{1-\tilde{\eta}}$ and $\tilde{\eta}\in
\lbrack 0,\frac{1}{3}]$, or $||\mathbf{G}^{T}||_{h}<\frac{1}{2\tilde{\eta}}$
and $\tilde{\eta}\in (\frac{1}{3},1];$

\item[iii)] the active wind $\mathbf{G}_{\tilde{\eta}}$ given by Eq. %
\eqref{nn} is restricted to either $||\mathbf{G}_{\tilde{\eta}}||_{h}<\frac{1%
}{1-\tilde{\eta}}$ and $\tilde{\eta}\in \lbrack 0,\frac{1}{3}]$, or $||%
\mathbf{G}_{\tilde{\eta}}||_{h}<\frac{1}{2\tilde{\eta}}$ and $\ \tilde{\eta}%
\in (\frac{1}{3},1].$
\end{itemize}
\end{lemma}

\begin{proof}
i) $\Leftrightarrow $ ii) Since $F(x,y)=\frac{\alpha ^{2}}{\alpha +\tilde{%
\eta}\bar{g}\beta }=\frac{||y||_{h}^{2}}{||y||_{h}-\tilde{\eta}h(y,\mathbf{G}%
^{T})},$\ it turns out that, for any\textbf{\ }$\tilde{\eta}\in \lbrack
0,1], $ we have 
\begin{equation*}
F(x,-\mathbf{G}^{T})=\frac{||-\mathbf{G}^{T}||_{h}^{2}}{||-\mathbf{G}%
^{T}||_{h}-\tilde{\eta}h(-\mathbf{G}^{T},\mathbf{G}^{T})}=\frac{||\mathbf{G}%
^{T}||_{h}}{1+\tilde{\eta}||\mathbf{G}^{T}||_{h}}.
\end{equation*}%
Therefore, the requirement \eqref{CC} is equivalent to%
\begin{equation}
(1-\tilde{\eta})||\mathbf{G}^{T}||_{h}<1,  \label{CC_SC}
\end{equation}%
for any $\tilde{\eta}\in \lbrack 0,1]$. On account of \eqref{CC_SC}, we
distinguish three cases:

a) if $\tilde{\eta}=1,$ then the sharp inequality \eqref{CC_SC} holds. Thus,
the strong convexity of the Matsumoto type metric $F$ (Eq. %
\eqref{Matsumoto_eta} with $\tilde{\eta}=1),$ certified by \cref{Lema2},
ensures the strong convexity of the cross-slope metric $\tilde{F}_{\tilde{%
\eta}=1}$, namely $||\mathbf{G}^{T}||_{h}<\frac{1}{2}$ (see also \cite[%
Theorem 1.1]{cross});

b) if $\tilde{\eta}\in (0,1),$ then $||\mathbf{G}^{T}||_{h}<\frac{1}{1-%
\tilde{\eta}}.$ If we combine the last inequality with the strong convexity
condition for the indicatrix $I_{F}$ (more precisely, for any $\tilde{\eta}%
\in (0,1),$ $||\mathbf{G}^{T}||_{h}<\frac{1}{2\tilde{\eta}}$ by \cref{Lema2}%
), we obtain that the indicatrix $I_{\tilde{F}}$ is strongly convex if and
only if either $||\mathbf{G}^{T}||_{h}<\frac{1}{1-\tilde{\eta}} $ and $%
\tilde{\eta}\in (0,\frac{1}{3}]$, or $||\mathbf{G}^{T}||_{h}<\frac{1}{2%
\tilde{\eta}}$ and $\tilde{\eta}\in (\frac{1}{3},1)$;

c) if $\tilde{\eta}=0,$ then $F=h$ and the inequality \eqref{CC_SC} yields $%
||\mathbf{G}^{T}||_{h}<1,$ i.e. the strong convexity condition for the
Randers metric $\tilde{F}_{\tilde{\eta}=0}.$

One can summarize that the sharp inequality \eqref{CC} is equivalent to
either $||\mathbf{G}^{T}||_{h}<\frac{1}{1-\tilde{\eta}}$ and $\tilde{\eta}%
\in \lbrack 0,\frac{1}{3}]$, or $||\mathbf{G}^{T}||_{h}<\frac{1}{2\tilde{\eta%
}}$ and $\tilde{\eta}\in (\frac{1}{3},1]$.

The main key to prove that ii) is equivalent to iii) is the observation that 
$||\mathbf{G}_{\tilde{\eta}}||_{h}\leq ||\mathbf{G}^{T}||_{h},$ for any $%
\tilde{\eta}\in \lbrack 0,1]$ and, furthermore, the maximum of $||\mathbf{G}%
_{\tilde{\eta}}||_{h}$ coincides with $||\mathbf{G}^{T}||_{h}$, since $%
\mathbf{G}_{MAT}$ must vanish for some direction.
\end{proof}

Notice that by \cref{Lema3}, the force of the active wind $\mathbf{G}_{\tilde{\eta}}$ can be described in terms of the gravitational wind $\mathbf{G}^{T}$, i.e. $||\mathbf{G}^{T}||_{h}<\tilde{b}_{0}$, in the slippery-cross-slope problem (\cref{fig_quasi_plane_convexity}), where
\begin{equation}
\ \tilde{b}_{0}=\left\{ 
\begin{array}{cc}
\frac{1}{1-\tilde{\eta}}, & \text{if }\tilde{\eta}\in \lbrack 0,\frac{1}{3}]
\\ 
\frac{1}{2\tilde{\eta}}, & \text{if \ }\tilde{\eta}\in (\frac{1}{3},1]%
\end{array}%
.\right.  \label{Strong_C}
\end{equation}%
Moreover, since $\tilde{b}_{0}\in (1,\frac{3}{2}]$ if $\tilde{\eta}\in (0,%
\frac{1}{2}),$ we outline that the slippery-cross-slope metric $\tilde{F}_{\tilde{\eta}}$ can be still strongly convex even if $||\mathbf{G}^{T}||_{h}>1$, unlike all other navigation problems listed in \Cref{Intro}. Namely, $||\mathbf{G}^{T}||_{h}<\tilde{b}_{0}\leq\frac{3}{2}$ for $\tilde{\eta}\in (0,\frac{1}{2})$ (see \cref{fig_quasi_plane_convexity} as well as \Cref{fig_qs_plane_indix} on further reading).   

The results obtained in steps I and II carry out the proof of \cref{Theorem1}.

\begin{figure}[h!]
\centering
~\includegraphics[width=0.60\textwidth]{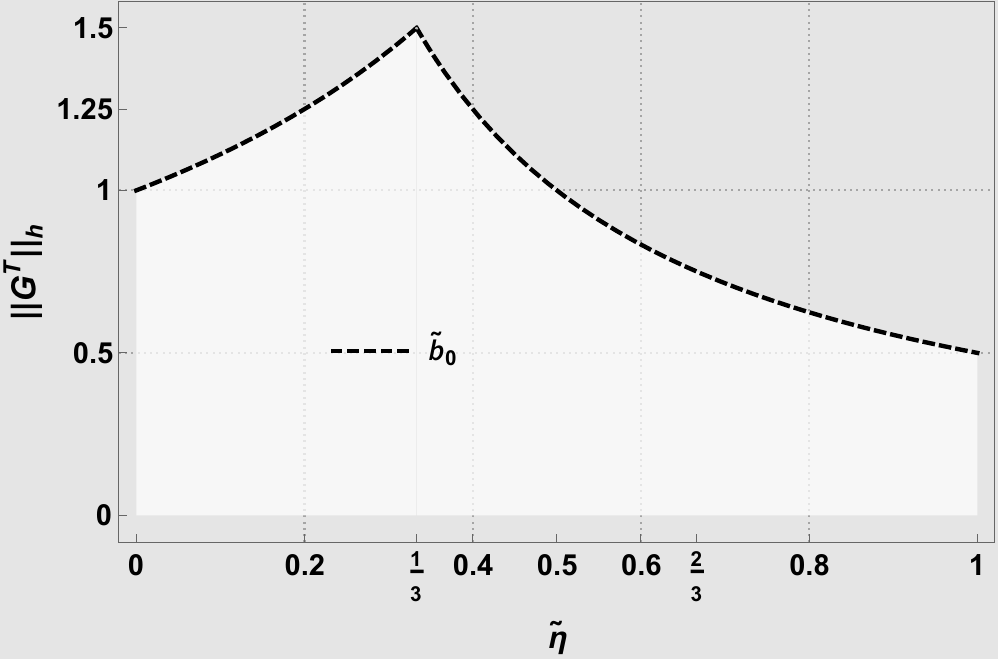} 
\caption{The supremum of the allowable (boundary) force of the gravitational wind due to the convexity restriction, i.e. $||\mathbf{G}^{T}||_{h}<\tilde{b}_{0}$, where $\tilde{b}_{0}$ is given by \eqref{Strong_C}. If $||\mathbf{G}^{T}||_{h}<0.5$, then $\tilde{F}_{\tilde{\eta}}$ is strongly convex for any $\tilde{\eta}\in [0, 1]$.}
\label{fig_quasi_plane_convexity}
\end{figure}

\section{Proof of \cref{Theorem2}}
\label{Sec_4}

The proof of \cref{Theorem2} comprises some technical
computations which aim to reach the spray coefficients related to the
slippery-cross-slope metric $\tilde{F}_{\tilde{\eta}}$. Once this is done, we
can immediately supply the equations of the time geodesics. According to the
previous section the metric $\tilde{F}_{\tilde{\eta}} $ is the unique positive root of Eq. \eqref{MAMA_4}, for each $\tilde{%
\eta}\in \lbrack 0,1]$, and a candidate for a general $(\alpha ,\beta )$%
-metric. Indeed, dividing Eq. \eqref{MAMA_4} by $\alpha ^{4}$ and making use
of the notations $\tilde{\phi}=\frac{\tilde{F}}{\alpha }$ and $s=\frac{\beta 
}{\alpha}$, we obtain an equivalent equation, that is  
\begin{equation}
\begin{array}{c}
||\mathbf{G}^{T}||_{h}^{2}[1-\left( 1-\tilde{\eta}\right) ^{2}||\mathbf{G}%
^{T}||_{h}^{2}]\tilde{\phi}^{4}+2\left[ 1-\left( 2-\tilde{\eta}\right)
\left( 1-\tilde{\eta}\right) ||\mathbf{G}^{T}||_{h}^{2}\right] \bar{g}s%
\tilde{\phi}^{3} \\ 
~ \\ 
+[1-2\left( 1-\tilde{\eta}\right) ||\mathbf{G}^{T}||_{h}^{2}-\left( 2-\tilde{%
\eta}\right) ^{2}\bar{g}^{2}s^{2}]\tilde{\phi}^{2}-2\left( 2-\tilde{\eta}%
\right) \bar{g}s\tilde{\phi}-1=0.%
\end{array}
\label{PHI}
\end{equation}%
This also admits a sole positive root denoted by $\tilde{\phi}_{\tilde{\eta}%
},$ for each $\tilde{\eta}\in \lbrack 0,1]$, which depends on the variables $%
||\mathbf{G}^{T}||_{h}$ and $s,$ i.e. $\tilde{\phi}_{\tilde{\eta}}=\tilde{%
\phi}_{\tilde{\eta}}(||\mathbf{G}^{T}||_{h}^{2},s),$ where $\tilde{\eta}$
plays only as a parameter. Thus, the slippery-cross-slope metric $\tilde{F}_{%
\tilde{\eta}}$ is a general $(\alpha ,\beta )$-metric, $\tilde{F}_{\tilde{%
\eta}}(x,y)=\alpha \tilde{\phi}_{\tilde{\eta}}(||\mathbf{G}%
^{T}||_{h}^{2},s), $ where $\ \bar{g}^{2}b^{2}=||\mathbf{G}^{T}||_{h}^{2}$
and $\tilde{\phi}_{\tilde{\eta}}$ is a positive $C^{\infty }$-function as
well as $\alpha $ and $\beta $ are given by \eqref{NOT}. Due to strong convexity of $\tilde{F}_{\tilde{\eta}}$ established by \cref{Theorem1} (namely, $||\mathbf{G}^{T}||_{h}<\tilde{b}_{0},$ where $\tilde{b}_{0}$ given
by \eqref{Strong_C}), if we apply the direct implication of \cref{Prop1}, we
can state that the function $\tilde{\phi}_{\tilde{\eta}}$ satisfies the
following inequalities%
\begin{equation*}
\tilde{\phi}_{\tilde{\eta}}-s\tilde{\phi}_{\tilde{\eta}2}>0,\qquad \bar{g}%
^{2}(\tilde{\phi}_{\tilde{\eta}}-s\tilde{\phi}_{\tilde{\eta}2})+(||\mathbf{G}%
^{T}||_{h}^{2}-\bar{g}^{2}s^{2})\tilde{\phi}_{\tilde{\eta}22}>0,
\end{equation*}%
when $n\geq 3$ (or only the right-hand side inequality, when $n=2)$, for any $s$ such that $%
|s|\leq \frac{||\mathbf{G}^{T}||_{h}}{\bar{g}}<\frac{\tilde{b}_{0}}{\bar{g}}$%
. Moreover, by applying \cref{Prop2} to the general $(\alpha ,\beta )$%
-metric $\tilde{F}_{\tilde{\eta}}$ we will provide its spray coefficients,
for each $\tilde{\eta}\in \lbrack 0,1].$

We now work toward the establishment of some relations between the function $%
\tilde{\phi}_{\tilde{\eta}}$ and its derivatives.

\begin{lemma}
\label{Lema4} Let $M$ be an $n$-dimensional manifold, $n>1,$ with the
slippery-cross-slope metric $\tilde{F}_{\tilde{\eta}}=\alpha \tilde{\phi}%
_{\tilde{\eta}}(||\mathbf{G}^{T}||_{h}^{2},s).$ The function $\tilde{\phi}_{%
\tilde{\eta}}$ and its derivative with respect to $s,$ i.e. $\tilde{\phi}_{%
\tilde{\eta}2}$ hold the following relations%
\begin{equation}
C\tilde{\phi}_{\tilde{\eta}2}=\bar{g}A\tilde{\phi}_{\tilde{\eta}},\qquad C(%
\tilde{\phi}_{\tilde{\eta}}-s\tilde{\phi}_{\tilde{\eta}2})=B, \qquad 
C\tilde{\phi}_{\tilde{\eta}}=B+\bar{g}sA\tilde{\phi}_{\tilde{\eta}},\qquad
\left( 2-\tilde{\eta}\right) B-2A=\tilde{\eta}\tilde{\phi}_{\eta }^{2},%
\label{3.1}
\end{equation}%
for each $\tilde{\eta}\in \lbrack 0,1],$ where%
\begin{equation}
\begin{array}{l}
A=-\left[ 1-\left( 2-\tilde{\eta}\right) \left( 1-\tilde{\eta}\right) ||%
\mathbf{G}^{T}||_{h}^{2}\right] \tilde{\phi}_{\tilde{\eta}}^{2}+(2-\tilde{%
\eta})^{2}\bar{g}s\tilde{\phi}_{\tilde{\eta}}+2-\tilde{\eta}, \\ 
~ \\ 
B=-\left[1-2(1-\tilde{\eta})||\mathbf{G}^{T}||_{h}^{2}\right]\tilde{\phi}_{\tilde{\eta}%
}^{2}+2(2-\tilde{\eta})\bar{g}s\tilde{\phi}_{\tilde{\eta}}+2, \\ 
~ \\ 
\begin{split}
C=& \ 2||\mathbf{G}^{T}||_{h}^{2}\left[1-\left( 1-\tilde{\eta}\right) ^{2}||%
\mathbf{G}^{T}||_{h}^{2}\right]\tilde{\phi}_{\tilde{\eta}}^{3}+3\left[ 1-\left( 2-%
\tilde{\eta}\right) \left( 1-\tilde{\eta}\right) ||\mathbf{G}^{T}||_{h}^{2}%
\right] \bar{g}s\tilde{\phi}_{\tilde{\eta}}^{2} \\
& +\{\left[ 1-2\left( 1-\tilde{\eta}\right) ||\mathbf{G}^{T}||_{h}^{2}\right]
-\left(2-\tilde{\eta}\right) ^{2}\bar{g}^{2}s^{2}\}\tilde{\phi}_{\tilde{\eta}}-\left( 2-\tilde{\eta}\right) \bar{g}s.
\end{split}%
\end{array}
\label{3.2}
\end{equation}%
Moreover,

\begin{itemize}
\item[i)] $C\neq 0$ for any $\tilde{\eta}\in \lbrack 0,1]$ and $s$ such that 
$|s|\leq \frac{||\mathbf{G}^{T}||_{h}}{\bar{g}}<\frac{\tilde{b}_{0}}{\bar{g}}
$ and,%
\begin{equation}
\tilde{\phi}_{\tilde{\eta}2}=\frac{\bar{g}A}{C}\tilde{\phi}_{\tilde{\eta}%
},\quad \tilde{\phi}_{\tilde{\eta}}-s\tilde{\phi}_{\tilde{\eta}2}=\frac{B}{C}%
,\text{ \ \ }\tilde{\phi}_{\tilde{\eta}22}=\frac{\bar{g}^{2}}{C^{3}}(A^{2}B+%
\tilde{\eta}^{2}\tilde{\phi}_{\tilde{\eta}}^{4}).  \label{33}
\end{equation}

\item[ii)] $B\neq 0$ for any $\tilde{\eta}\in \lbrack 0,1]$ and $s$ such
that $|s|\leq \frac{||\mathbf{G}^{T}||_{h}}{\bar{g}}<\frac{\tilde{b}_{0}}{%
\bar{g}}.$
\end{itemize}
\end{lemma}

\begin{proof}
We take into account the fact that $\tilde{\phi}_{\tilde{\eta}}$ checks
identically Eq. \eqref{PHI}, for any $\tilde{\eta}\in \lbrack 0,1]$ and
thus, we get the following identity%
\begin{equation}
\begin{array}{c}
||\mathbf{G}^{T}||_{h}^{2}[1-\left( 1-\tilde{\eta}\right) ^{2}||\mathbf{G}%
^{T}||_{h}^{2}]\tilde{\phi}_{\tilde{\eta}}^{4}+2\left[ 1-\left( 2-\tilde{\eta%
}\right) \left( 1-\tilde{\eta}\right) ||\mathbf{G}^{T}||_{h}^{2}\right] \bar{%
g}s\tilde{\phi}_{\tilde{\eta}}^{3} \\ 
~ \\ 
+[1-2\left( 1-\tilde{\eta}\right) ||\mathbf{G}^{T}||_{h}^{2}-\left( 2-\tilde{%
\eta}\right) ^{2}\bar{g}^{2}s^{2}]\tilde{\phi}_{\tilde{\eta}}^{2}-2\left( 2-%
\tilde{\eta}\right) \bar{g}s\tilde{\phi}_{\tilde{\eta}}-1=0.%
\end{array}
\label{3.3}
\end{equation}%
The first relation in \eqref{3.1} is obtained by differentiating the identity %
\eqref{3.3} with respect to $s$. Based on this, we then get the second
identity in \eqref{3.1}. By using the notations \eqref{3.2} and Eq. %
\eqref{3.3}, one can prove the last two relations in \eqref{3.1}.

Now in order to prove i) we suppose, towards a contradiction, that there
exists $s_{0}\in \lbrack -b,b],$ $b=\frac{||\mathbf{G}^{T}||_{h}}{\bar{g}}<%
\frac{\tilde{b}_{0}}{\bar{g}}$, with $\tilde{b}_{0}$ given by %
\eqref{Strong_C}, such that $C(||\mathbf{G}^{T}||_{h}^{2},s_{0})=0.$ Under
this assumption, the relations \eqref{3.1} imply that $A(||\mathbf{G}%
^{T}||_{h}^{2},s_{0})=B(||\mathbf{G}^{T}||_{h}^{2},s_{0})=\tilde{\eta}\tilde{%
\phi}_{\tilde{\eta}}^{2}(||\mathbf{G}^{T}||_{h}^{2},s_{0})=0$ and thus,
the sole possibility is $\tilde{\eta}=0,$ because $\tilde{\phi}_{\tilde{\eta}}(||%
\mathbf{G}^{T}||_{h}^{2},s_{0})>0,$ for any $\tilde{\eta}\in \lbrack 0,1].$
By substituting $\tilde{\eta}=0$ together with the above consequences of our
assumption in the identity \eqref{3.3}, it results%
\begin{equation}
||\mathbf{G}^{T}||_{h}^{2}(1-||\mathbf{G}^{T}||_{h}^{2})\tilde{\phi}%
_{0}^{4}+[2\bar{g}s_{0}\tilde{\phi}_{0}+1]^{2}=0.  \label{3.4}
\end{equation}%
This contradicts the fact that $||\mathbf{G}^{T}||_{h}^{2}(1-||\mathbf{G}%
^{T}||_{h}^{2})\tilde{\phi}_{0}^{4}(||\mathbf{G}^{T}||_{h}^{2},s_{0})>0$.
Thus, $C\neq 0$ everywhere here and the claims \eqref{33} are justified.

To show the statement ii), we again argue by contradiction. Assume that
there is $\tilde{s}\in \lbrack -b,b],$ $b=\frac{||\mathbf{G}^{T}||_{h}}{%
\bar{g}}<\frac{\tilde{b}_{0}}{\bar{g}}$, with $\tilde{b}_{0}$ given by %
\eqref{Strong_C}, such that $B(||\mathbf{G}^{T}||_{h}^{2},\tilde{s})=0$.
Thus, we are searching for $\tilde{s}$ in the interval $[-b,b].$ If we take $%
s=\tilde{s}$ in the third formula in \eqref{3.1}, an immediate consequence is $%
\tilde{s}\neq 0$, because of $\tilde{\phi}_{\tilde{\eta}}(||\mathbf{G}%
^{T}||_{h}^{2},\tilde{s})>0,$ $C(||\mathbf{G}^{T}||_{h}^{2},\tilde{s})\neq 0$
and $B(||\mathbf{G}^{T}||_{h}^{2},\tilde{s})=0.$

Moreover, under our assumption, the second formula in \eqref{3.2} turns out
that $\tilde{\phi}_{\tilde{\eta}}(||\mathbf{G}^{T}||_{h}^{2},\tilde{s})$
satisfies the polynomial equation%
\begin{equation}
\lbrack 1-2(1-\tilde{\eta})||\mathbf{G}^{T}||_{h}^{2}]\tilde{\phi}_{\tilde{%
\eta}}^{2}-2(2-\tilde{\eta})\bar{g}\tilde{s}\tilde{\phi}_{\tilde{\eta}}-2=0
\label{SI}
\end{equation}%
and Eq. \eqref{3.3} is reduced to 
\begin{equation}
\begin{array}{c}
2||\mathbf{G}^{T}||_{h}^{2}[1-\left( 1-\tilde{\eta}\right) ^{2}||\mathbf{G}%
^{T}||_{h}^{2}]\tilde{\phi}_{\tilde{\eta}}^{2}+\left[ 2+\tilde{\eta}-2\left(
2-\tilde{\eta}\right) \left( 1-\tilde{\eta}\right) ||\mathbf{G}^{T}||_{h}^{2}%
\right] \bar{g}\tilde{s}\tilde{\phi}_{\tilde{\eta}} \\ 
\\ 
+1-2\left( 1-\tilde{\eta}\right) ||\mathbf{G}^{T}||_{h}^{2}=0,%
\end{array}
\label{SII}
\end{equation}%
for $s=\tilde{s}$ and for any $\tilde{\eta}\in \lbrack 0,1].$ Since $||%
\mathbf{G}^{T}||_{h}<\tilde{b}_{0}$, with $\tilde{b}_{0}$ given by %
\eqref{Strong_C}, $1-\left( 1-\tilde{\eta}\right) ^{2}||\mathbf{G}%
^{T}||_{h}^{2}\neq 0$ for any $\tilde{\eta}\in \lbrack 0,1],$ but may exist some $%
\tilde{\eta}\in \lbrack 0,\frac{1}{2})$ such that $1-2(1-\tilde{\eta})||%
\mathbf{G}^{T}||_{h}^{2}=0.$ Thus, two cases must be distinguished.

a) if $1-2(1-\tilde{\eta})||\mathbf{G}^{T}||_{h}^{2}\neq 0,$ for any $%
\tilde{\eta}\in \lbrack 0,1],$ then by Eqs. \eqref{SI} and \eqref{SII} we obtain%
\begin{equation}
\tilde{\phi}_{\tilde{\eta}}(||\mathbf{G}^{T}||_{h}^{2},\tilde{s})=-\frac{1+4%
\tilde{\eta}||\mathbf{G}^{T}||_{h}^{2}}{\bar{g}\tilde{s}(2+\tilde{\eta}+4%
\tilde{\eta}||\mathbf{G}^{T}||_{h}^{2})}.  \label{SIII}
\end{equation}%
Since $\tilde{\phi}_{\tilde{\eta}}(||\mathbf{G}^{T}||_{h}^{2},\tilde{s})>0,$
then $\tilde{s}\in \lbrack -b,0).$ Having \eqref{SIII} for any $\tilde{\eta}%
\in \lbrack 0,1]$, and by using again Eq. \eqref{SI}, it results $\sqrt{%
\tilde{\eta}}\tilde{s}=-\frac{1+4\tilde{\eta}||\mathbf{G}^{T}||_{h}^{2}}{2%
\bar{g}\sqrt{2+\tilde{\eta}+4\tilde{\eta}||\mathbf{G}^{T}||_{h}^{2}}}$ which
contradicts $\tilde{s}\in \lbrack -b,0),$ if $\tilde{\eta}\neq 0,$ due
to the condition $||\mathbf{G}^{T}||_{h}<\tilde{b}_{0}$, where $\tilde{b}_{0}$ is 
given by \eqref{Strong_C}. Moreover, if $\tilde{\eta}=0$ the last relation
also gives a contradiction.

b) if $1-2(1-\tilde{\eta})||\mathbf{G}^{T}||_{h}^{2}=0$ for some $\tilde{\eta}\in \lbrack 0,\frac{1}{2})$,  then Eq. %
\eqref{SI} leads to 
\begin{equation}
\tilde{\phi}_{\tilde{\eta}}(||\mathbf{G}^{T}||_{h}^{2},\tilde{s})=-\frac{1}{%
(2-\tilde{\eta})\bar{g}\tilde{s}},  \label{SIV}
\end{equation}%
which together with Eq. \eqref{SII} yields $\sqrt{\tilde{\eta}}\tilde{s}=-%
\frac{\sqrt{1+\tilde{\eta}}||\mathbf{G}^{T}||_{h}}{\bar{g}\sqrt{2(2-\tilde{%
\eta})}}$ and $\tilde{s}\in \lbrack -b,0).$ If $\tilde{\eta}\neq 0,$ then $%
\tilde{s}=-\frac{\sqrt{1+\tilde{\eta}}||\mathbf{G}^{T}||_{h}}{\bar{g}\sqrt{2%
\tilde{\eta}(2-\tilde{\eta})}}$ and it contradicts $\tilde{s}\in \lbrack
-b,0).$ Obviously, $\tilde{\eta}=0$ provides another contradiction.

Summing up, we have $B(||\mathbf{G}^{T}||_{h}^{2},s)\neq 0,$ \ for any $s\in
\lbrack -b,b],$ $b=\frac{||\mathbf{G}^{T}||_{h}}{\bar{g}}<\frac{\tilde{b}_{0}%
}{\bar{g}}$, with $\tilde{b}_{0}$ given by \eqref{Strong_C}.

\end{proof}

We notice that according to \cref{Prop1} we knew that $\tilde{\phi}_{\tilde{%
\eta}}-s\tilde{\phi}_{\tilde{\eta}2}>0,$ when $n\geq 3,$ for any $\tilde{\eta%
}\in \lbrack 0,1]$ and $s$ such that $|s|\leq \frac{||\mathbf{G}^{T}||_{h}}{%
\bar{g}}<\frac{\tilde{b}_{0}}{\bar{g}}$. Now by \cref{Lema4}, we have
established that $\tilde{\phi}_{\tilde{\eta}}-s\tilde{\phi}_{\tilde{\eta}%
2}\neq 0$ also when $n=2,$ for any $\ \tilde{\eta}\in \lbrack 0,1]$ and $%
|s|\leq \frac{||\mathbf{G}^{T}||_{h}}{\bar{g}}<\frac{\tilde{b}_{0}}{\bar{g}}%
. $ In addition, the functions $A,$ $B,$ $C$ given in \eqref{3.2} are
homogenous of degree zero with respect to $y$ because of the same
homogeneity degree of $\tilde{\phi}_{\tilde{\eta}},$ for any $\tilde{\eta}%
\in \lbrack 0,1]$.

\begin{lemma}
\label{Lema5} The derivatives of the function $\tilde{\phi}_{\tilde{\eta}}$
with respect to $b^{2}=\frac{||\mathbf{G}^{T}||_{h}^{2}}{\bar{g}^{2}}$ and $%
s,$ i.e. $\tilde{\phi}_{\tilde{\eta}1}$ and $\tilde{\phi}_{\tilde{\eta}12}$, respectively, hold the following relations%
\begin{equation}
\begin{array}{l}
\tilde{\phi}_{\tilde{\eta}1}=\frac{\bar{g}^{2}}{2C}[\left( 1-\tilde{\eta}%
\right) B-\tilde{\eta}\tilde{\phi}_{\tilde{\eta}}^{2}]\tilde{\phi}_{\tilde{%
\eta}}^{2}, \quad 
\tilde{\phi}_{\tilde{\eta}12}=\frac{\bar{g}^{3}}{2C^{3}}\left[ A(B+C\tilde{%
\phi}_{\tilde{\eta}})[\left( 1-\tilde{\eta}\right) B-\tilde{\eta}\tilde{\phi}%
_{\tilde{\eta}}^{2}]+\tilde{\eta}^{2}H\tilde{\phi}_{\tilde{\eta}}^{4}\right] 
\tilde{\phi}_{\tilde{\eta}},%
\end{array}
\label{RRR}
\end{equation}%
for any $\tilde{\eta}\in \lbrack 0,1]$, where $H=2+\bar{g}s\tilde{\phi}_{%
\tilde{\eta}}.$
\end{lemma}

\begin{proof}
Differentiating  the identity \eqref{3.3}, it turns out%
\begin{equation*}
\frac{\partial \tilde{\phi}_{\tilde{\eta}}}{\partial ||\mathbf{G}%
^{T}||_{h}^{2}}=\frac{1}{2C}[(1-\tilde{\eta})B-\tilde{\eta}\tilde{\phi}%
_{\eta }^{2}]\tilde{\phi}_{\eta }^{2}. 
\end{equation*}%
This substituted in $\tilde{\phi}_{\eta 1}=\bar{g}^{2}\frac{\partial \tilde{%
\phi}_{\eta }}{\partial ||\mathbf{G}^{T}||_{h}^{2}}$ gives the left-hand side 
expression in \eqref{RRR}. By \eqref{3.1} and differentiation of the formula \eqref{3.2} w.r.t. $s$, one has the identities 
\begin{eqnarray*}
A_{2} &=&\frac{\bar{g}}{C}[2A^{2}+\tilde{\eta}(2-\tilde{\eta})\tilde{\phi}_{%
\tilde{\eta}}^{2}],\qquad B_{2}=\frac{2\bar{g}}{C}(AB+\tilde{\eta}\tilde{\phi%
}_{\tilde{\eta}}^{2}), \\
C_{2} &=&-\frac{\bar{g}}{C\tilde{\phi}_{\eta }}\{AB+\tilde{\eta}[2(1-\tilde{%
\eta})-(2-\tilde{\eta})H]\tilde{\phi}_{\tilde{\eta}}^{2}\}+3\bar{g}A,
\end{eqnarray*}%
where $A_{2}=\frac{\partial A}{\partial s},$ $B_{2}=\frac{\partial B}{%
\partial s},$ $C_{2}=\frac{\partial B}{\partial s}.$ These, along with%
\begin{equation*}
\tilde{\phi}_{\tilde{\eta}12}=\frac{\bar{g}^{2}}{2C^{2}}\{(1-\tilde{\eta}%
)B_{2}C+2\bar{g}A[\left( 1-\tilde{\eta}\right) B-2\tilde{\eta}\tilde{\phi}_{%
\tilde{\eta}}^{2}]-[\left( 1-\tilde{\eta}\right) B-\tilde{\eta}\tilde{\phi}_{%
\tilde{\eta}}^{2}]C_{2}\}\tilde{\phi}_{\eta }^{2},
\end{equation*}%
lead to the second formula in \eqref{RRR}. 
\end{proof}

Taking into account \cref{Prop2}, \eqref{RRR} and \eqref{3.3},
the only thing left to exploit is the fact that the differential $1$-form $\beta$ is
closed, i.e. $s_{ij}=0$ in the slippery-cross-slope metric $\tilde{F}_{\tilde{\eta}}$, because it includes the gravitational wind $\mathbf{G}^{T}=-\bar{g}\omega ^{\sharp }$, where $\omega ^{\sharp }=h^{ji}\frac{\partial p}{\partial x^{j}}\frac{\partial }{\partial x^{i}}$ is the
gradient vector field. Indeed, it immediately results that $w_{i}=-\bar{g}%
\frac{\partial p}{\partial x^{i}}$ and $\frac{\partial w_{i}}{\partial x^{j}}%
=\frac{\partial w_{j}}{\partial x^{i}},$ where $w_{i}=h_{ij}w^{j}$ and $%
w^{i} $ denote the components of $\mathbf{G}^{T}$ (for more details, see 
\cite{slippery}), as well as the following relations%
\begin{equation}
r_{ij}=-\frac{1}{\bar{g}}w_{i|j},\qquad r_{i}=\frac{1}{\bar{g}^{2}}%
w_{i|j}w^{j},\qquad r^{i}=\frac{1}{\bar{g}^{2}}w_{\text{ }%
|j}^{i}w^{j},\qquad r=-\frac{1}{\bar{g}^{3}}w_{i|j}w^{i}w^{j},  \label{RR}
\end{equation}%
with $w_{i|j}=\frac{\partial w_{i}}{\partial x^{j}}-\Gamma _{ij}^{k}w_{k},$ $%
w_{\text{ }|j}^{i}=h^{ik}w_{k|j}$ and $\Gamma _{ij}^{k}=\frac{1}{2}%
h^{km}\left( \frac{\partial h_{jm}}{\partial x^{i}}+\frac{\partial h_{im}}{%
\partial x^{j}}-\frac{\partial h_{ij}}{\partial x^{m}}\right) $. We notice
that $||\mathbf{G}^{T}||_{h}$ is constant if and only if $r_{i}=0$ and 
moreover under either of the statements of this equivalence, $r^{i}=r=r_{0}=0$ (see \cite{slippery}).

\begin{lemma}
\label{Prop5}Let $M$ be an $n$-dimensional manifold, $n>1,$ with the
slippery-cross-slope metric $\tilde{F}_{\tilde{\eta}},$ for any $\tilde{%
\eta}\in \lbrack 0,1].$ The spray coefficients $\tilde{\mathcal{G}}_{\tilde{\eta}}^{i}$ of $\tilde{F}_{\tilde{\eta}}$ are related to the spray coefficients \linebreak $\mathcal{G}_{\alpha }^{i}=\frac{1}{4}h^{im}\left( 2\frac{\partial h_{jm}}{\partial x^{k}}-\frac{\partial h_{jk}}{\partial x^{m}}\right) y^{j}y^{k}$ of $\alpha $ by 
\begin{equation}
\tilde{\mathcal{G}}_{\tilde{\eta}}^{i}(x,y)=\mathcal{G}_{\alpha }^{i}(x,y)+%
\left[ \mathit{\Theta }(r_{00}+2\alpha ^{2}Rr)+\alpha \mathit{\Omega }r_{0}%
\right] \frac{y^{i}}{\alpha }-\left[ \mathit{\Psi }(r_{00}+2\alpha
^{2}Rr)+\alpha \mathit{\Pi }r_{0}\right] \frac{w^{i}}{\bar{g}}-\alpha
^{2}Rr^{i},  \label{SPRAY}
\end{equation}%
where%
\begin{equation}
\begin{array}{l}
r_{00}=-\frac{1}{\bar{g}}w_{i|j}y^{i}y^{j},\text{ \ \ }r_{0}=\frac{1}{\bar{g}%
^{2}}w_{i|j}w^{j}y^{i},\text{ \ \ }r=-\frac{1}{\bar{g}^{3}}w_{i|j}w^{i}w^{j},%
\text{ \ \ }r^{i}=\frac{1}{\bar{g}^{2}}w_{\text{ }|j}^{i}w^{j}, \\ 
~ \\ 
R=\frac{\bar{g}^{2}}{2\alpha ^{4}B}[\left( 1-\tilde{\eta}\right) \alpha
^{2}B-\tilde{\eta}\tilde{F}_{\tilde{\eta}}^{2}]\tilde{F}_{\tilde{\eta}}^{2},
\\ 
\\ 
\mathit{\Theta }=\frac{\bar{g}\alpha }{2E\tilde{F}_{\tilde{\eta}}}(\alpha
^{6}AB^{2}-\tilde{\eta}^{2}\bar{g}\beta \tilde{F}_{\tilde{\eta}}^{5}),\text{
\ \ }\mathit{\Psi }=\frac{\bar{g}^{2}\alpha ^{2}}{2E}(\alpha ^{4}A^{2}B+%
\tilde{\eta}^{2}\tilde{F}_{\tilde{\eta}}^{4}), \\ 
~ \\ 
\mathit{\Omega }=\frac{\bar{g}^{2}}{\alpha ^{2}BE}\{[\left( 1-\tilde{\eta}%
\right) \alpha ^{2}B-\tilde{\eta}\tilde{F}_{\tilde{\eta}}^{2}](\alpha
^{6}B^{3}+\tilde{\eta}^{2}||\mathbf{G}^{T}||_{h}^{2}\tilde{F}_{\tilde{\eta}%
}^{6})-\tilde{\eta}^{2}\alpha ^{2}\tilde{F}_{\tilde{\eta}}^{5}(\bar{g}\beta
B+||\mathbf{G}^{T}||_{h}^{2}A\tilde{F}_{\tilde{\eta}})\}, \\ 
~ \\ 
\mathit{\Pi }=\frac{\bar{g}^{3}}{2\alpha ^{3}BE}\{[\left( 1-\tilde{\eta}%
\right) \alpha ^{2}B-\tilde{\eta}\tilde{F}_{\tilde{\eta}}^{2}](2\alpha
^{6}AB^{2}-\tilde{\eta}^{2}\bar{g}\beta \tilde{F}_{\tilde{\eta}}^{5})+\tilde{%
\eta}^{2}\alpha ^{2}B\tilde{F}_{\tilde{\eta}}^{4}(2\alpha ^{2}+\bar{g}\beta 
\tilde{F}_{\tilde{\eta}})\}\tilde{F}_{\tilde{\eta}},%
\end{array}
\label{terms}
\end{equation}%
with 
\begin{equation}
\begin{array}{l}
A=-\frac{1}{\alpha ^{2}}\{\left[ 1-\left( 2-\tilde{\eta}\right) \left( 1-%
\tilde{\eta}\right) ||\mathbf{G}^{T}||_{h}^{2}\right] \tilde{F}_{\tilde{\eta}%
}^{2}-(2-\tilde{\eta})^{2}\bar{g}\beta \tilde{F}_{\tilde{\eta}}-(2-\tilde{%
\eta})\alpha ^{2}\}, \\ 
~ \\ 
B=-\frac{1}{\alpha ^{2}}\{[1-2(1-\tilde{\eta})||\mathbf{G}^{T}||_{h}^{2}]%
\tilde{F}_{\tilde{\eta}}^{2}-2(2-\tilde{\eta})\bar{g}\beta \tilde{F}_{\tilde{%
\eta}}-2\alpha ^{2}\}, \\ 
~ \\ 
C=\frac{1}{\alpha \tilde{F}_{\eta }}\left( \alpha ^{2}B+\bar{g}\beta A\tilde{%
F}_{\eta }\right) ,\text{ \ } \\ 
~ \\ 
E=\alpha ^{6}BC^{2}+(||\mathbf{G}^{T}||_{h}^{2}\alpha ^{2}-\bar{g}^{2}\beta
^{2})(\alpha ^{4}A^{2}B+\tilde{\eta}^{2}\tilde{F}_{\tilde{\eta}}^{4}).%
\end{array}
\label{ABC}
\end{equation}
\end{lemma}

\begin{proof}
Once we have the derivatives $\tilde{\phi}_{\eta 1},$ $\tilde{\phi}_{\eta 2},
$ $\tilde{\phi}_{\eta 12}$ and $\tilde{\phi}_{\eta 22}$ by Eqs. \eqref{33} and %
\eqref{RRR}, a straightforward computation leads to the following expressions%
\begin{equation}
\begin{array}{l}
s\tilde{\phi}_{\tilde{\eta}}+(b^{2}-s^{2})\tilde{\phi}_{\tilde{\eta}2}=\frac{%
1}{\bar{g}C}(\bar{g}sB+||\mathbf{G}^{T}||_{h}^{2}A\tilde{\phi}_{\tilde{\eta}%
}), \\ 
~ \\ 
(\tilde{\phi}_{\tilde{\eta}}-s\tilde{\phi}_{\tilde{\eta}2})\tilde{\phi}%
_{\eta 2}-s\tilde{\phi}_{\tilde{\eta}}\tilde{\phi}_{\tilde{\eta}22}=\frac{%
\bar{g}}{C^{3}}(AB^{2}-\tilde{\eta}^{2}\bar{g}s\tilde{\phi}_{\tilde{\eta}%
}^{4}), \\ 
~ \\ 
\tilde{\phi}_{\tilde{\eta}}-s\tilde{\phi}_{\tilde{\eta}2}+(b^{2}-s^{2})%
\tilde{\phi}_{\tilde{\eta}22}=\frac{1}{C^{3}}[BC^{2}+(||\mathbf{G}%
^{T}||_{h}^{2}-\bar{g}^{2}s^{2})(A^{2}B+\tilde{\eta}^{2}\tilde{\phi}_{\eta
}^{4})], \\ 
~ \\ 
(\tilde{\phi}_{\tilde{\eta}}-s\tilde{\phi}_{\tilde{\eta}2})\tilde{\phi}_{%
\tilde{\eta}12}-s\tilde{\phi}_{\tilde{\eta}1}\tilde{\phi}_{\tilde{\eta}22}=%
\frac{\bar{g}^{3}}{2C^{4}}\{[\left( 1-\tilde{\eta}\right) B-\tilde{\eta}%
\tilde{\phi}_{\tilde{\eta}}^{2}](2AB^{2}-\tilde{\eta}^{2}\bar{g}s\tilde{\phi}%
_{\tilde{\eta}}^{5})+\tilde{\eta}^{2}BH\tilde{\phi}_{\tilde{\eta}}^{4}\}%
\tilde{\phi}_{\tilde{\eta}}.%
\end{array}
\label{RRRR}
\end{equation}%
By applying \cref{Prop2} and making use of Eqs. \eqref{RRRR}, we provide our claim. 
\end{proof}

Notice that if $||\mathbf{G}^{T}||_{h}^{2}$
is constant, then the formula \eqref{SPRAY} is reduced to 
\begin{equation}
\tilde{\mathcal{G}}_{\tilde{\eta}}^{i}(x,y)=\mathcal{G}_{\alpha
}^{i}(x,y)+r_{00}\left( \Theta \frac{y^{i}}{\alpha }-\Psi \frac{w^{i}}{\bar{g%
}}\right) ,  \label{S10}
\end{equation}%
because $r^{i}=r=$\ $r_{0}=0$ under this assumption. Therefore, the ODE
system \eqref{GGG} which yields the time-minimizing trajectories $\gamma
(t)=(\gamma ^{i}(t)),$ $i=1,...,n$ on the slippery cross slope under the
influence of the active wind $\mathbf{G}_{\tilde{\eta}}$ is provided by the
system \eqref{geo1}, where the spray coefficients are $\tilde{\mathcal{G}}_{%
\tilde{\eta}}^{i}(\gamma (t),\dot{\gamma}(t))$ from \eqref{SPRAY}, with $%
\tilde{F}_{\tilde{\eta}}(\gamma (t),\dot{\gamma}(t))=1$.

It is worth noting that the presented investigation can be generalized by making the assumption  that the slippery cross slope is non-uniform. This means that the along-traction coefficient $\tilde{\eta}$ could be non-constant, being a function of the position $x\in M$, i.e. $\tilde{\eta}=\tilde{\eta}(x)\in \lbrack 0,1]$. In such improved model the impact of a varying  traction yields a slippery-cross-slope metric which is a little bit more extensive than the general $(\alpha, \beta )$-metrics, namely $\tilde{F}_{\tilde{\eta}}(x,y)=\alpha \tilde{\phi}_{\tilde{\eta}}(||G^{T}||_{h}^{2},s,\tilde{\eta}(x))$. Then $\tilde{\phi}_{\tilde{\tilde{\eta}}}$ would depend in addition on a third variable $\tilde{\eta}(x)$. 
Analogous technical scenario will occur in the extended theory, if we admit the varying 
self-speed $f$ of a craft on the slippery cross slope, i.e. $||u||_{h}=f(x)\in (0,1]$ (see, e.g. \cite{kopi6} in this regard).

\section{Examples}
\label{Sec_5}

In this section, our objective is to apply the general theory developed in this paper by emphasizing a two-dimensional case based on an inclined plane as well as a Gaussian bell-shaped surface. For  comparison and clarity, we continue the line of investigation presented in recent research \cite{slippery, cross}. 

Let $(M, h)$ be a surface embedded in $\mathbb{R}^{3}$, i.e. a $2$%
-dimensional Riemannian manifold, and let $\pi _{O}$ be the tangent plane to 
$M$ at an arbitrary point $O\in M$. We parametrize $M$ by $(x^{1},x^{2})\in M
$ $\rightarrow (x=x^{1},y=x^{2},z=f(x^{1},x^{2}))\in \mathbb{R}^{3},$ where $%
f$ a smooth function on $M$. Thus, the Riemannian metric $h$ induced on $M$
is given by $\left( h_{ij}(x^{1},x^{2})\right) =\left( 
\begin{array}{cc}
1+f_{x^{1}}^{2} & f_{x^{1}}f_{x^{2}} \\ 
f_{x^{1}}f_{x^{2}} & 1+f_{x^{2}}^{2}%
\end{array}%
\right) ,$ $i,j=1,2,$ the functions $f_{x^{1}}$ and $f_{x^{2}}$ denoting the
partial derivatives of $f$ with respect to $x^{1}$ and $x^{2}$,
respectively, and the tangent plane $\pi _{O}$ to $M$ is spanned by the
vectors $\frac{\partial }{\partial x^{1}}=(1,0,f_{x^{1}})$ and $\frac{%
\partial }{\partial x^{2}}=(0,1,f_{x^{2}})$.

If $\mathbf{G}$ denotes a gravitational field in $\mathbb{R}^{3}$ that
affects a mountain slope $M$, it can be decomposed into two orthogonal
components, $\mathbf{G}=\mathbf{G}^{T}+\mathbf{G}^{\perp }$, where $\mathbf{G%
}^{\perp }$ is normal to $M$ in $O$ and $\mathbf{G}^{T}$ is tangent to $M$
in $O$. The latter, acting  along the negative gradient, is called a gravitational wind, and it can be locally read as 
\begin{equation}
\mathbf{G}^{T}=-\frac{\bar{g}}{q+1}(f_{x^{1}},f_{x^{2}},q)=-\frac{\bar{g}}{%
q+1}\left( f_{x^{1}}\frac{\partial }{\partial x^{1}}+f_{x^{2}}\frac{\partial 
}{\partial x^{2}}\right) ,  \label{wind}
\end{equation}%
with $||\mathbf{G}^{T}||_{h}=\bar{g}\sqrt{\frac{q}{q+1}},$ where $%
q=f_{x^{1}}^{2}+f_{x^{2}}^{2}.$ We notice that $\mathbf{G}^{T}=%
0$ in the critical points of $M$ because in such points $%
q=0$.

By constructing a rectangular basis $\{e_{1},e_{2}\}$ in the tangent plane $%
\pi _{O}$, where $e_{1}$ has the same direction as $\mathbf{G}^{T}$, the
parametric equations of the indicatrix of the slippery cross slope metric $%
\tilde{F}_{\tilde{\eta}}$ in the coordinates $(X,Y)$, which correspond to $%
\{e_{1},e_{2}\}$, can be written as 
\begin{equation}
\left\{ 
\begin{array}{rll}
X & = & (1-\tilde{\eta}||\mathbf{G}^{T}||_{h}\cos \theta )\cos \theta +||%
\mathbf{G}^{T}||_{h} \\ 
~ &  &  \\ 
Y & = & (1-\tilde{\eta}||\mathbf{G}^{T}||_{h}\cos \theta )\sin \theta 
\end{array}%
,\right.   \label{2_indicatrix}
\end{equation}%
for any clockwise direction $\theta \in \lbrack 0,2\pi )$ of the
self-velocity $u$, with $||u||_{h}=1,$ and for each along-traction
coefficient $\tilde{\eta}\in \lbrack 0,1].$ An equivalent equation to the system 
\eqref{2_indicatrix} is%
\begin{equation}
\sqrt{\left( X-||\mathbf{G}^{T}||_{h}\right) ^{2}+Y^{2}}=X^{2}+Y^{2}-(2-%
\tilde{\eta})X||\mathbf{G}^{T}||_{h}+(1-\tilde{\eta})||\mathbf{G}%
^{T}||_{h}^{2},  \label{EE}
\end{equation}%
which results from \eqref{2_indicatrix} by eliminating $\theta .$ 

Since any tangent vector of $\pi _{O}$ can be written as $y^{1}\frac{%
\partial }{\partial x^{1}}+y^{2}\frac{\partial }{\partial x^{2}}%
=Xe_{1}+Ye_{2},$ with $e_{1}=-\frac{1}{\sqrt{q(q+1)}}(f_{x^{1}},f_{x^{2}},q)$
and $e_{2}=\frac{1}{\sqrt{q}}(f_{x^{2}},-f_{x^{1}},0),$ we can obtain the
link between the coordinates $(X,Y)$ and $(y^{1},y^{2})$, namely 
\begin{eqnarray*}
X &=&-\sqrt{\frac{q+1}{q}}\left( y^{1}f_{x^{1}}+y^{2}f_{x^{2}}\right) =\frac{%
h(y,\mathbf{G}^{T})}{||\mathbf{G}^{T}||_{h}}=-\frac{\bar{g}\beta }{||\mathbf{%
G}^{T}||_{h}}, \\
Y &=&\frac{1}{\sqrt{q}}\left( y^{1}f_{x^{2}}-y^{2}f_{x^{1}}\right) , \\
X^{2}+Y^{2}
&=&(y^{1})^{2}+(y^{2})^{2}+(y^{1}f_{x^{1}}+y^{2}f_{x^{2}})^2=h_{ij}y^{i}y^{j}=%
\alpha ^{2}.
\end{eqnarray*}%
By the last relations, Eq. \eqref{EE} is equivalent to 
\begin{equation}
\sqrt{\alpha ^{2}+2\bar{g}\beta +||\mathbf{G}^{T}||_{h}^{2}}=\alpha ^{2}+(2-%
\tilde{\eta})\bar{g}\beta +(1-\tilde{\eta})||\mathbf{G}^{T}||_{h}^{2},
\label{III}
\end{equation}%
which yields Eq. \eqref{MAMA_general}, by Okubo's method. Furthermore, in the above study, the slippery-cross-slope metric $\tilde{F}_{\tilde{\eta}}$, which satisfies Eq. \eqref{MAMA_general}, is described only at the regular points $O$ of $M$ ($q(O)\neq 0$), however, it is actually  well defined at the critical points of $M$ as well, where it simplifies to the Riemannian metric $h$.

\subsection{Inclined plane}

In order to compare the current results to the previous investigations we adjust Eqs. \eqref{2_indicatrix} to the planar slope given by $z=x/2$ (i.e. $f(x^{1},x^{2})=x/2$, where $x=x^{1},$ $y=x^{2}$ and $q=1/4)$, taking the regular point $O=(0,0)$. Its slope angle is  $26.6^{\circ}$ and it turns out that $h=\sqrt{h_{ij}y^{i}y^{j}}$, where $\left(h_{ij}(x^{1},x^{2})\right) =\left( 
\begin{array}{cc}
5/4 & 0 \\ 
0 & 1%
\end{array}%
\right), i, j=1,2$ as well as $\mathbf{G}^{T}=-\frac{2\bar{g}}{5}\frac{\partial }{\partial x^{1}}$, and $||\mathbf{G}^{T}||_{h}=\frac{\bar{g}}{\sqrt 5}$. Furthermore, it results $y^{1}=-2X/\sqrt{5}$ and $y^{2}=-Y$.  
\begin{figure}[h!]
\centering
~\includegraphics[width=0.467\textwidth]{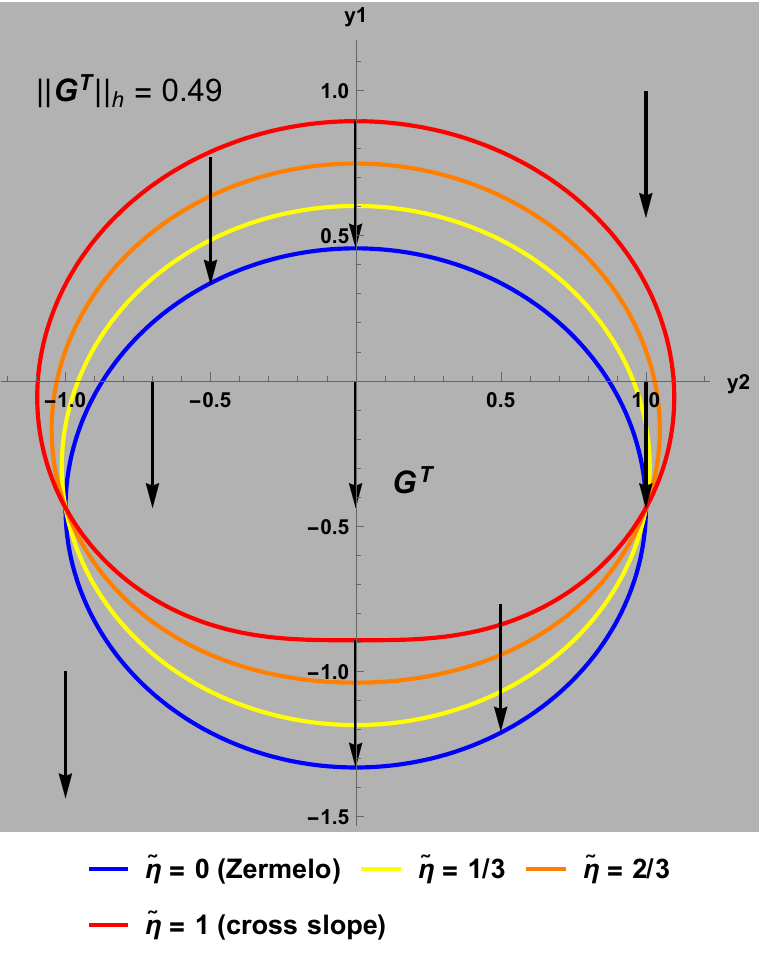} ~
\includegraphics[width=0.499\textwidth]{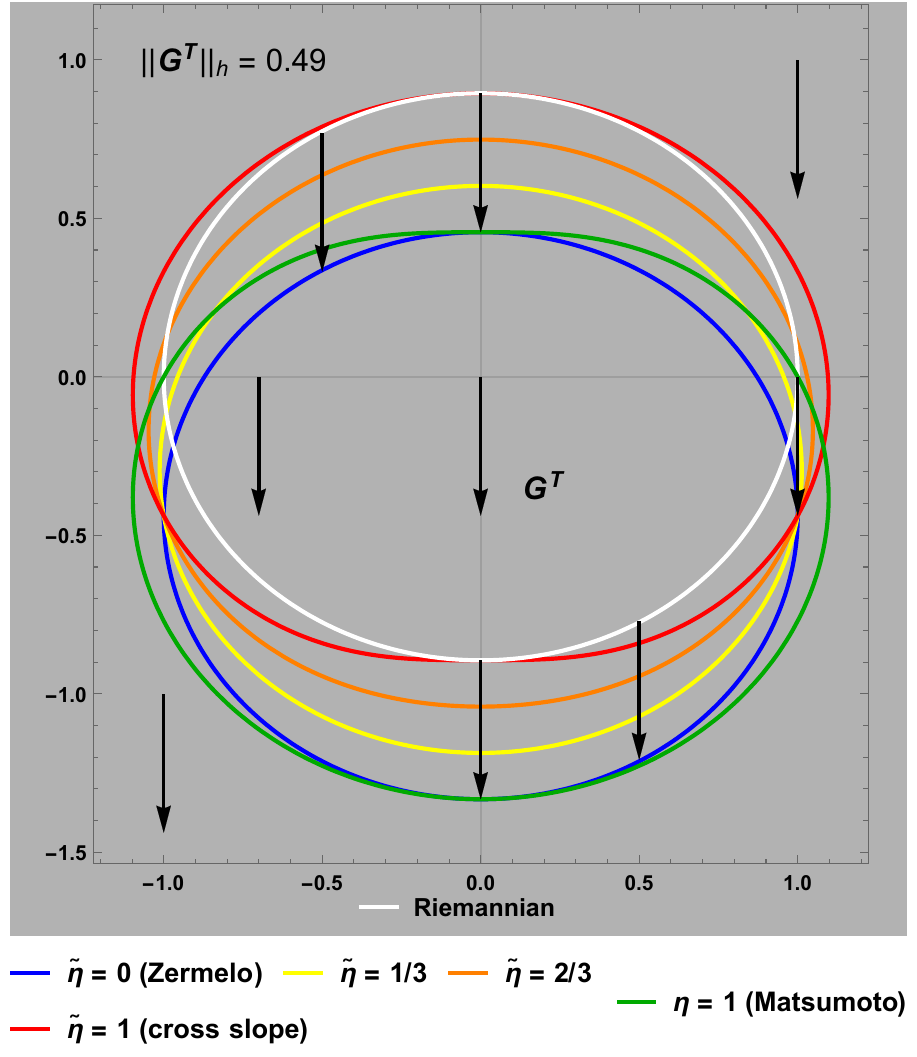} 
\caption{Left: The comparison of the Finslerian indicatrices (lima\c{c}ons) centered at the
origin on the planar slippery cross slope $z=x/2$ (with the slope angle $26.6^{\circ}$)
under the action of the gravitational wind (black arrows) of constant force 
 $||\mathbf{G}^{T}||_{h}=0.49$ in relation to various values of the
along-traction coefficient $\protect\tilde{\eta}$, presented in the coordinate system $y^1Oy^2$; $t=1$. The new slippery-cross-slope indicatrices ($\protect%
\tilde{\eta} \in \{\frac13\ (\textnormal{yellow}), \frac23\ (\textnormal{orange})\}$) are located between the boundary cases, i.e. the Zermelo case ($\protect\tilde{\eta} =0$, blue ellipse) and the cross slope case ($\protect\tilde{\eta} =1$, red). The
steepest downhill direction is indicated by the negative axis $y^1$. All the slippery-cross-slope indicatrices, i.e. for any $\protect\tilde{\eta}\in[0, 1]$ intersect each other in two fixed points, which correspond to $\theta\in\{90^{\circ}, 270^{\circ}\}$ or, equivalently, $\tilde{\theta}\in\{63.9^{\circ}, 296.1^{\circ}\}$, respectively, and $||v_{\tilde{\eta}}||_h\approx1.114$, where $||u||_h=1$. 
Right: As on the left, compared in addition to the Riemannian indicatrix (white) generated by the velocity $v_{\tilde{\eta}}=u$ in this case as well as the Matsumoto indicatrix (green).   
}
\label{fig_qs_plane_indicatrix}
\end{figure}
Basing on the general theory developed in the preceding sections, namely the strong convexity condition $||\mathbf{G}^{T}||_{h}<\tilde{b}_{0},$ where $\tilde{b}_{0}$ given by \eqref{Strong_C}, we get the equivalent  relation for the inclined plane, i.e. $\bar{g}<\delta_{1}(\tilde{\eta})$, where 
\begin{equation}
\delta_{1}(\tilde{\eta})=\left\{ 
\begin{array}{cc}
\frac{\sqrt{5}}{1-\tilde{\eta}}, & \text{for}\quad \tilde{\eta}\in \lbrack 0,\frac{1}{3}]
\\~\\
\frac{\sqrt{5}}{2\tilde{\eta}}, & \text{for}\quad\tilde{\eta}\in (\frac{1}{3},1]%
\end{array}%
.\right.  \label{plane_convexity}
\end{equation}%
One can observe that the condition for strong convexity for the planar slope under consideration leads to 
the restriction $\bar{g}<\sqrt{5}/2$ in the most stringent case, i.e. $\tilde{\eta}=1$ (CROSS)\footnote{Since $||\mathbf{G}^{T}||_{h}=\bar{g}/\sqrt{5}$ for the inclined plane $z=x/2$, while $||\mathbf{G}^{T}||_{h}<0.5$ in the cross slope case ($\tilde{\eta}=1$).}; for clarity, see \Cref{fig_quasi_plane_convexity} and \Cref{fig_quasi_Gauss_convexity} in this regard.  On the other end,  the convexity condition in the Zermelo case reads $\bar{g}<\sqrt{5}$ and $||\mathbf{G}^{T}||_{h}<1$. In what follows, we therefore consider a bit weaker wind force, i.e. $0.49$ so that it is possible to compare all the cases and to make the differences between the indicatrices more visible. 
Also, it is worth mentioning that all time geodesics are the Euclidean straight lines here because the Finsler metric is the same at every point, i.e. it is position-independent.

In \cref{fig_qs_plane_indicatrix}, we compare the new slippery-cross-slope indicatrices which correspond to various values of the along-traction coefficient $\tilde{\eta}$, including the edge cases (ZNP and CROSS), where the gravitational wind force is constant on the planar slope, i.e. $||\mathbf{G}^{T}||_{h}=0.49$. Notice that all the lima\c{c}ons\footnote{The lima\c{c}ons simplify to the ellipses in the cases of Zermelo and Riemannian. We recall that the rigid translation of the Riemannian ellipse by $\mathbf{G}^{T}$ yields the congruent Zermelo (Randers) indicatrix; see, e.g., \cite{chern_shen}.}, i.e. for all $\protect\tilde{\eta}\in[0, 1]$, under the impact of the wind, intersect each other in two fixed points, which correspond to $\theta\in\{90^{\circ}, 270^{\circ}\}$ or, equivalently, $\tilde{\theta}\in\{63.9^{\circ}, 296.1^{\circ}\}$, respectively, and $||v_{\tilde{\eta}}||_h\approx1.114$; cf. \cite[p. 16]{cross}.

\begin{figure}[h!]
\centering
~\includegraphics[width=0.55\textwidth]{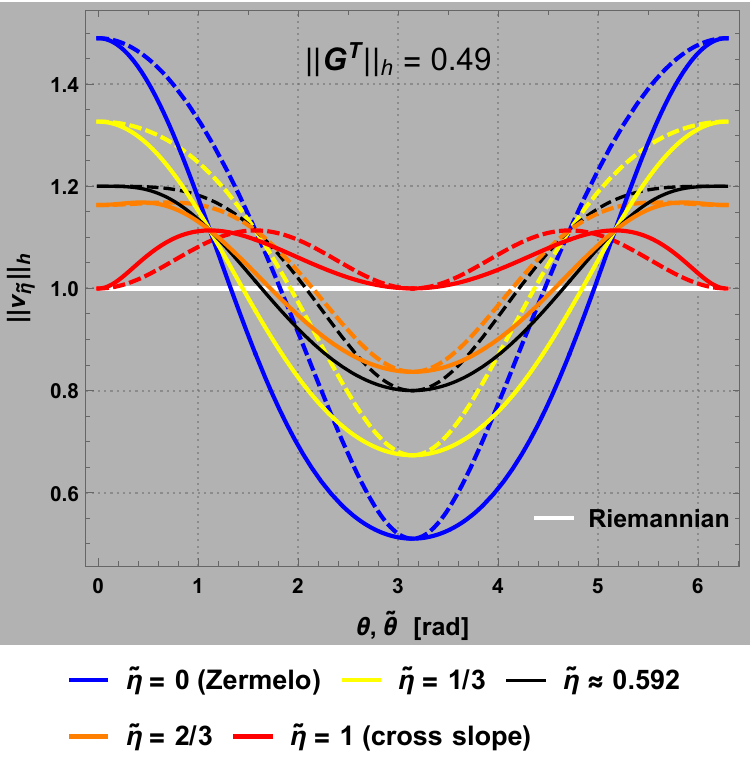} 
\caption{A comparison of the resultant speeds $||v_{\protect\tilde{\eta}}||_{h}$ on the planar
slippery cross slope (given by $z=x/2$) being the functions of direction $\protect\theta$ of the
self-velocity $u$ (dashed colors) and direction $\tilde{\protect\theta}$ of
the resultant velocity $v_{\tilde{\eta}}$ (solid colors), under the action of a constant gravitational wind ($||\mathbf{G}^{T}||_{h}=0.49$), for various values of the along-traction coefficient, i.e. 
$\tilde{\eta}~\in~\{0 \ \textnormal{(ZNP, blue)}, 1/3\  \textnormal{(yellow)},  \sim~0.592\  (\textnormal{black}), 2/3\  \textnormal{(orange)}, 1 \ \textnormal{(CROSS, red)}\}$. The Riemannian case ($v_{\tilde{\eta}}=u$) is marked in white. The angles $\protect\theta ,\tilde{\protect\theta}\in \lbrack 0, 2
\protect\pi)$ are measured clockwise from the steepest downhill
direction $-y^{1}$ (or $OX$).  For all $\protect\tilde{\eta}\in[0, 1]$ the speed curves intersect each other at two fixed points, which correspond to $\theta\in\{90^{\circ}, 270^{\circ}\}$ (dashed) or, equivalently, $\tilde{\theta}\in\{63.9^{\circ}, 296.1^{\circ}\}$ (solid), respectively, and $||v_{\tilde{\eta}}||_h\approx1.114$, where $||u||_h=1$. 
 The extremum speeds (global), i.e. $||v_{\protect\eta}||_{h}\in \{1.49,0.51\}$ are obtained only if  $\tilde{\eta}=0$.}
\label{fig_sx_speeds}
\end{figure}

We analyze next the resultant speeds $||v_{\tilde{\eta}}||_{h}$ related to the above indicatrices on the planar slippery cross slope, presented as the functions of both direction $\theta$ of the velocity $u$ and direction 
$\tilde{\theta}$ of the velocity $v$,   
in the presence of the same gravitational wind $\mathbf{G}^{T}$. The outcome is presented in \cref{fig_sx_speeds}, including in addition the unperturbed Riemannian case, where $v_{\tilde{\eta}}=u$. In particular, observe that for any given along-traction coefficient, the lowest speed is reached when going the steepest uphill direction, i.e. $\theta=\tilde{\theta}=\pi$ (and also downhill in the boundary case $\tilde{\eta}=1$), which is quite natural being in line with the real world situations. However, interestingly, the highest speed is not always obtained on the slippery cross slope when moving exactly in the steepest downhill direction, i.e. $\theta=0$. This is somewhat in contrast with a natural intuition, since a runner can usually expect to reach the highest speed on a hillside when running downhill along the negative gradient path. More precisely, for any  $\tilde{\eta}\in[0, \frac{1}{49} \left(99-13 \sqrt{29}\right)\approx 0.592)$ there exist one global maximum (for $\theta=0$) and one global minimum (for $\theta=\pi$). In turn, for $\tilde{\eta}\in[\sim 0.592, 1)$ there are two minimum values when $\theta=\pi$ (global) or $\theta=0$ (local) as well as two equal maximum values when $\theta=\arccos\{100(\tilde{\eta}-1)/[\tilde{\eta}(49\tilde{\eta}-98)]\}$ or $\theta=2\pi-\arccos\{100(\tilde{\eta}-1)/[\tilde{\eta}(49\tilde{\eta}-98)]\}$ (global). The last formulae come from the elementary differential analysis of the function $||v_{\tilde{\eta}}||_{h}$ based on the system  \eqref{2_indicatrix}; for the sake of clarity, see \cref{fig_sx_max_speeds}. Remark that ${||v_{\tilde{\eta}}||_{h}}_{\max}\approx1.20007$ for the critical case $\tilde{\eta}\approx 0.592$, where the corresponding direction $\theta$ is still zeroed. To complement, in the edge case $\tilde{\eta}=1$ this has a global minimum value of 1 when $\theta\in\{0, \pi\}$ and a global maximum value of about 1.114 when $\theta\in\{\pi/2,  3\pi/2\}$.   

\begin{figure}[H]
\centering
~\includegraphics[width=0.52\textwidth]{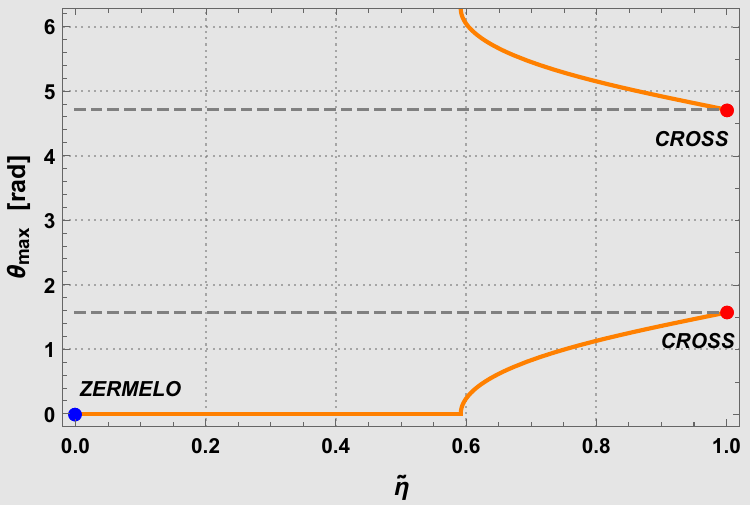} 
\caption{The relation between the along-traction coefficient $\tilde{\eta}$ and the direction $\theta_{\max}$ of the velocity $u$ for which  the corresponding resultant speed $||v_{\tilde{\eta}}||_h$ is maximal on the slippery cross slope, modeled by the inclined plane $z=x/2$, where $||\mathbf{G}^{T}||_{h}=0.49$. 
Note that the highest speed is not always obtained on the slippery cross slope when moving exactly in the steepest downhill direction. For any  $\tilde{\eta}\in[0, \sim 0.592]$ the highest speed $||v_{\tilde{\eta}}||_{h}$ is reached if indeed $\theta=0$, however for $\tilde{\eta}\in(\sim 0.592, 1]$ the maximum is obtained if $\theta=\arccos\{100(\tilde{\eta}-1)/[\tilde{\eta}(49\tilde{\eta}-98)]\}$ or $\theta=2\pi-\arccos\{100(\tilde{\eta}-1)/[\tilde{\eta}(49\tilde{\eta}-98)]\}$. In particular, for $\tilde{\eta}=1$ this is a maximum value of about 1.114 when $\theta\in\{\pi/2,  3\pi/2\}$.  
}
\label{fig_sx_max_speeds}
\end{figure}

Next, we show the deformations of the initial Riemannian indicatrix (white) in relation to various along-traction coefficients $\tilde{\eta}$ and wind forces $||\mathbf{G}^{T}||_{h}$. The outcomes being the convex lima\c{c}ons are compared in \cref{fig_qs_plane_indicatrix2}. In particular, the Zermelo ellipse (dashed and solid blue) can be obtained by a rigid  $\mathbf{G}^{T}$-translation of the Riemannian $h$-circle, which is well known fact in Finsler geometry \cite{chern_shen, BRS}.  
\begin{figure}[h!]
\centering
~\includegraphics[width=0.32\textwidth]{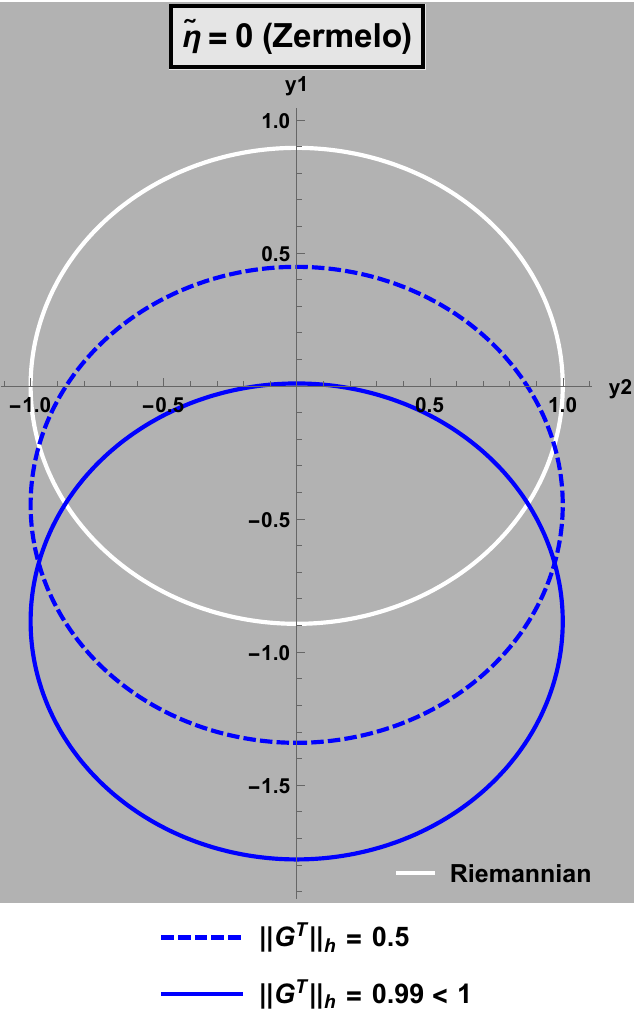} 
~\includegraphics[width=0.32\textwidth]{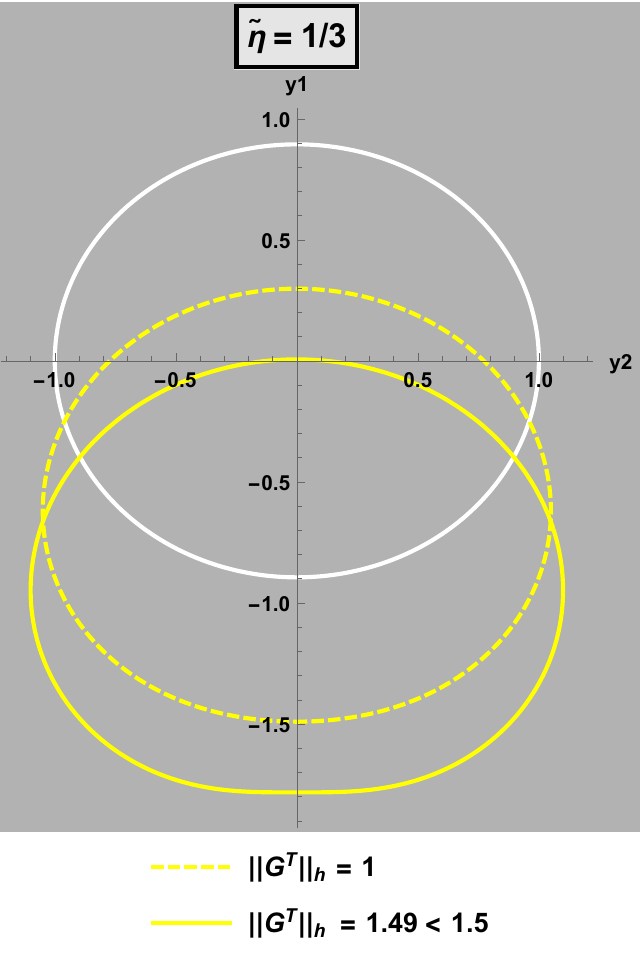} 
~\includegraphics[width=0.32\textwidth]{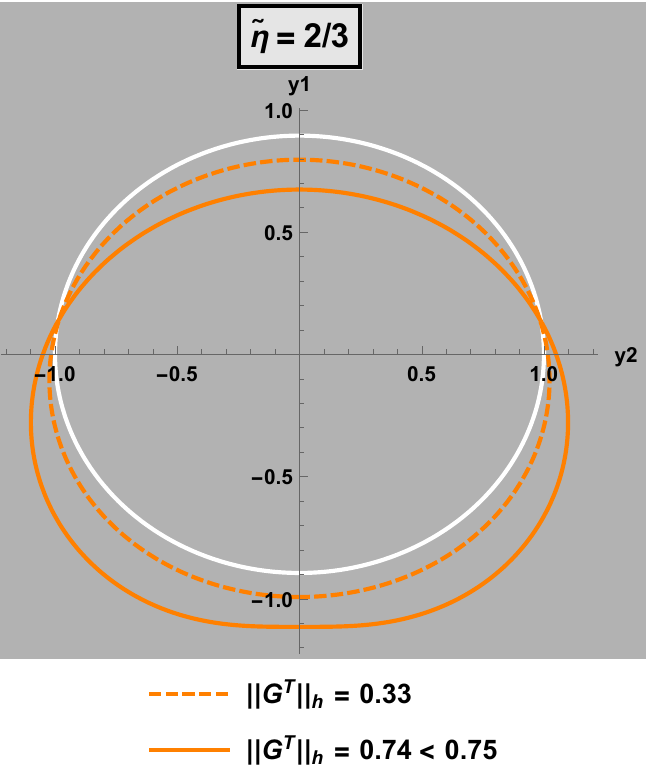} 
\par
\ 
\par
~\includegraphics[width=0.34\textwidth]{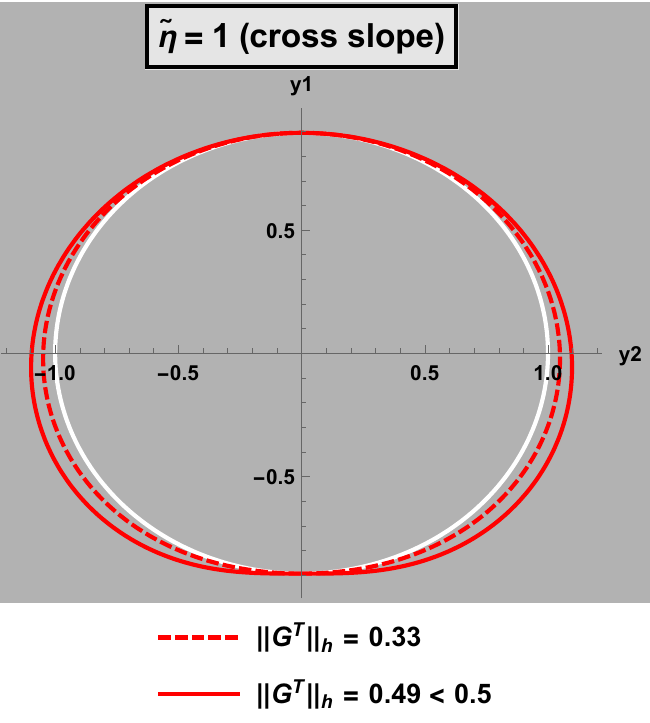} \qquad 
~\includegraphics[width=0.42\textwidth]{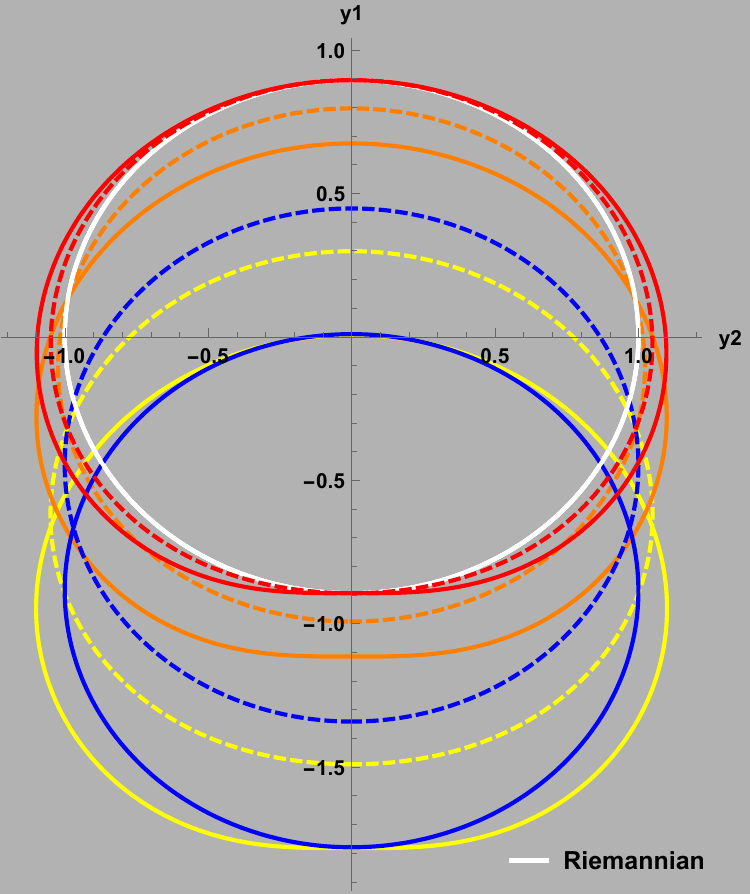} 
\caption{The evolution of the Finslerian indicatrices centered at the origin on the planar slippery cross slope $z=x/2$, under the action of the gravitational wind of various forces and  along-traction coefficients $\protect\tilde{\eta}\in\{0, 1/3, 2/3, 1\}$, presented in the coordinate system $y^1Oy^2$; $t=1$. All the outcomes are compared with each other on the bottom right. The norms $||\mathbf{G}^{T}||_{h}$, which correspond to the solid colors, represent approximately the strongest
allowable (because of the restrictions on convexity) gravitational wind on the slope in each case $\protect\eta$ being considered. The steepest downhill direction on the slope is indicated
by the negative axis $y^1$, and the elliptical Riemannian indicatrix ($h$-circle) is marked in white in each case.}
\label{fig_qs_plane_indicatrix2}
\end{figure}

\ 

Finally, another new feature merits mentioning here. Namely, we have
\begin{remark}
Among the solutions to other problems of optimal navigation (MAT, ZNP, CROSS, SLIPPERY), the upper bound of weak wind force did not exceed 1 due to the restrictions on strong convexity. Namely, $||\mathbf{G}^{T}||_{h}<1=||u||_h$ in the most relaxed case, i.e. the Zermelo navigation. Now, in contrast, exceeding the value 1 (i.e. $||\mathbf{G}^{T}||_{h}<3/2$) can still admit strong convexity of the Finslerian indicatrix (lima\c{c}on) for $\tilde{\eta}\in(0, 1/2)$, containing its center within. Consequently, this also preserves local optimality of time geodesics on the slippery cross slope in the presence of the gravitational wind $\mathbf{G}^{T}$. The related example is presented in \cref{fig_qs_plane_indix}.   
\end{remark}

\begin{figure}[h!]
\centering
~\includegraphics[width=0.42\textwidth]{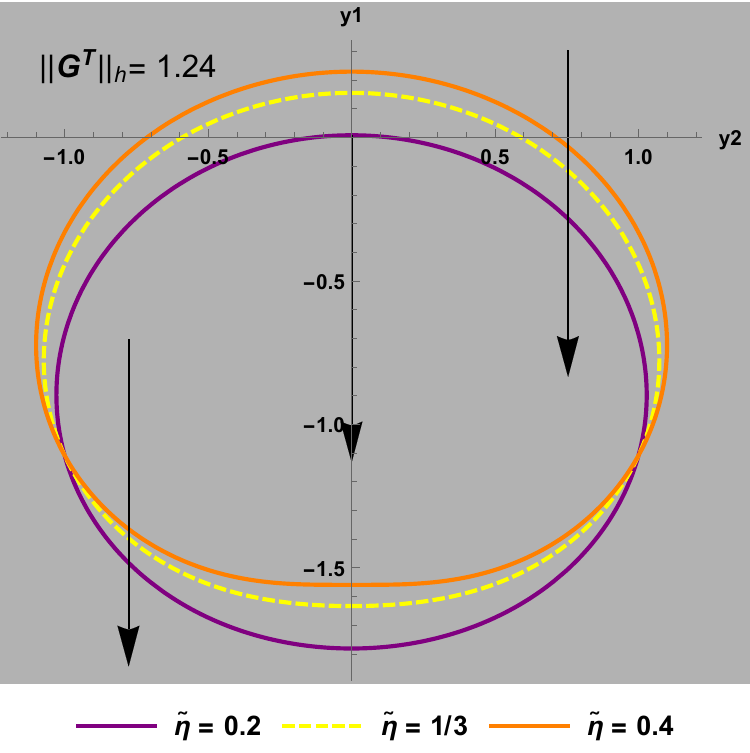} 
~\includegraphics[width=0.42\textwidth]{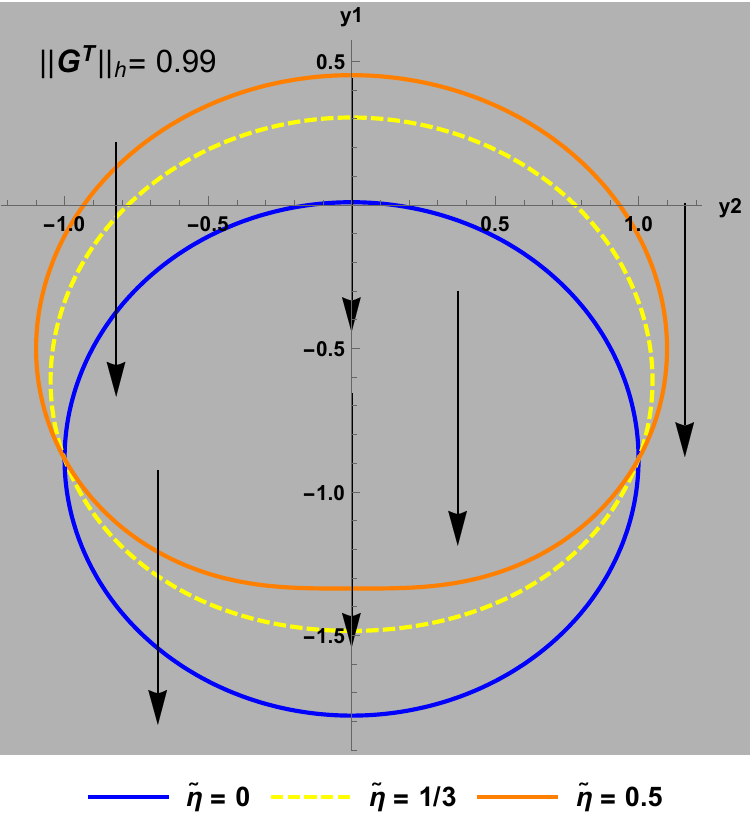}
\caption{Left: the slippery-cross-slope indicatrices (lima\c{c}ons) centered at $(0, 0)$ are still convex, although $||\mathbf{G}^{T}||_{h}=1.24>1=||u||_h$, where $\tilde{\eta}\in\{0.2, 1/3, 0.4\}$, which is in contrast to all other navigation problems aforementioned. Right: the convex slippery-cross-slope indicatrices with $\tilde{\eta}\in\{1/3, 0.5\}$ are compared to the Zermelo ellipse  ($\tilde{\eta}=0$, blue) under about strongest allowable (due to the restriction on convexity) wind in the Zermelo case, i.e. $||\mathbf{G}^{T}||_{h}=0.99<1$. 
}
\label{fig_qs_plane_indix}
\end{figure}

\subsection{Gaussian bell-shaped hillside 
}

We consider a Gaussian bell-shaped hillside $\mathfrak{G}$ given by the
Gaussian function $z=\frac{3}{2}e^{-(x^{2}+y^{2})}$ like in \cite%
{slippery,cross}. It follows from Eq. \eqref{wind} that the gravitational
wind related to $\mathfrak{G}$ is now 
\begin{equation*}
\mathbf{G}^{T}=\frac{3\bar{g}e^{-(x^{2}+y^{2})}}{%
9(x^{2}+y^{2})e^{-2(x^{2}+y^{2})}+1}\left(
x,y,-3(x^{2}+y^{2})e^{-(x^{2}+y^{2})}\right) ,
\end{equation*}%
since $f(x^{1},x^{2})=\frac{3}{2}e^{-(x^{2}+y^{2})}$, where $x=x^{1},$ $%
y=x^{2}$ and $q=9(x^{2}+y^{2})e^{-2(x^{2}+y^{2})}.$ 
For convenience, we apply the following parametrization to $\mathfrak{G}$ 
\begin{equation*}
(\rho ,\varphi )\in \mathfrak{G}\rightarrow (x=\rho \cos \varphi ,\ y=\rho
\sin \varphi ,\ z=3/2e^{-\rho ^{2}})\in \mathbb{R}^{3},
\end{equation*}%
where $\rho \geq 0$ and $\varphi \in \lbrack 0,2\pi )$. It turns out that $%
\mathbf{G}^{T}(\rho ,\varphi )=\frac{3\bar{g}\rho e^{-\rho ^{2}}}{9\rho
^{2}e^{-2\rho ^{2}}+1}\frac{\partial }{\partial \rho }$ and, since $%
(h_{ij}(\rho ,\varphi ))=\left( 
\begin{array}{cc}
9\rho ^{2}e^{-2\rho ^{2}}+1 & 0 \\ 
0 & \rho ^{2}%
\end{array}%
\right) ,$ \ $i,j=1,2,$ the force of the gravitational wind is $||\mathbf{G}%
^{T}||_{h}=\frac{3\bar{g}\rho e^{-\rho ^{2}}}{\sqrt{9\rho ^{2}e^{-2\rho
^{2}}+1}}$ as well as $||\mathbf{G}^{T}||_{h}\in \lbrack 0,\frac{3\bar{g}}{\sqrt{%
2e+9}}],$ for any $\rho \geq 0$. Therefore, the maximum wind force $\frac{3%
\bar{g}}{\sqrt{2e+9}}\approx 0.79\bar{g}$ is obtained for $\rho =\frac{1}{%
\sqrt{2}}\approx 0.71$. The rescaled gravitational acceleration $\bar{g}$
needs handling with greater care to ensure that the geodesics will be indeed
optimal in the sense of time. Taking into consideration \cref{Theorem1}, the
indicatrix of the slippery-cross-slope metric $\tilde{F}_{\tilde{\eta}}$ on
the entire Gaussian bell-shaped hillside $\mathfrak{G}$ is strongly convex
if and only if $\bar{g}<\frac{\sqrt{2e+9}}{6}\approx 0.63$, for any
along-traction coefficient $\tilde{\eta}\in \lbrack 0,1]$. More precisely,
recalling the general condition $||\mathbf{G}^{T}||_{h}<\tilde{b}_{0},$
where $\tilde{b}_{0}$ is given by \eqref{Strong_C}, we have proved the
following result

\begin{lemma}
\label{lemma_gauss} The indicatrix of the slippery-cross-slope metric $\tilde{F}_{\tilde{\eta}}$  is strongly convex 
 on the entire surface $\mathfrak{G}$ if and only if $\bar{g}<\delta _{2}(%
\tilde{\eta}),$ where $\ \delta _{2}(\tilde{\eta})=\left\{ 
\begin{array}{cc}
\frac{\sqrt{2e+9}}{3(1-\tilde{\eta})}, & \text{if}\quad \tilde{\eta}\in
\lbrack 0,\frac{1}{3}] \\ 
~ &  \\ 
\frac{\sqrt{2e+9}}{6\tilde{\eta}}, & \text{if}\quad \tilde{\eta}\in (\frac{1%
}{3},1]%
\end{array}%
.\right. $
\end{lemma}

Next, having in mind the restriction $\bar{g}<\delta _{i}(\tilde{\eta})$, $%
i=1,2,$ we compare the allowable\footnote{%
Because of the restrictions on strong convexity.} values of $\bar{g}$ expressed
as a function of the along-traction coefficient $\tilde{\eta}$ for the
surface $\mathfrak{G}$ and the inclined plane studied in the preceding
subsection. The graphical outcome is presented in \cref{fig_quasi_Gauss_convexity} (left). It is easily seen that the most restrictive case regarding the convexity among all possible scenarios on the slippery cross slope of $\mathfrak{G}$ refers to CROSS ($\tilde{\eta}=1$). However, it is more stringent for the surface  $\mathfrak{G}$ than the inclined plane, since $\sqrt{2e+9}/6\approx 0.63<%
\sqrt{5}/2\approx 1.12$. Moreover, referring to ZNP, the strong convexity
condition on $\mathfrak{G}$ requests $\bar{g}<\sqrt{2e+9}/{3}\approx 1.27$,
while $\bar{g}<\sqrt{5}\approx 2.24$ for the inclined plane; cf. \cite%
{slippery}. Remark that 
the strongest possible (the supremum) 
gravitational wind on the Gaussian bell-shaped hillside, expressed in terms
of the norm $||\mathbf{G}^{T}||_{h}$ \textquotedblleft
blows\textquotedblright\ along the parallel $\rho ={1}/\sqrt{2}\approx 0.71$%
, for any given acceleration $\bar{g}>0$. As expected, the maximum value of $%
\delta _{i}(\tilde{\eta})$, $i=1,2,$ is achived for $\tilde{\eta}=1/3$ on
both types of the slopes. 
 Moreover, we compare the domains of $\bar{g}$ on $\mathfrak{G}$ in the slippery-cross-slope problem being investigated in this paper with the ``common'' slippery slope (including the across-traction coefficient $\eta$) analyzed in \cite{slippery} recently. This is illustrated in \cref{fig_quasi_Gauss_convexity} (right). 
\begin{figure}[h]
\centering
~\includegraphics[width=0.48\textwidth]{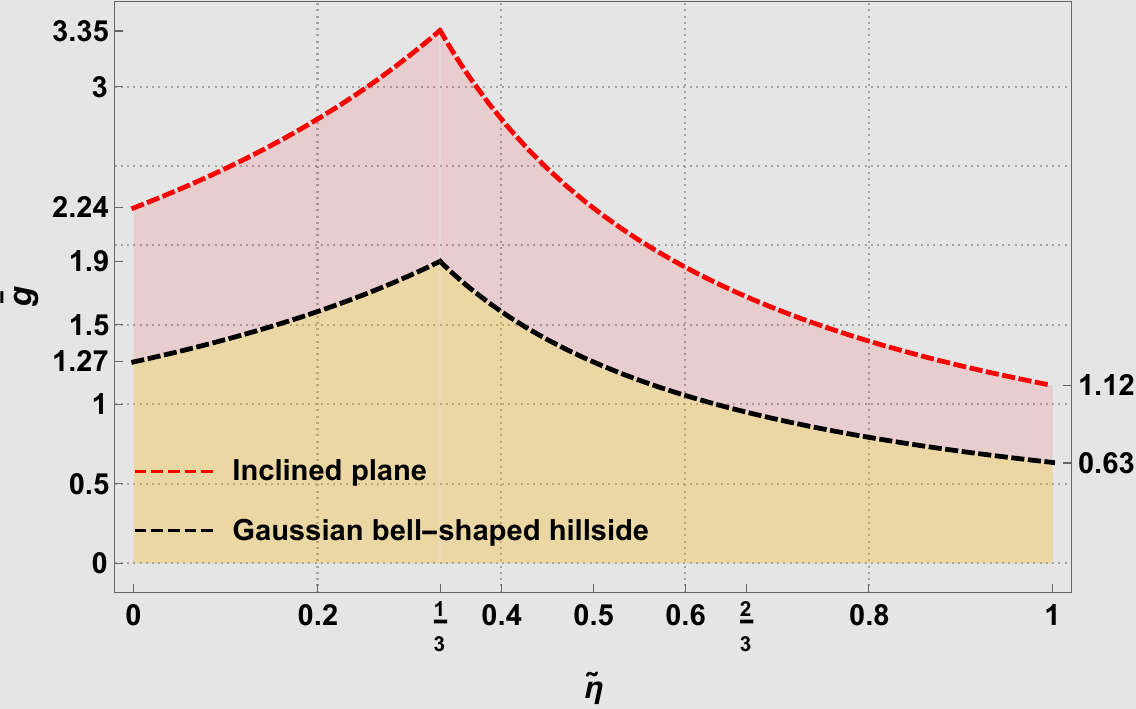} 
~\includegraphics[width=0.48\textwidth]{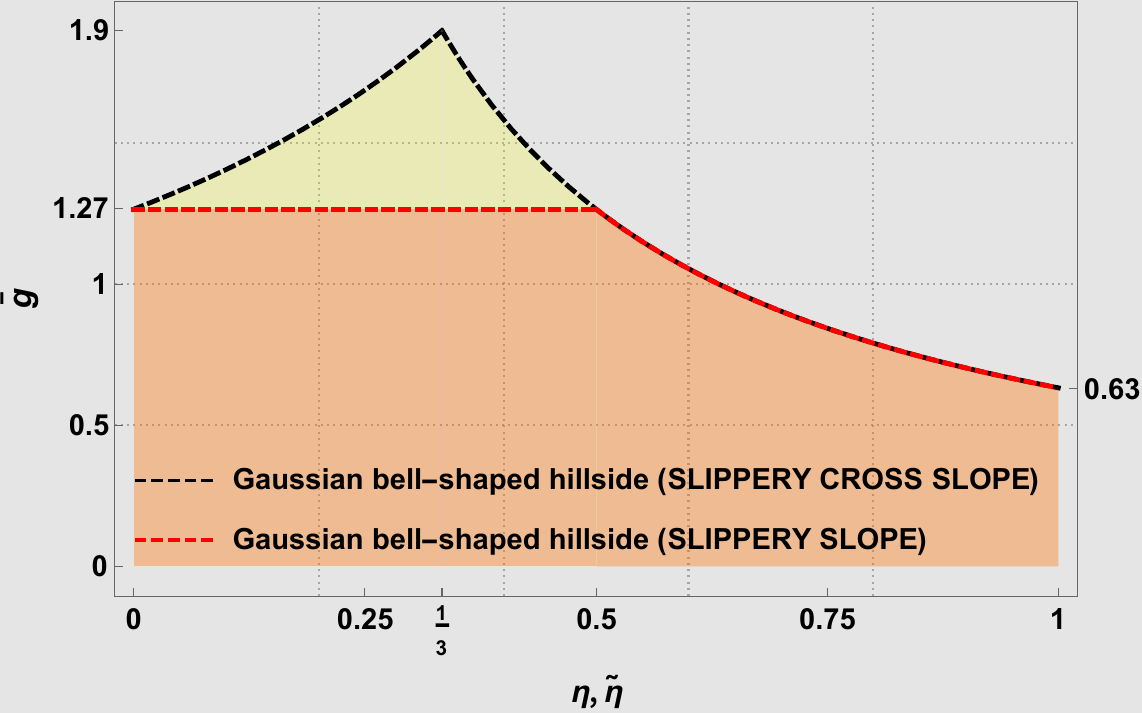}
\caption{Left: the allowable 
 rescaled gravitational acceleration $\bar{g}$ as a function of the along-traction coefficient $\tilde{\eta}$ in the slippery-cross-slope problem, considered on the Gaussian bell-shaped hillside $\mathfrak{G}$ (as per \cref{lemma_gauss}) 
 and the inclined plane $z=x/2$.  
 Right: like on the left, but the allowable rescaled gravitational acceleration $\bar{g}$ refers to the same Gaussian bell-shaped hillside $\mathfrak{G}$ in slippery-cross-slope problem versus the slippery slope problem (with across-traction represented by a parameter $\eta\in[0, 1]$ \cite{slippery}).   
}
\label{fig_quasi_Gauss_convexity}
\end{figure}

Further on, we show the $\tilde{F}_{\tilde{\eta}}$-geodesic equations, which
are related to $\mathfrak{G}$. We thus get \cite{slippery, cross} 
\begin{equation}
\alpha ^{2}=(1+9\rho ^{2}e^{-2\rho ^{2}})\dot{\rho}^{2}+\rho ^{2}\dot{\varphi%
}^{2},\qquad \beta =-3\rho e^{-\rho ^{2}}\dot{\rho},  \label{4.4}
\end{equation}%
\begin{equation}
\mathcal{G}_{\alpha }^{1}=\frac{\rho }{2(9\rho ^{2}e^{-2\rho ^{2}}+1)}\left[
9(1-2\rho ^{2})e^{-2\rho ^{2}}\dot{\rho}^{2}-\dot{\varphi}^{2}\right]
,\qquad \mathcal{G}_{\alpha }^{2}=\frac{1}{\rho }\dot{\rho}\dot{\varphi},
\label{4.5}
\end{equation}%
\begin{equation}
\begin{array}{l}
r_{00}=-\frac{3e^{-\rho ^{2}}}{9\rho ^{2}e^{-2\rho ^{2}}+1}\left[ (1-2\rho
^{2})\dot{\rho}^{2}+\rho ^{2}\dot{\varphi}^{2}\right] ,\qquad r_{0}=\frac{%
9\rho (1-2\rho ^{2})e^{-2\rho ^{2}}}{(9\rho ^{2}e^{-2\rho ^{2}}+1)^{2}}\dot{%
\rho}, \\ 
~ \\ 
r=-\frac{27\rho ^{2}(1-2\rho ^{2})e^{-3\rho ^{2}}}{(9\rho ^{2}e^{-2\rho
^{2}}+1)^{3}},\qquad r^{1}=\frac{9\rho (1-2\rho ^{2})e^{-2\rho ^{2}}}{(9\rho
^{2}e^{-2\rho ^{2}}+1)^{3}},\qquad r^{2}=0.%
\end{array}
\label{4.6}
\end{equation}

Owing to \cref{Theorem2} and \cref{Prop5}, the time geodesics $\gamma
(t)=(\rho (t),\varphi (t))$ on the slippery cross slope of the surface $%
\mathfrak{G}$ are are provided by the solutions of the ODE system 
\begin{equation*}
\left\{ 
\begin{array}{rll}
0 & = & \ddot{\rho}+\frac{\rho }{9\rho ^{2}e^{-2\rho ^{2}}+1}\left[
9(1-2\rho ^{2})e^{-2\rho ^{2}}\dot{\rho}^{2}-\dot{\varphi}^{2}\right]
+2\left\{ \mathit{\tilde{\Theta}}(r_{00}+2\alpha ^{2}\tilde{R}r)+\alpha 
\mathit{\tilde{\Omega}}r_{0}\right\} \frac{\dot{\rho}}{\alpha } \\ 
~ &  &  \\ 
&  & -\frac{6\rho e^{-\rho ^{2}}}{9\rho ^{2}e^{-2\rho ^{2}}+1}\left\{ 
\mathit{\ \tilde{\Psi}}(r_{00}+2\alpha ^{2}\tilde{R}r)+\alpha \mathit{\tilde{%
\Pi}}r_{0}\right\} -\frac{18\rho (1-2\rho ^{2})e^{-2\rho ^{2}}}{(9\rho
^{2}e^{-2\rho ^{2}}+1)^{3}}\alpha ^{2}\tilde{R} \\ 
~ &  &  \\ 
0 & = & \ddot{\varphi}+\frac{2}{\rho }\dot{\rho}\dot{\varphi}+2\left\{ 
\mathit{\tilde{\Theta}}(r_{00}+2\alpha ^{2}\tilde{R}r)+\alpha \mathit{\tilde{%
\Omega}}r_{0}\right\} \frac{\dot{\varphi}}{\alpha }%
\end{array}%
\right. ,
\end{equation*}%
where $\mathit{\tilde{\Theta}},$ $\tilde{R},$ $\mathit{\tilde{\Omega}},$ $%
\mathit{\tilde{\Pi}}$ and $\mathit{\tilde{\Psi}}$ are given by Eq. %
\eqref{geo_tilde}, with $\mathbf{G}^{T}(\rho ,\varphi )=\frac{3\bar{g}\rho
e^{-\rho ^{2}}}{9\rho ^{2}e^{-2\rho ^{2}}+1}\frac{\partial }{\partial \rho },
$  $\bar{g}<\delta _{2}(\tilde{\eta})$, together with Eqs. %
\eqref{4.4} and \eqref{4.6}, $\rho =\rho (t)$, $\varphi =\varphi (t)$. %
\cref{fig_qs_Gauss} shows the slippery-cross-slope geodesics generated for
various along-traction coefficients, i.e. $\tilde{\eta}\in \{0,1/3,2/3,1\}$.
Also, the corresponding unit time fronts are presented in solid colors. 
\begin{figure}[H]
\centering
~\includegraphics[width=0.52\textwidth]{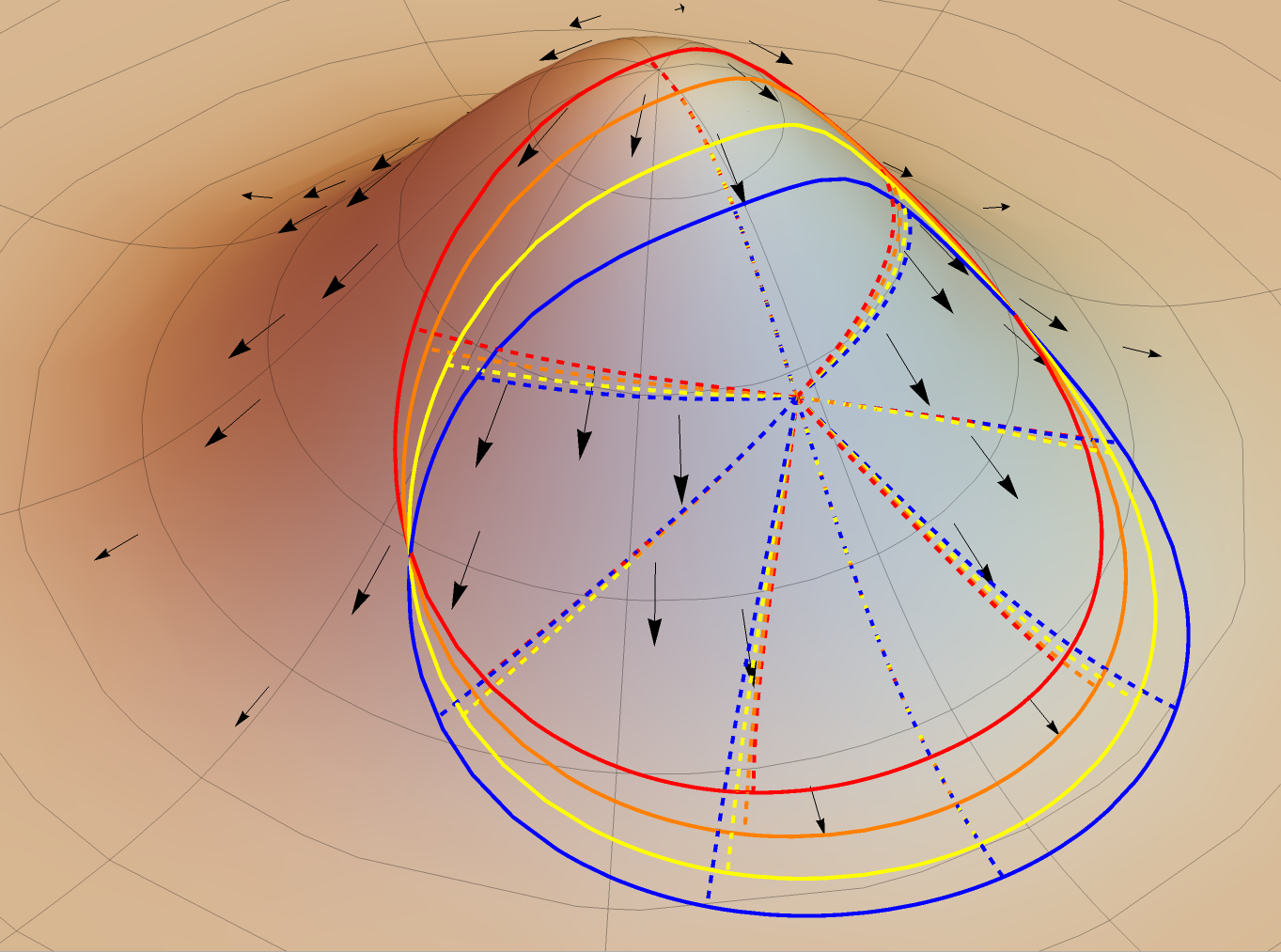}
~\includegraphics[width=0.45\textwidth]{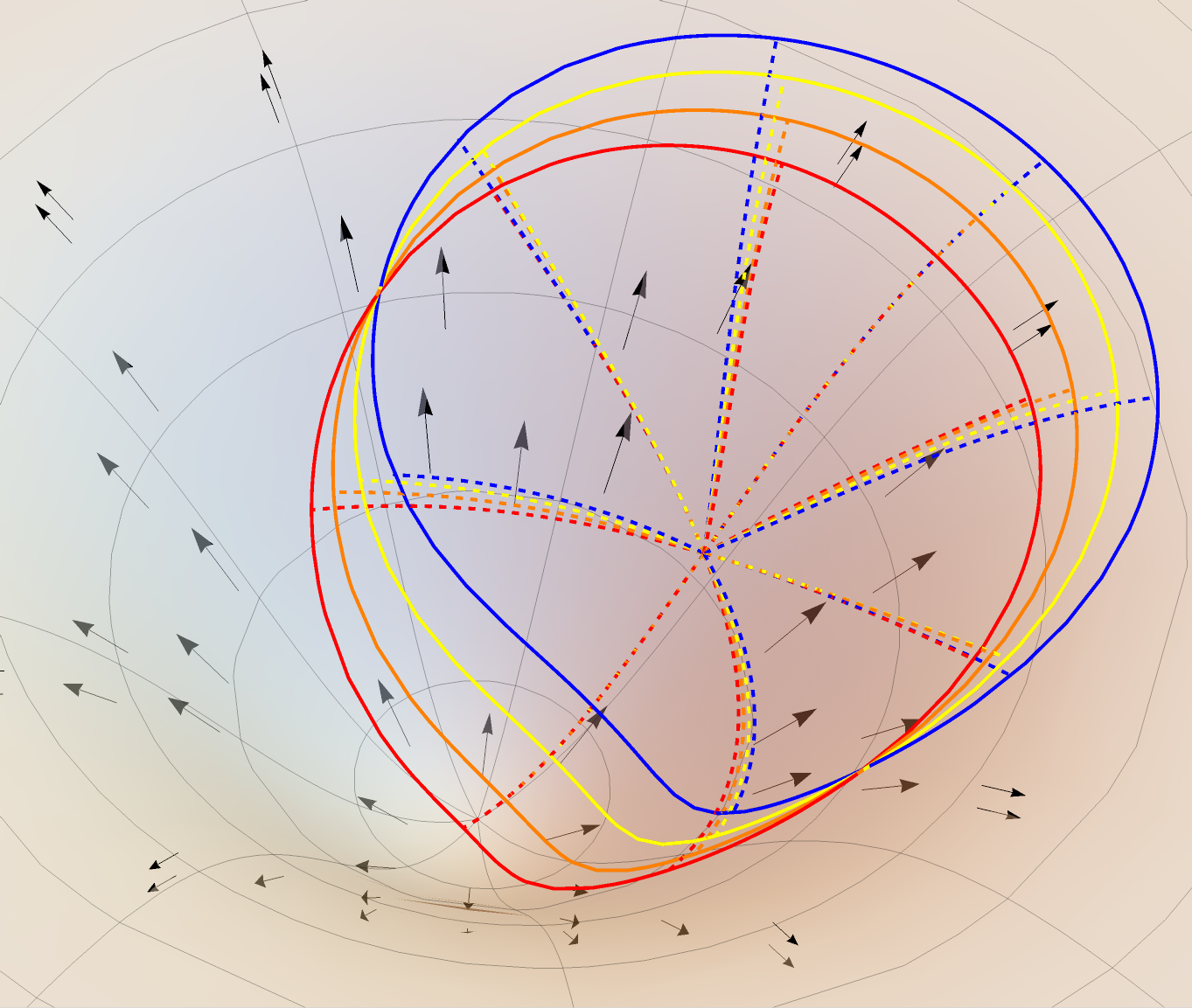}
\caption{Left: the unit time fronts (solid colors) and the related time-minimizing geodesics (dashed colors; drawn with a step of $\Delta\protect\theta=\protect\pi/4$) on the slippery cross slope modeled by the rotational Gaussian bell-shaped surface $\mathfrak{G}$, under the action of gravitational wind, for various along-traction coefficients, i.e. $\tilde{\eta}\in\{1/3 \ \textnormal{(yellow)}, 2/3\  \textnormal{(orange)}, 0 \ \textnormal{(blue, the Zermelo case)}, 1 \ \textnormal{(red, the cross slope case)}\}$; $\bar{g}=0.63$. The gravitational wind $\mathbf{G}^{T}$  (marked by black arrows) ``blows'' in the steepest downhill direction. The initial point is located on the parallel of the strongest gravitational wind, i.e.  $(\protect\rho(0)=1/\protect\sqrt{2}, \protect\varphi(0)= -\protect\pi/4)$. 
Right: as on the left, in inside and reversed view of the surface $\mathfrak{G}$.}
\label{fig_qs_Gauss}
\end{figure}

\noindent Moreover, the new solutions are compared to the Riemannian (white)
and classical Matsumoto (green) geodesics as well as their fronts, under the action of the gravitational wind $\mathbf{G}^{T}$, where $\bar{g}=0.63$ (see \cref{fig_qs_Gauss_2}). 

Finallly, remark that one can apply the current results for the analogous models of
a sinkhole or a valley, where the gravity effect is reverse, namely, with
the uphill action along the surface gradient (see the right-hand side images in \cref{fig_qs_Gauss} and \cref{fig_qs_Gauss_2}).
\begin{figure}[H]
\centering
~\includegraphics[width=0.48\textwidth]{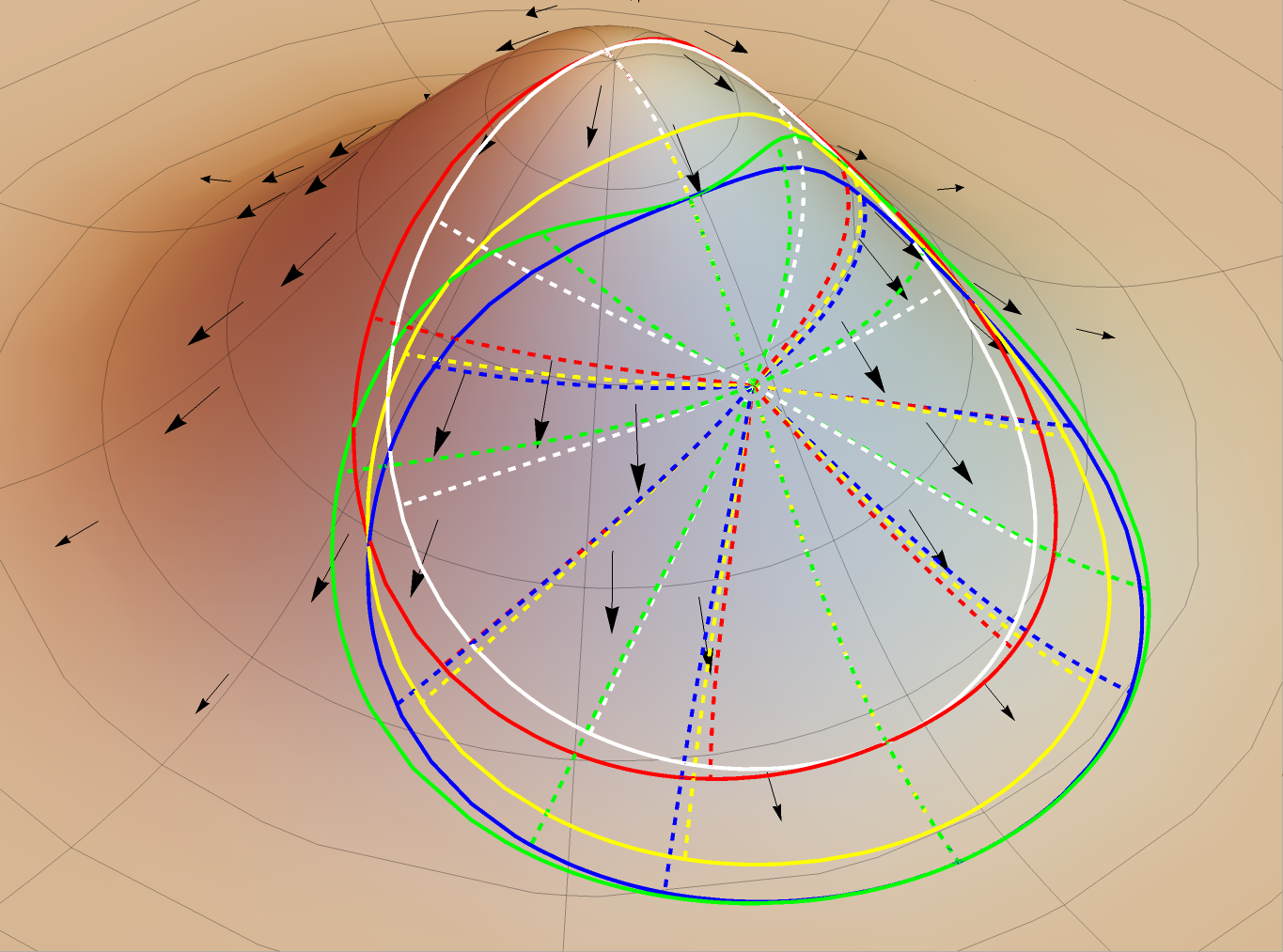}
~\includegraphics[width=0.48\textwidth]{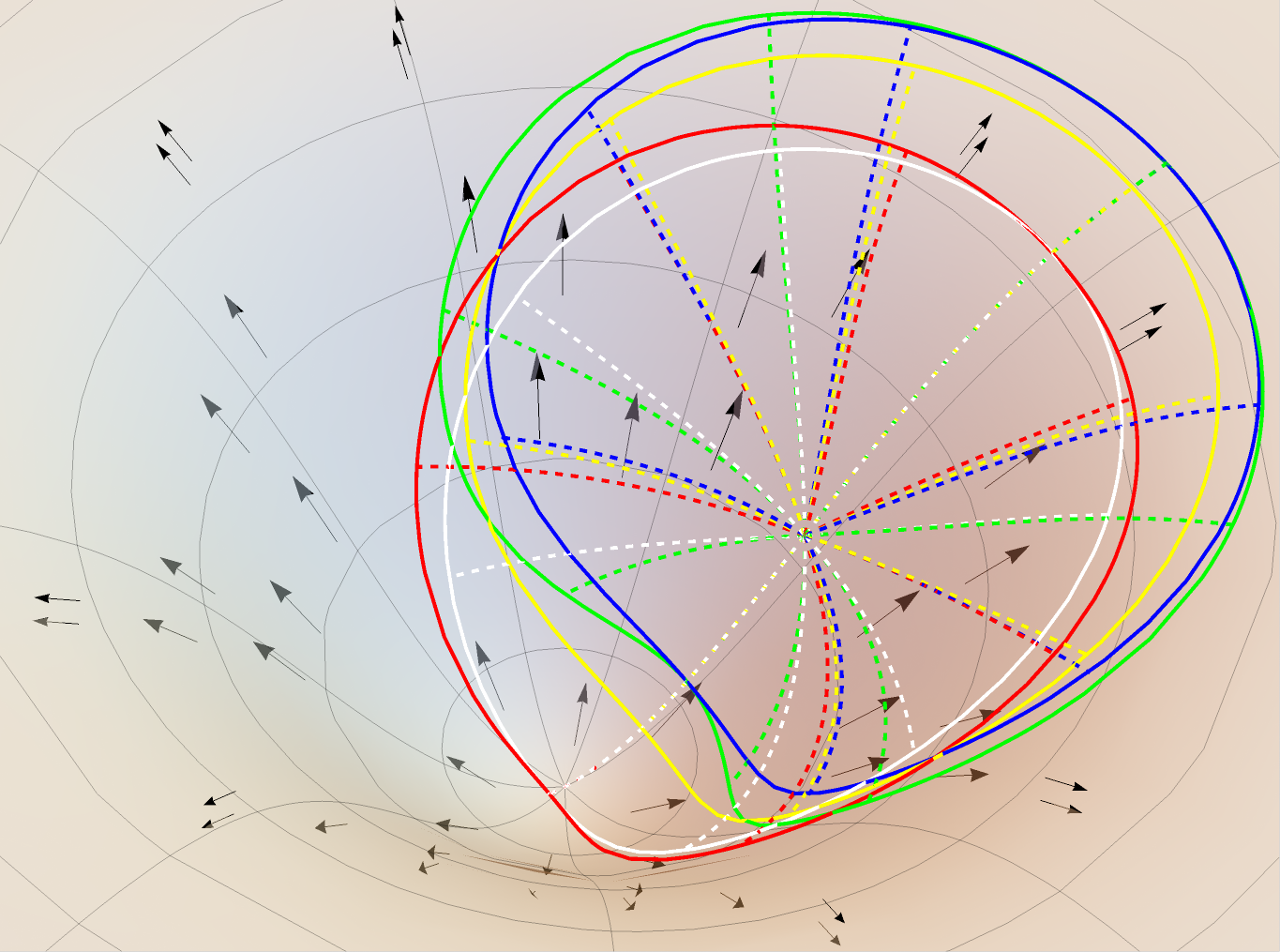}
\caption{Left: the unit time fronts (solid colors) and the related time-minimizing geodesics (dashed colors) on the slippery cross slope as in \cref{fig_qs_Gauss}, where $\tilde{\eta}\in\{1/3 \ \textnormal{(yellow)}, 0 \ \textnormal{(blue, the Zermelo case)}, 1 \ \textnormal{(red, the cross slope case)}$, compared in addition to the Matsumoto (green) and Riemannian (white) cases; $\bar{g}=0.63$. In particular, it can be observed how the initial (unperturbed) Riemannian geodesics and time front are deformed under the action of gravitational wind (marked by black arrows), depending on the along-traction coefficient $\tilde{\eta}$. 
Right: as on the left, in inside and reversed view of the surface $\mathfrak{G}$.}
\label{fig_qs_Gauss_2}
\end{figure}

\smallskip \noindent \textbf{Acknowledgements.} 
The authors are greatly indebted to Reviewers for their time and effort including the constructive and detailed comments and suggestions that helped to improve the initial version of this paper. 
During the final stage of the work the second author was partially supported by the Gdynia Maritime University project reference WN/PI/2024/01.

\bibliographystyle{plain}

\begin{thebibliography}{99}
\bibitem{cross} N. Aldea, P. Kopacz, The slope-of-a-mountain problem in a cross gravitational wind. Nonlinear Anal.-Theor. 235 (2023) 113329. 

\bibitem{Nicprw} N. Aldea, P. Kopacz, R. Wolak, Randers metrics based on
deformations by gradient wind. Period. Math. Hung. 86(3) (2023) 266-280. 

\bibitem{slippery} N. Aldea, P. Kopacz, Time geodesics on a slippery slope under gravitational wind. Nonlinear Anal.-Theor. 227 (2023) 113160. 

\bibitem{baoroblesricci} D. Bao, C. Robles, Ricci and flag curvatures in
Finsler geometry, in: \emph{A sampler of Riemann-Finsler geometry} (eds. D.
Bao et al.), Math. Sci. Res. Inst. Publ. 50, Cambridge Univ. Press, 2004, pp. 197-259. 

\bibitem{BRS} D. Bao, C. Robles, Z. Shen, Zermelo navigation on
Riemannian manifolds. J. Diff. Geom. 66(3) (2004) 377-435. 


\bibitem{brody3} D.C. Brody, G.W. Gibbons, D.M. Meier, A Riemannian approach to Randers geodesics. J. Geom. Phys. 106 (2016) 98-101. 

\bibitem{B-Miron} I. Buc\u{a}taru, R. Miron, Finsler-Lagrange geometry.
Applications to dynamical systems. Editura Academiei Rom\^{a}ne, Bucuresti, 2007. 


\bibitem{CJS} E. Caponio, M.\'{A}. Javaloyes, M. S\'{a}nchez, Wind
Finslerian structures: from Zermelo's navigation to the causality of
spacetimes. arXiv e-print, (2014) 1407.5494.


\bibitem{Sabau_Tai} P. Chansri, P. Chansangiam, S.V. Sab\u{a}u, The
geometry on the slope of a mountain. Miskolc Math. Notes {21}(2) (2020) 747-762. 

\bibitem{Sabau_pedal} P. Chansri, P. Chansangiam, S.V. Sab\u{a}u, Finslerian indicatrices as
algebraic curves and surfaces. Balk. J. Geom. Appl. {25}(1) (2020) 19-33. 

\bibitem{chern_shen} S-.S. Chern, Z. Shen, Riemann-Finsler geometry. Nankai
Tracts in Mathematics. World Scientific, River Edge (N.J.), London,
Singapore, 2005.


\bibitem{Musznay} B. Hubicska, Z. Muzsnay, Holonomy in the quantum navigation
problem. Quantum Inf. Process. 18(10) (2019) 1-10. 

\bibitem{JS} M.\'{A}. Javaloyes, M. S\'{a}nchez, On the definition and examples
of Finsler metrics. Ann. Sc. Norm. Super. Pisa Cl. Sci. (5) 13(3) (2014) 813-858. 

\bibitem{JPS} M.\'{A}. Javaloyes, E. Pend\'{a}s-Recondo, M. S\'{a}nchez, A
general model for wildfire propagation with wind and slope, SIAM J. Appl. Algebra Geom. 7 (2) (2023)  414-439.


\bibitem{Kristaly} S. Kaj\'{a}nt\'{o}, A. Krist\'{a}ly, Unexpected behaviour
of flag and S-curvatures on the interpolated Poincar\'{e} metric, J. Geom.
Anal. 31 (2021) 10246-10262. 

\bibitem{kopi6} P. Kopacz, On generalization of Zermelo navigation problem on Riemannian
  manifolds, Int. J. Geom. Methods Mod. Phys. 16(4) (2019) 19500580. 


\bibitem{markvorsen} S. Markvorsen, A Finsler geodesic spray paradigm for
wildfire spread modelling. Nonlinear Anal. RWA 28 (2016) 208-228. 

\bibitem{matsumoto} M. Matsumoto, A slope of a mountain is a Finsler surface
with respect to a time measure, J. Math. Kyoto Univ. 29 (1989) 17-25. 


\bibitem{cr} C. Robles, Geodesics in Randers spaces of constant curvature. T.
Am. Math. Soc. 359(4) (2007) 1633-1651. 

\bibitem{SH} Z. Shen, Finsler Metrics with $\mathbf{K}=0$ and $\mathbf{S}=0$,
Can. J. Math. 55(1) (2003) 112-132. 

\bibitem{S-Sabau} H. Shimada, S.V. Sab\u{a}u, Introduction to Matsumoto
metric. Nonlinear Anal.-Theor. 63 (2005) 165-168. 

\bibitem{Yajima} T. Yajima, Y. Tazawa, Classification of time-optimal paths under an external force based on Jacobi stability in Finsler space, J. Optim. Theory Appl.  200 (2024) 1216-1238. 

\bibitem{Y-Sabau} R. Yoshikawa, S.V. Sab\u{a}u, Kropina metrics and Zermelo
navigation on Riemannian manifolds, Geom. Dedicata 171(1) (2013) 119-148. 

\bibitem{Yu} C. Yu, H. Zhu, On a new class of Finsler metrics, Diff. Geom.
Appl. 29(2) (2011) 244-254. 

\bibitem{Zer0} E. Zermelo, \"{U}ber die Navigation in der Luft als Problem der Variationsrechnung. Jahresber. Deutsch. Math.-Verein. 89 (1930) 44--48. 

\bibitem{Zer} E. Zermelo, \"{U}ber das Navigationsproblem bei ruhender oder
ver\"{a}nderlicher Windverteilung. ZAMM-Z. Angew. Math. Me. 11(2) (1931) 114--124. 

\bibitem{YBShen} X. Zhang, Y. Shen, On Einstein Matsumoto metrics. Publ.
Math. 85(1-2) (2014) 15-30. 

\end{thebibliography}

\end{document}